\documentclass[11pt, letterpaper]{article}
\usepackage{fullpage}

\usepackage{fullpage}
\usepackage{layout}
\usepackage{multirow}

\usepackage{amsfonts}
\usepackage{amsmath}
\usepackage{amsthm}
\usepackage{amssymb} 

\usepackage{mdframed} 
\usepackage{thmtools}

\definecolor{shadecolor}{gray}{0.95}
\declaretheoremstyle[
headfont=\normalfont\bfseries,
notefont=\mdseries, notebraces={(}{)},
bodyfont=\normalfont,
postheadspace=0.5em,
spaceabove=1pt,
mdframed={
  skipabove=8pt,
  skipbelow=8pt,
  hidealllines=true,
  backgroundcolor={shadecolor},
  innerleftmargin=4pt,
  innerrightmargin=4pt}
]{shaded}

 \usepackage{xcolor}
\usepackage{color}
\usepackage{graphicx}

\usepackage{verbatim}

\newcommand{\R}{\mathbb{R}} 

\newcommand{\cA}{{\cal A}}

\newcommand{\cN}{{\cal N}}
\newcommand{\cO}{{\cal O}}
\newcommand{\cP}{{\cal P}}

\newcommand{\cX}{{\cal X}}

\usepackage[colorinlistoftodos,bordercolor=orange,backgroundcolor=orange!20,linecolor=orange,textsize=scriptsize]{todonotes}

\newcommand{\eqdef}{\overset{\text{def}}{=}} 

\DeclareMathOperator{\rank}{rank}       
\DeclareMathOperator{\diag}{diag}       

\providecommand{\range}[1]{{\rm Range}\left( #1\right)}

\declaretheorem[within=section]{definition}
\declaretheorem[sibling=definition]{theorem}

\declaretheorem[sibling=definition]{lemma}

\theoremstyle{remark}
\newtheorem{example}{Example} 
\newtheorem{remark}{Remark} 
\usepackage[colorlinks=true,linkcolor=blue]{hyperref}

\usepackage[noend]{algpseudocode}
\usepackage{verbatim}
\usepackage[ruled,vlined]{algorithm2e}
\usepackage{wrapfig}

 \usepackage{xcolor}
\usepackage{color}
\usepackage{graphicx}
\usepackage{fullpage}

\usepackage{hyperref}

\usepackage{subcaption}

\usepackage{hyperref}       
\usepackage{url}            
\usepackage{booktabs}       
\usepackage{amsfonts}       
\usepackage{nicefrac}       
\usepackage{microtype}      

\title{A Nonconvex Projection Method for Robust PCA}

\author{Aritra Dutta \thanks{Division of Computer, Electrical and Mathematical Sciences and Engineering (CEMSE), King Abdullah University of Science and Technology (KAUST), Thuwal, Saudi Arabia, 
e-mail: aritra.dutta@kaust.edu.sa.} \qquad 
      Filip Hanzely \thanks{Division of Computer, Electrical and Mathematical Sciences and Engineering (CEMSE), King Abdullah University of Science and Technology (KAUST), Thuwal, Saudi Arabia, e-mail: filip.hanzely@kaust.edu.sa.} \qquad
      Peter Richt\'{a}rik
        \thanks{Division of Computer, Electrical and Mathematical Sciences and Engineering (CEMSE), King Abdullah University of Science and Technology (KAUST), Kingdom of Saudi Arabia, e-mail: peter.richtarik@kaust.edu.sa --- School of Mathematics,  University of Edinburgh, United Kingdom, e-mail: peter.richtarik@ed.ac.uk --- Moscow Institute of Physics and Technology, Dolgoprudny, Russia}}

\date{}

\begin{document}

\maketitle

\begin{abstract}
Robust principal component analysis (RPCA) is a well-studied problem whose goal is to decompose a matrix into the sum of low-rank and sparse components. In this paper, we propose a nonconvex  feasibility reformulation of RPCA problem and apply an alternating projection method to solve it. To the best of our knowledge, this is the first paper proposing a method that solves RPCA problem without considering any objective function, convex relaxation, or surrogate convex constraints. 
We demonstrate through extensive numerical experiments on a variety of applications, including shadow removal, background estimation, face detection, and galaxy evolution, that our approach matches and often significantly outperforms current state-of-the-art in various ways. 
\end{abstract}

\noindent Principal component analysis (PCA) \cite{pca} addresses the problem of best approximation of a matrix $A\in\mathbb{R}^{m\times n}$ by a matrix of rank $r$: \begin{eqnarray}\label{pca}
X^*=\arg\min_{\substack{{X}\in\mathbb{R}^{m\times n}\\{\rm rank}({X})\le r}}\|A-{X}\|_F^2,
\end{eqnarray}
where $\|\cdot\|_F$ denotes the Frobenius norm of matrices. The solution to \eqref{pca} is given by
\begin{align}\label{hardthresholding}
X^*=H_r(A)\eqdef U\Sigma_rV^T,
\end{align}
where $A$ has singular value decompositions 
$
A=U\Sigma V^T,
$
and $\Sigma_r(A)$ is the diagonal matrix obtained from $\Sigma$ by hard-thresholding that keeps the $r$  largest singular values only  and replaces the other singular values by 0. In many real-world problems, if sparse large errors or outliers are present in the data matrix, PCA fails to deal with it. Therefore, it is natural to consider a {\em robust} matrix decomposition model in which we wish to decompose $A$ into the sum of a low-rank matrix $L$ and an error matrix $S$: $A=L+S$. However, without further assumptions, the problem is ill-posed. We assume that the error matrix $S$ is {\em sparse} and that it allows its entries to have arbitrarily large magnitudes. That is, given $A$, we consider the problem of finding a low rank matrix $L$ and a sparse matrix $S$ such that
\begin{eqnarray}\label{mainproblem}
A=L+S.
\end{eqnarray}
In this context, the celebrated principal component pursuit (PCP) formulation of the problem uses the $\ell_0$ norm (cardinality) to address the sparsity constraint and \eqref{mainproblem}. Therefore, PCP is the constrained minimization problem  \cite{candeslimawright,rpca_1}: 
\begin{eqnarray}\label{pcp}
\min_{L,S}{\rm rank}(L)+\lambda\|S\|_{\ell_0} 
\quad \text{subject to} \quad A=L+S,
\end{eqnarray}
where $\lambda>0$ is a balancing parameter. Since both ${\rm rank}(L)$ and $\|S\|_{\ell_0}$ are non-convex, one often replaces the rank function by the (convex) nuclear norm and $\ell_0$ by the (convex)  $\ell_1$ norm. This replacement leads to the immensely  popular  {\em robust principal component analysis (RPCA)} \cite{APG,LinChenMa,candeslimawright}, which can be seen as a convex relaxation of \eqref{pcp}: 
\begin{eqnarray}\label{rpca}
\min_{L,S}\|L\|_*+\lambda\|S\|_{\ell_1} \quad \text{subject to} \quad  A=L+S,
\end{eqnarray}
where $\|\cdot\|_*$ denotes the nuclear norm (sum of the singular values) of matrices. Under some reasonable assumptions on the low-rank and sparse components, \cite{rpca_1,candeslimawright} showed that \eqref{pcp} can be provably solved via \eqref{rpca}. A vast literature is dedicated to solving the RPCA problem, and among them the exact and inexact augmented Lagrangian method of multipliers \cite{LinChenMa}, accelerated proximal gradient method \cite{APG}, alternating direction method \cite{adm_rpca}, alternating projection with intermediate denoising \cite{NIPS2014_5430}, dual approach \cite{dual_rpca}, and SpaRCS \cite{SpaRCS} are a few popular ones. Recently, Yi et al.\ \cite{RPCAgd}, 
Zhang and Yang \cite{zhangpca} proposed  a manifold optimization to solve RPCA. We refer the reader to \cite{rpca_methods} for a comprehensive review of the RPCA algorithms. However, besides formulation \eqref{rpca}, other tractable reformulations of \eqref{pcp} exist as well. For instance, by relaxing the equality constraint in \eqref{pcp} and moving it to the objective as a penalty, together with adding explicit constraints on the target rank $r$ and target sparsity $s$ leads to the following formulation \cite{godec}:
\begin{eqnarray}\label{rpca1}
&\min_{L,S}\|A-L-S\|_{F}^2 \nonumber\\ &\text{subject to} \quad {\rm rank}(L)\le r\;\;{\rm and}\;\;\|S\|_0\le s.
\end{eqnarray}
One can extend the above model to the case of partially observed data that leads to the {\em robust matrix completion (RMC)} problem \cite{Chen_RMC,rmc_taoyuan,CherapanamjeriGJ17}: 
\begin{eqnarray}\label{rmc}
&\min_{L,S}\|\cP_{\Omega}(A-L-S)\|_{F}^2\nonumber\\ &\text{subject to} \quad {\rm rank}(L)\le r\;\;{\rm and}\;\;\|\cP_{\Omega}(S)\|_0\le s',
\end{eqnarray}
where $\Omega\subseteq [m]\times [n]$ is the set of observed data entries, and $\mathcal{P}_{\Omega}$ is the restriction operator defined by \[\left(\mathcal{P}_{\Omega} [X]\right)_{ij} =\begin{cases} X_{ij} & (i,j)\in\Omega\\
0 & \text{otherwise.}
\end{cases}\]  We note that with some modifications, problem \eqref{rpca1} is contained in the larger class of problem presented by \eqref{rmc}. We also refer to some recent work on RMC problem or outlier based PCA in \cite{CherapanamjeriJN17,CherapanamjeriGJ17}. An extended model of \eqref{rpca} can also be referred to as a more general problem as in \cite{rmc_taoyuan} (see problem (1.2) in \cite{rmc_taoyuan}). More specifically, when $\Omega=[m]\times [n]$, that is, the whole matrix is observed, then \eqref{rmc} is \eqref{rpca1}. One can also think of the matrix completion (MC) problem as a special case of \eqref{rmc} \cite{candesplan,Jain,caicandesshen,JN15,candes_MC,KeshavanMC,candes_MC2,Marecek_MC}. For MC problems, $S=0$. Therefore, \eqref{rmc} is a generalization of two fundamental problems: RPCA and RMC. 

\textbf{Contributions.} We solved the RPCA and RMC problems by addressing the original decomposition problem \eqref{mainproblem} directly, without introducing any optimization objective or surrogate constraints. This is a novel approach because we aim to find a point at the intersection of three sets, two of which are non-convex. We formulate both RPCA and RMC as set feasibility problems and propose alternating projection algorithm to solve them. This leads to Algorithm~\ref{rpca_algo} and~3. Our approach is described in the next section. We also propose a convergence analysis of our algorithm. 

Our feasibility approach does not require one to use the hard to interpret  parameters (such as $\lambda$) and surrogate functions (such as the nuclear norm, or $\ell_1$ norm) which makes our approach unique compared to existing models. Instead, we rely on two direct parameters: the target rank $r$ and the desired level of sparsity~$s$. By performing extensive numerical experiments on both synthetic and real datasets, we show that our approach can match or outperform state-of-the-art methods in solving the RPCA and RMC problems. More precisely, when the sparsity level is low, our feasibility approach can viably reconstruct any target low rank, which the RPCA algorithms can not. Moreover, our approach can tolerate denser outliers than can the state-of-the-art RPCA algorithms when the original matrix has a low-rank structure (see details in the experiment section). These attributes make our approach attractive to solve many real-world problems because our performance matches or outperforms that of state-of-the-art RPCA algorithms in solution quality, and do this in comparable or less time.

\section{Nonconvex Feasibility and Alternating Projections~\label{sec:setfeas}}
Set feasibility problem aims to find a point in the intersection of a collection of closed sets, that is:
\begin{eqnarray}\label{convex_feasibiloity}
{\rm Find}\;\;x\in\cX \qquad \text{where} \qquad  \cX\eqdef\cap_{i}^m\cX_i \neq \emptyset,
\end{eqnarray}
for closed sets $\cX_i$. Usually, sets $\cX_i$s are assumed to be simple and easy to project on. A special case of the above setting for convex sets $\cX_i$ is the convex feasibility problem and is already well studied. In particular, a very efficient convex feasibility algorithm is known as the alternating projection algorithm \cite{kaczmarz1937angenaherte,bauschke1996projection}, in which each iteration picks one set $\cX_i$ and projects the current iterate on it. There are two main methods to choosing the sets $\cX_i$ -- traditional cyclic method and randomized method \cite{strohmer2009randomized,gower2015randomized,necoara2018randomized}, and in general, randomized method is faster and not vulnerable to adversarial set order. 

We also note that the alternating projection algorithm for convex feasibility problem does not converge in general to the projection of the starting point onto $\cX$, but rather finds a close-to feasible point in $\cX$, except the case when $\cX_i$s are affine spaces. However, once an exact projection onto $\cX$ is desired, Dykstra's algorithm \cite{boyle1986method} should be applied.

\begin{algorithm}
\SetAlgoLined
	\SetKwInOut{Input}{Input}
	\SetKwInOut{Output}{Output}
	\SetKwInOut{Init}{Initialize}
	
	\nl\Input{$\Pi_i(\cdot)$ -- Projector onto $\cX_i$ for each $i\in \{1,\dots m\}$, starting point $x_0$ }
	
	\nl \For{$k=0,1,\dots$}
	{
	     \nl Choose via some rule $i$	 (e.g., cyclically or randomly)
	     
	     \nl $x_{k+1}=\Pi_i(x_k)$	     
	     
	}	
	\nl \Output{$x_{k+1}$}
	\caption{Alternating projection method for set feasibility}\label{alg:feas}
\end{algorithm}

On the other hand, for general nonconvex sets $\cX_i$,  projection algorithms might not  converge. In some special settings, some forms of convergence (for example, local convergence) can  be guaranteed even without convexity \cite{transversalityMOR2008,LLM2009,HL2013,transversality2015,NonconvexFeasPang2015}.

\subsection{Set feasibility for RPCA\label{sec_rpca_feas}}

In this scope, we define $\alpha$-sparsity as it appears in the last convex constraint $\cX_3$. We do it so that our approach is directly comparable to the approaches from~\cite{RPCAgd,zhangpca}. However, we note that the $\ell_0$-ball constraint can be applied as well. 

\begin{definition}[$\alpha$-sparsity] A matrix $S\in\mathbb{R}^{m\times n}$ is considered to be $\alpha$-sparse if each row and column of $S$ contains at most $\alpha n$ and $\alpha m$ nonzero entries, respectively. That is, the cardinality of the support set of each row and column of the matrix $S$ do not exceed $\alpha n$ and $\alpha m$, respectively. Formally, we write
$$
\|S_{(i,.)}\|_0\le \alpha n\;\;{\rm and}\;\;\|S_{(.,j)}\|_0\le\alpha m\;\;{\rm for\;all}\;i\in[m], j\in[n],
$$
where $i^{{\rm th}}$ row and $j^{{\rm th}}$ column of $S$ are $S_{(i,.)}$ and $S_{(.,j)}$, respectively. 
\end{definition}

\begin{figure*}[h]
    \centering
    \begin{subfigure}{0.24\textwidth}
    \includegraphics[width=\textwidth]{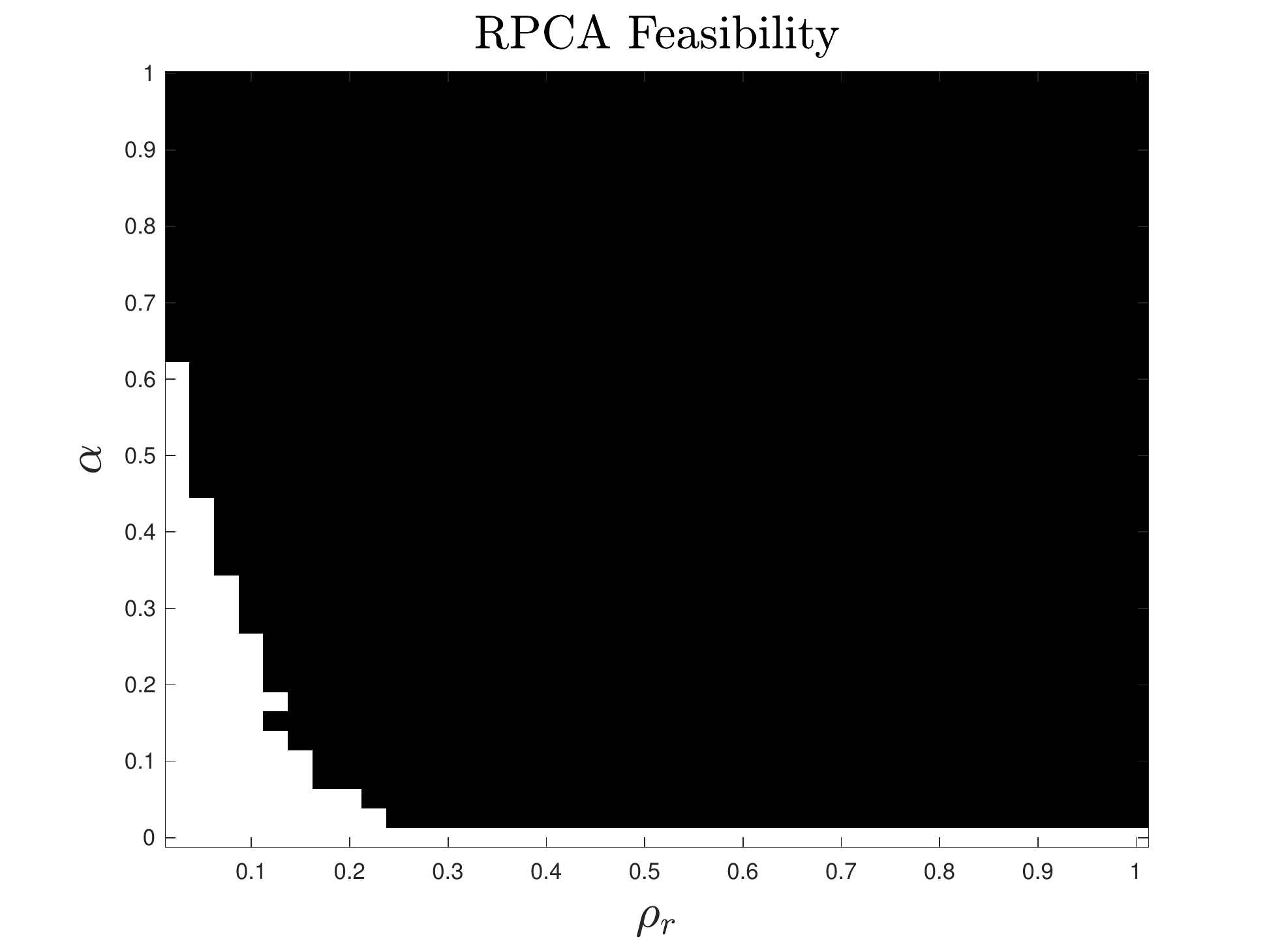}
        \caption{}
      \label{1a}
      \end{subfigure}
    \begin{subfigure}{0.24\textwidth}
    \includegraphics[width = \textwidth]{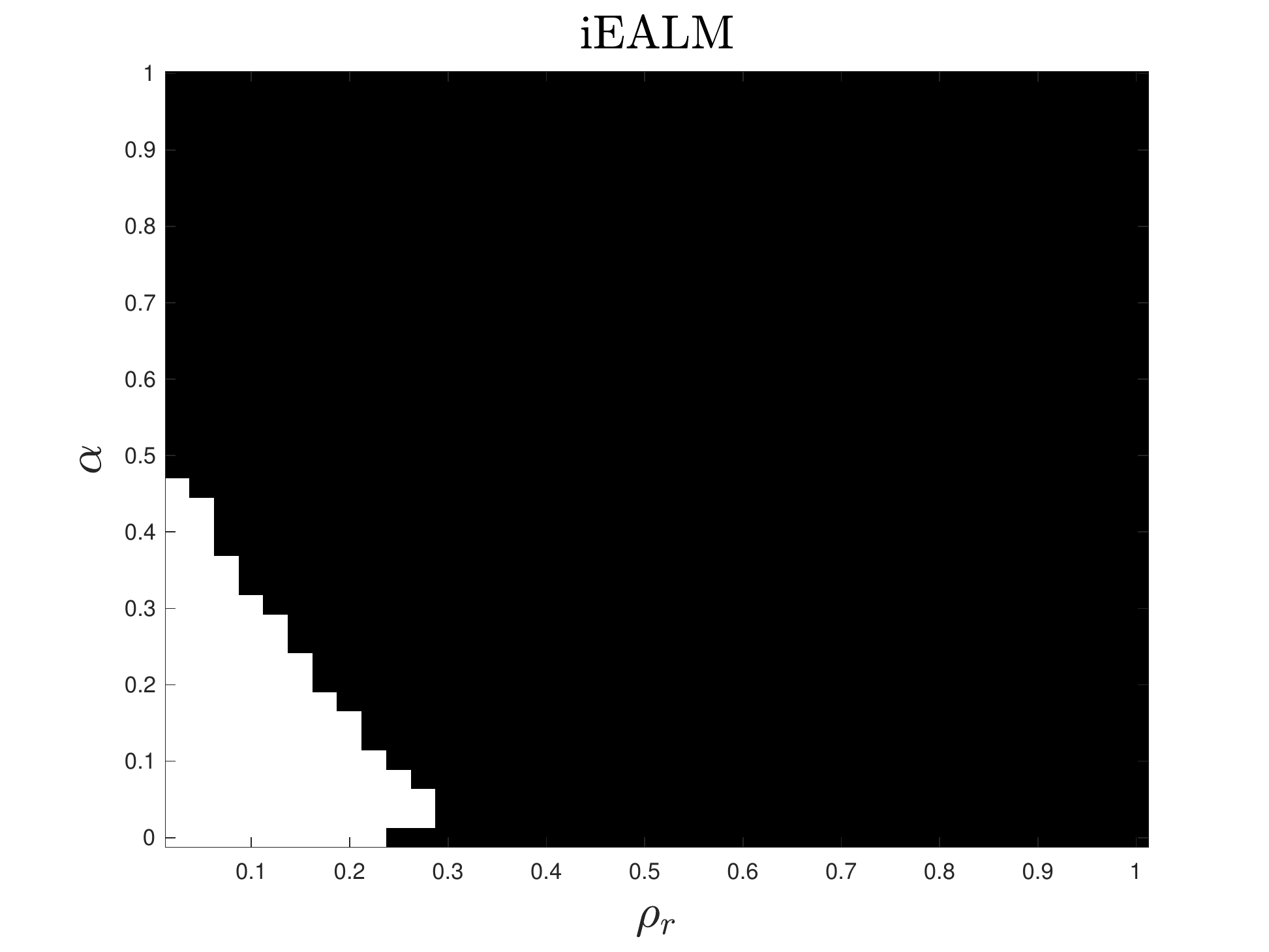}
    \caption{}
    \label{1b}
  \end{subfigure} 
   \begin{subfigure}{0.24\textwidth}
    \includegraphics[width = \textwidth]{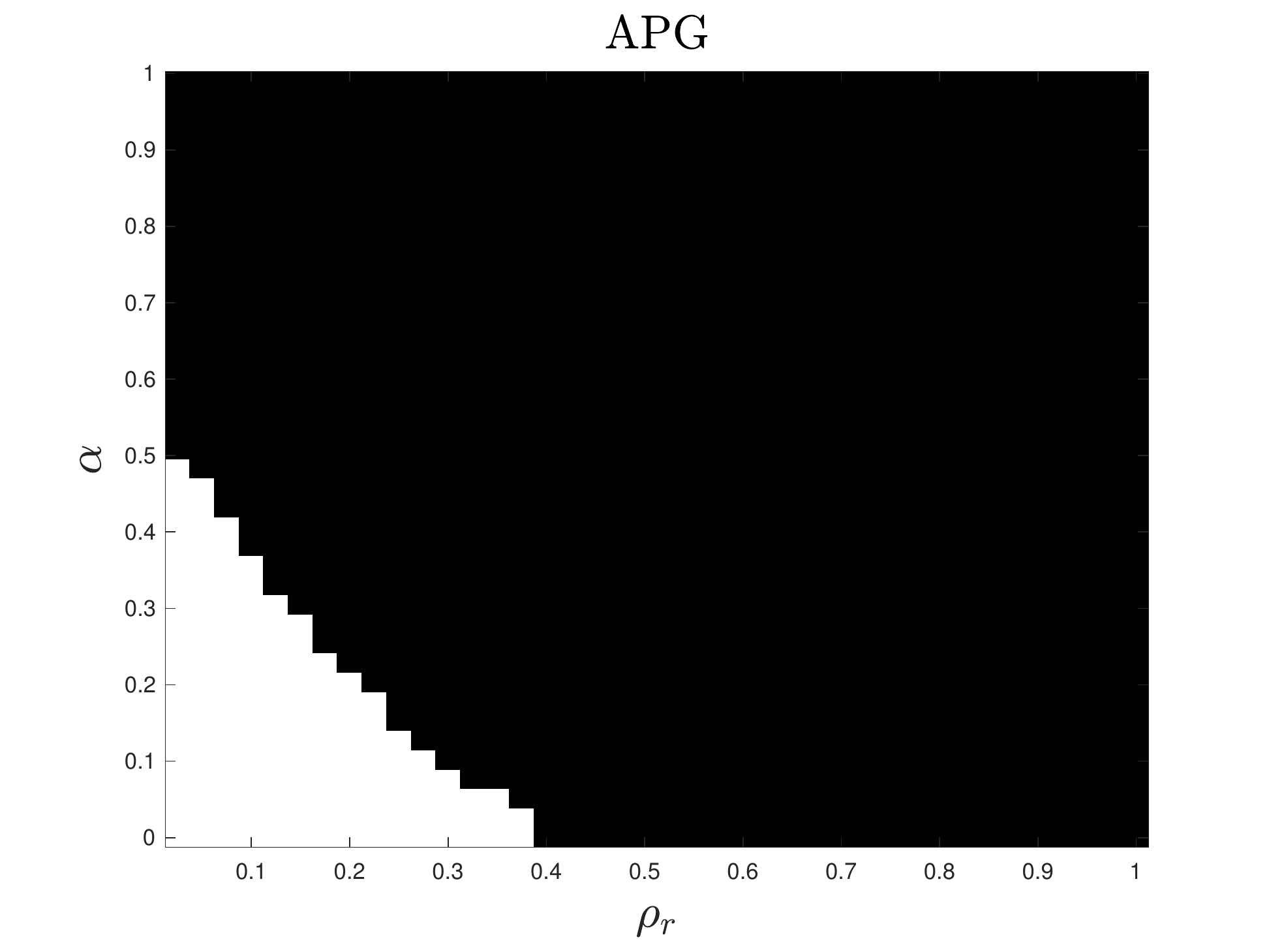}
    \caption{}
    \label{1c}
  \end{subfigure} 
   \begin{subfigure}{0.24\textwidth}
    \includegraphics[width = \textwidth]{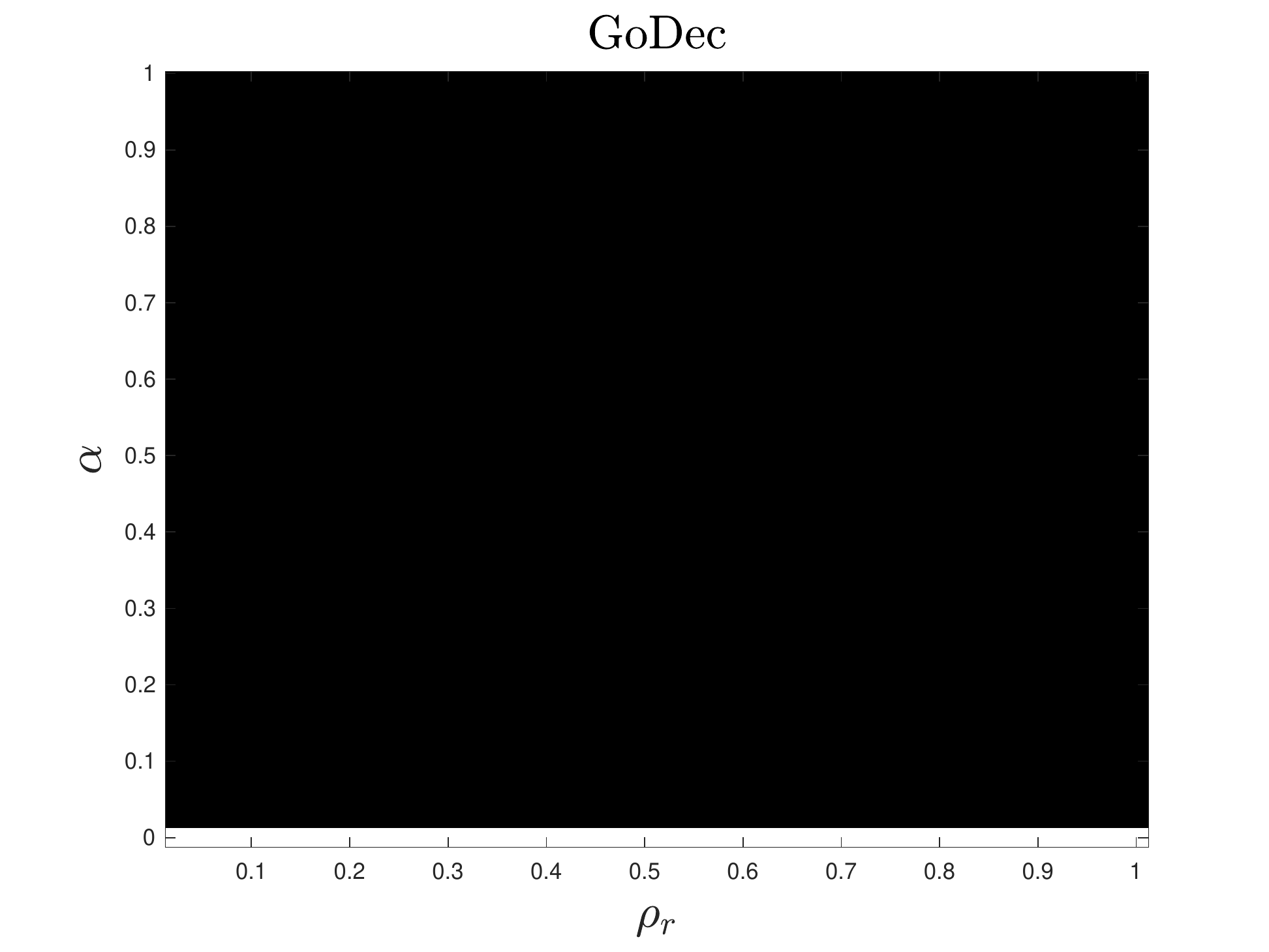}
    \caption{}
    \label{1d}
  \end{subfigure} 
\caption{\small{Phase transition diagram for RPCA F, iEALM, APG, and GoDec with respect to rank and error sparsity. Here, $\rho_r={\rm rank}(L)/m$ and $\alpha$ is the sparsity measure. We have $(\rho_r,\alpha)\in (0.025,1]\times(0,1)$ with $r=5:5:200$ and $\alpha = {\tt linspace}(0,0.99,40)$. We perform 10 runs of each algorithm.}}
    \label{syntheticdata}
\end{figure*}
\begin{figure*}
    \centering
    \begin{subfigure}{0.32\textwidth}
    \includegraphics[width=\textwidth]{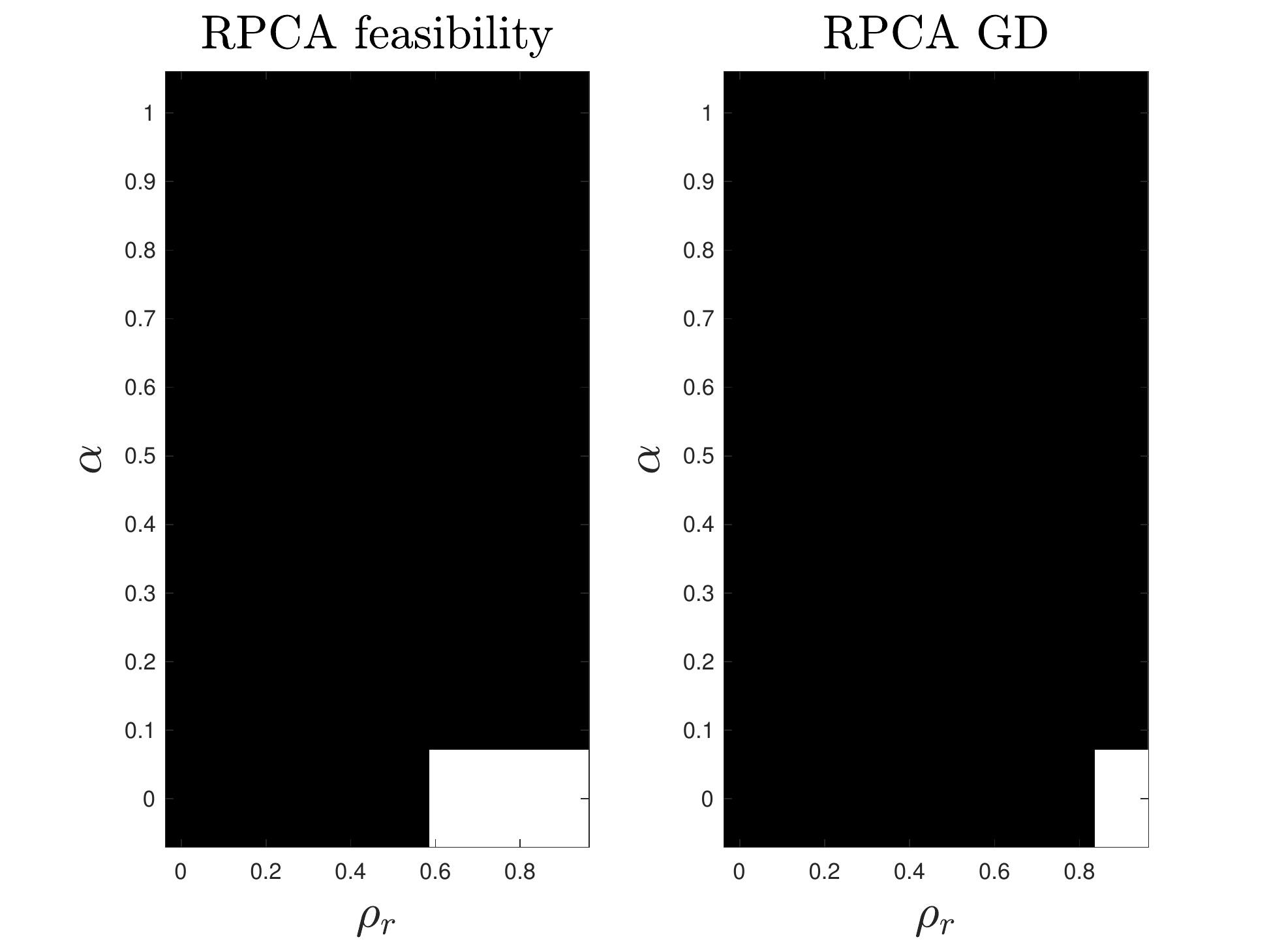}
        \caption{}
      \label{a}
      \end{subfigure}
    \begin{subfigure}{0.32\textwidth}
    \includegraphics[width = \textwidth]{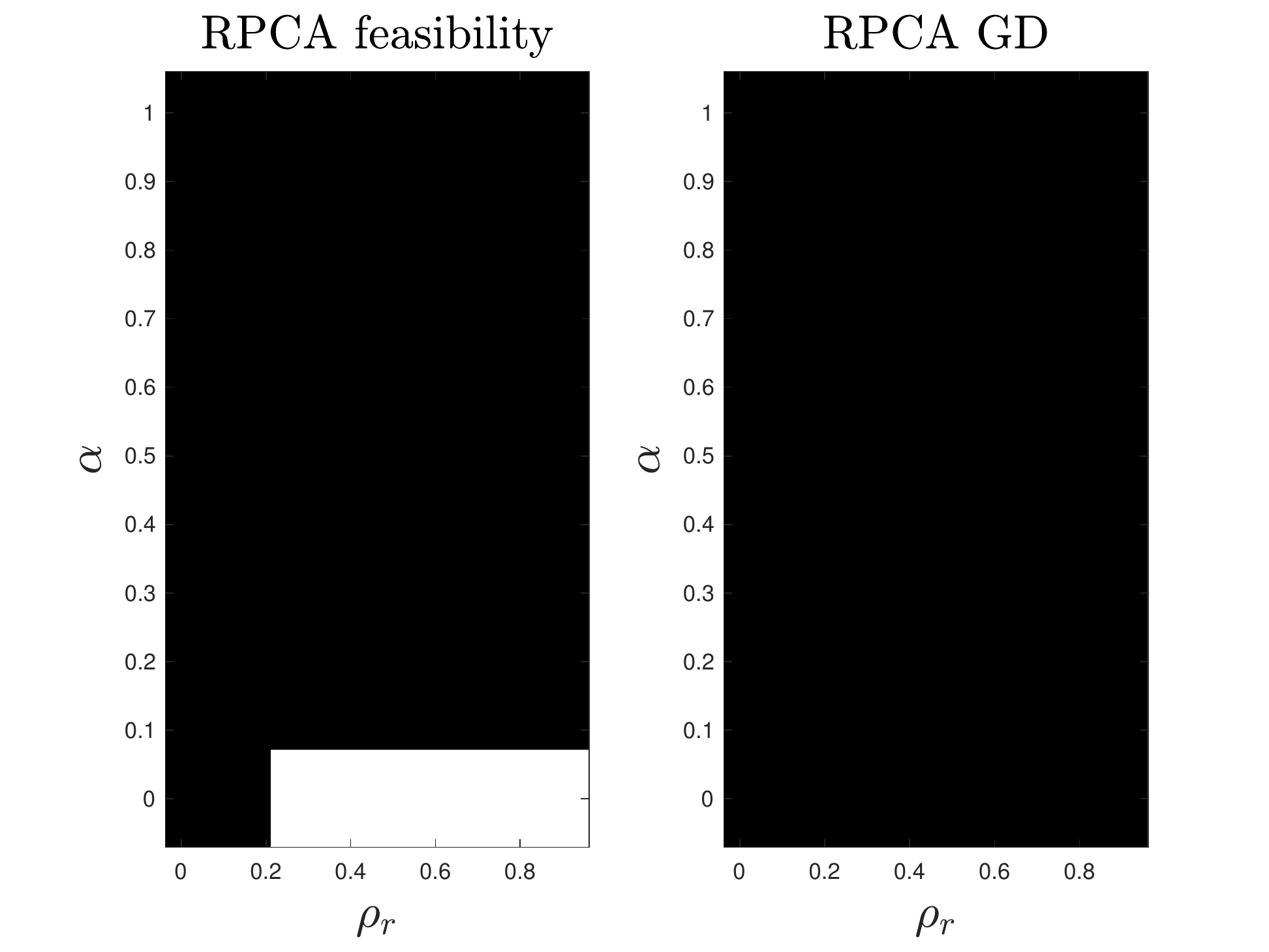}
    \caption{}
    \label{b}
  \end{subfigure} 
   \begin{subfigure}{0.32\textwidth}
    \includegraphics[width = \textwidth]{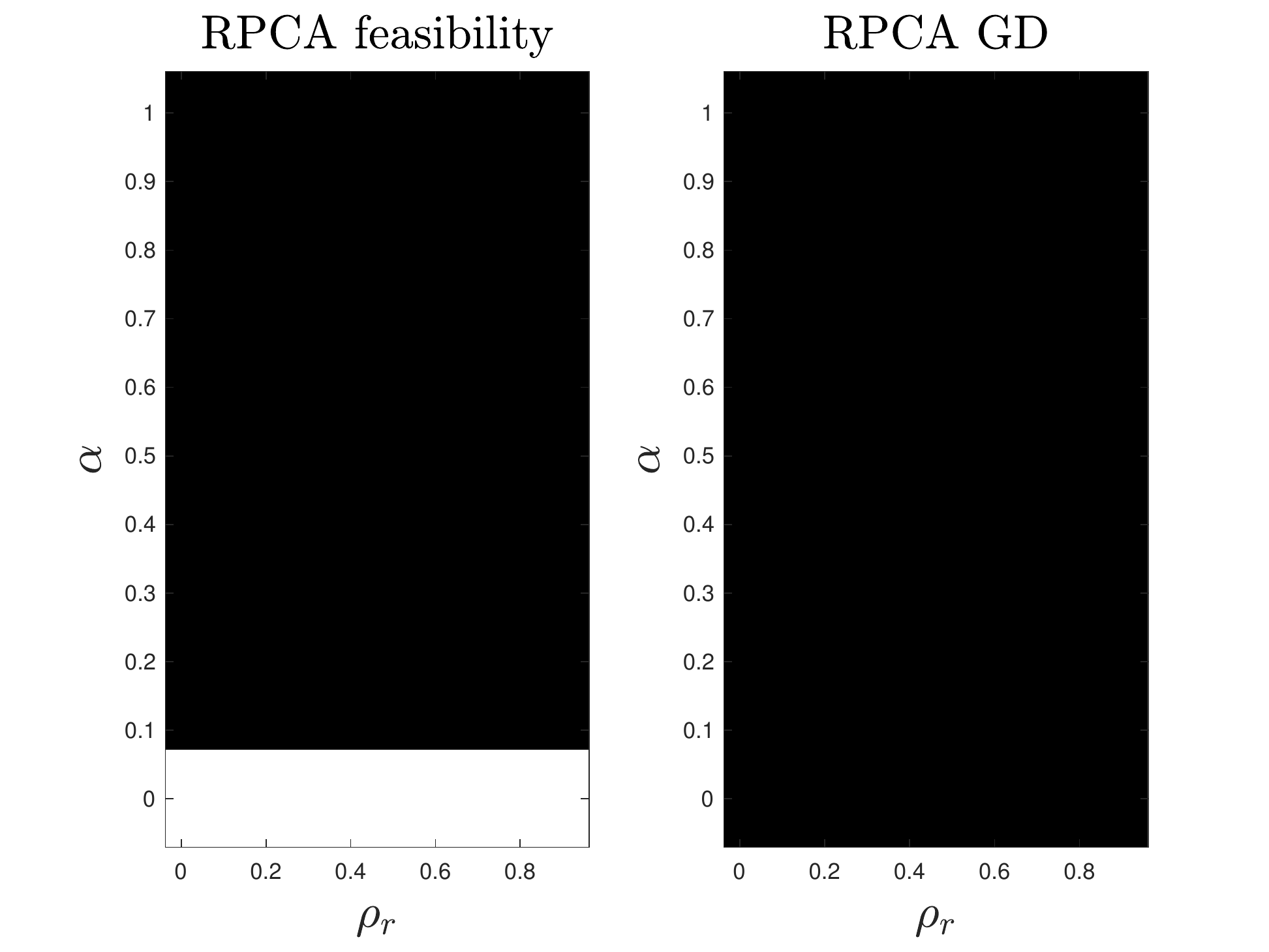}
    \caption{}
    \label{c}
  \end{subfigure} 
 \caption{\small{Phase transition diagram for Relative error for RMC problems: (a) $|\Omega^{C}|=0.5 (m.n)$, (b) $|\Omega^{C}|=0.75 (m.n)$, (c) $|\Omega^{C}|=0.9 (m.n)$. Here, $\rho_r={\rm rank}(L)/m$ and $\alpha$ is the sparsity measure. We have $(\rho_r,\alpha)\in (0.025,1]\times(0,1)$ with $r=5:25:200$ and $\alpha = {\tt linspace}(0,0.99,8)$.}}
    \label{syntheticdata_rel_error}
\end{figure*}

Now we consider the following reformulation of RPCA:
\begin{equation}\label{eq:feas_prob}
\text{Find}\quad M\eqdef [L,S] \in \cX \eqdef \cap_{i=1}^3 \cX_i \neq \emptyset,
\end{equation}
where 
\begin{eqnarray}
\cX_1 &\eqdef& \{M\,|\,L+S=A\} \label{eq:X1_defF}\\
\cX_2 &\eqdef& \{M\,|\,\rank(L)\leq r\} \nonumber\\
\cX_3 &\eqdef& \{M\,|\,\|S_{(i,.)}\|_0\le \alpha n\;\;{\rm and}\;\;\|S_{(.,j)}\|_0\le\alpha m\;\;\nonumber\\&&{\rm for\;all}\;i\in[m], j\in[n].\}  \nonumber
\end{eqnarray}

Clearly, $\cX_1$ is convex, but $\cX_2$ and $\cX_3$ are not. Nevertheless, the algorithm we propose -- alternating Frobenius norm projection on $\cX_i$ performs well to solve RPCA in practice. To validate the robustness of our algorithm, we compare our method to other state-of-the-art RPCA approaches on various practical problems. We also study the local convergence properties and show that despite the non-convex nature of the problem, the algorithms we propose often behave surprisingly well.

\section{The Algorithm}

Denote $\Pi_i$ to be projector onto $\cX_i$. Note that $\Pi_2$ does not include $S$ and projection onto $\Pi_3$ does not include $L$. Consequently, $\Pi_2 \Pi_3$ is a projector onto $\cX_2\cap \cX_3$. Because our goal is to find a point at the intersection of two sets, we shall employ a cyclic projection method (note that randomized method does not make sense). Indeed, steps~\ref{nl_tilde_L} and~\ref{nl_tilde_S} of Algorithm~\ref{rpca_algo} perform projection onto $\cX_1$, step~\ref{nl_Hr} performs projection onto $\cX_2$, and finally, step~\ref{nl_Ta} performs projection onto $\cX_3$. Later in this section we describe the exact implementation and prove correctness the steps mentioned above. 
\begin{algorithm}
\SetAlgoLined
	\SetKwInOut{Input}{Input}
	\SetKwInOut{Output}{Output}
	\SetKwInOut{Init}{Initialize}
	\nl\Input{$A\in\mathbb{R}^{m\times n}$ (the given matrix), rank $r$, sparsity level $\alpha\in(0,1]$}
	\nl\Init {$L_0, S_0$}
	\nl \For{$k=0,1,\dots$}
	{
	     \nl $\tilde{L} = \frac{1}{2}(L_k-S_k+A)$ \label{nl_tilde_L}
	     
		\nl $\tilde{S} = \frac{1}{2}(S_k-L_k+A)$ \label{nl_tilde_S}

		\nl $L_{k+1} = H_r(\tilde{L})$\label{nl_Hr}
		
		\nl $S_{k+1} = \mathcal{T}_{\alpha}(\tilde{S})$\label{nl_Ta}
	}
	\nl \Output{$L_{k+1}, S_{k+1}$}
	\caption{Alternating projection method for RPCA}\label{rpca_algo}
\end{algorithm}

Next, we propose an algorithm to solve the RMC problem \eqref{rmc}. Note that we use the generic hard thresholding operator as in \eqref{hardthresholding} in step~\ref{svd_rmc}  of Algorithm \ref{rmc_algo}. In practice, however, one can perform  cheap SVD. 
\begin{algorithm}
\SetAlgoLined
	\SetKwInOut{Input}{Input}
	\SetKwInOut{Output}{Output}
	\SetKwInOut{Init}{Initialize}
	\nl\Input{$A\in\mathbb{R}^{m\times n}$ (the given matrix), rank $r$, sparsity level $\alpha\in(0,1]$}
	\nl\Init {$L_0, S_0$}
	\nl \For{$k=0,1,\dots$}
	{
	     \nl $\tilde{L} = \frac{1}{2}\cP_{\Omega}(L_k-S_k+A)$ \label{nl_tilde_L2}
	     
	     \nl $\tilde{S} = \frac{1}{2}\cP_{\Omega}(S_k-L_k+A)$ \label{nl_tilde_S2} 
	        
		\nl $L_{k+1} = H_r(\tilde{L}+\cP_{\Omega^c}(L_k))$\label{svd_rmc}
		
		\nl $S_{k+1} = \mathcal{T}_{\alpha}(\tilde{S})$\label{nl_Ta}
	}
	\nl \Output{$L_{k+1}, S_{k+1}$}
	\caption{Alternating projection method for RMC}\label{rmc_algo}
\end{algorithm}

Finally, we provide a local convergence of Algorithm \ref{rpca_algo}, which depends on the local geometry of the optimal point, and is mostly linear, which we prove later. 

\paragraph{Projection on the linear constraint\label{sec:lc}.}
The next lemma provides an explicit formula for the projection onto $\cX_1$, which corresponds  to steps~\ref{nl_tilde_L} and~\ref{nl_tilde_S} of Algorithm~\ref{rpca_algo}.

\begin{lemma}\label{sp}
Solutions to
\[
\min_{L,S} \| L-L_0\|^2_{F}+ \| S-S_0\|^2_{F} \quad \text{subject to} \quad L+S=A
\] is $L^*=\frac12 (L_0-S_0+A)$ and $S^*=\frac12 (S_0-L_0+A)$.
\end{lemma}

We also provide an analogy to Lemma~\ref{sp} for the  RMC problem (steps~\ref{nl_tilde_L2} and~\ref{nl_tilde_S2} of Algorithm \ref{rmc_algo}). 

\begin{lemma}\label{lem:rmc_linproj}
Solutions to
\[
\min_{L,S} \|\cP_{\Omega}(L-L_0)\|_F^2+ \|\cP_{\Omega}(S-S_0)\|_F^2 \quad \text{subject to} \quad \cP_{\Omega}(L+S)=\cP_{\Omega}(A)
\]
are $L^*=\frac12 \cP_{\Omega}(L_0-S_0+A)$ and $S^*=\frac12 \cP_{\Omega}(S_0-L_0+A)$.
\end{lemma}

\paragraph{Projection on the low rank constraint\label{sec:lr}.}
Consider $L^{(r)}$ to be the projection of $L$ onto the rank $r$ constraint, that is,
\[
L^{(r)}=\arg\min_{L'} \|L'-L\|_F \quad \text{subject to} \quad  {\rm rank}(L')\leq r.
\]

It is known that $L^{(r)}$ can be computed as $r$-SVD of $L$. Fast $r$-SVD solvers has improved greatly in recent years \cite{halko2011finding,musco2015randomized,shamir2015stochastic,allen2016lazysvd}. 
Unfortunately, the most recent approaches \cite{shamir2015stochastic,allen2016lazysvd} were not applied in our setting because they are inefficient; they need to compute $LL^\top$ (or $L^\top L$), which is expensive. Instead, we use block Krylov approach from \cite{musco2015randomized}. For completeness, we quote the algorithm in Appendix. 
Regarding the computational complexity, it was shown that block Krylov SVD outputs $Z$ satisfying $\| L-ZZ^\top  L\|_F \leq (1+\tilde{\epsilon})\| L-L^{(r)}\|_F$ in 
\begin{align*} 
\cO\left(\|L \|_0 \frac{r\log n}{\sqrt{\tilde{\epsilon}}}+\frac{mr^2\log^2 n}{\tilde{\epsilon}} + \frac{r^3\log^3 n}{\tilde{\epsilon}^{3/2}} \right )
\end{align*}
flops. Therefore, projection on the low-rank constraint is not an issue for relatively small rank $r$.

\paragraph{Projection on sparsity constraint\label{sec:spars}.}

Projection onto $\cX_3$ simply keeps the $\alpha$-fraction of the largest elements in absolute value in each row and column and set the rest to zero. One can use a global hard-thresholding operator that considers $\ell_0$ constraint on the entire matrix. Instead, we proposed an operator $\mathcal{T}_{\alpha}(\cdot)$. Indeed, $\mathcal{T}_{\alpha}(\cdot)$ does not perform an explicit Euclidean projection onto $\cX_3$. Instead, it performs a projection onto a certain subset of $\cX_3$ and this is clear from the definition (11) (the subset is defined through support $\Omega_\alpha$). Formally, we define: 
\begin{eqnarray}\label{eq:T_def}
\mathcal{T}_{\alpha}[S] &\eqdef&\mathcal{P}_{\Omega_{\alpha}}(S)\in\mathbb{R}^{m\times n}: (i,j)\in\Omega_{\alpha}{\rm if}\nonumber\\
&&\;|S_{ij}|\ge|S_{(i,.)}^{(\alpha n)}|\;\;{\rm and}\; |S_{ij}|\ge|S_{(.,j)}^{(\alpha m)}|,\qquad
 \end{eqnarray}
 where $S_{(i,.)}^{(\alpha n)}$ and $S_{(.,j)}^{(\alpha m)}$ denote the $\alpha$ fraction of largest entries of $S$ along the $i^{{\rm th}}$ row and $j^{{\rm th}}$ column, respectively. This allows us to inexpensively compute an approximate projection onto $\cX_3$ which works well in practice. Remarkably, this does not affect our theoretical results (which are formulated for exact projection onto $\cX_3$) in any way.
We note that the operator $\mathcal{T}_{\alpha}(\cdot)$ is similar to that defined in \cite{RPCAgd,zhangpca}. Projection on sparsity constraint~\eqref{eq:T_def} can be implemented in $\cO(nd)$ time: for each row and each column we find $\alpha n$-th largest element (or $\alpha d$) and simultaneously (for other rows/columns) mask the rest. In our experiments, we use fast implementation of $n$-th element computation from \cite{nth_element}. 
\begin{figure*}
    \centering
    \includegraphics[width = \textwidth]{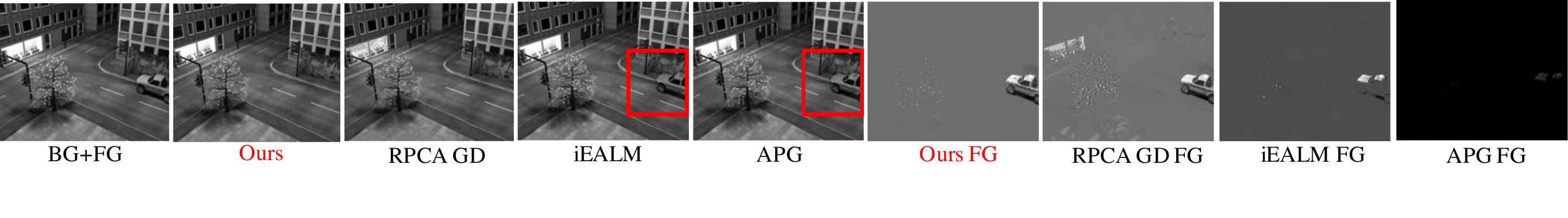}
    \caption{\small{Background and foreground separation on Stuttgart dataset {\tt Basic} video. Except RPCA GD and our method, all other methods fail to remove the static foreground object.}}
    \label{Qual_basic}
\end{figure*}
\begin{figure*}
    \centering
    \includegraphics[width = \textwidth]{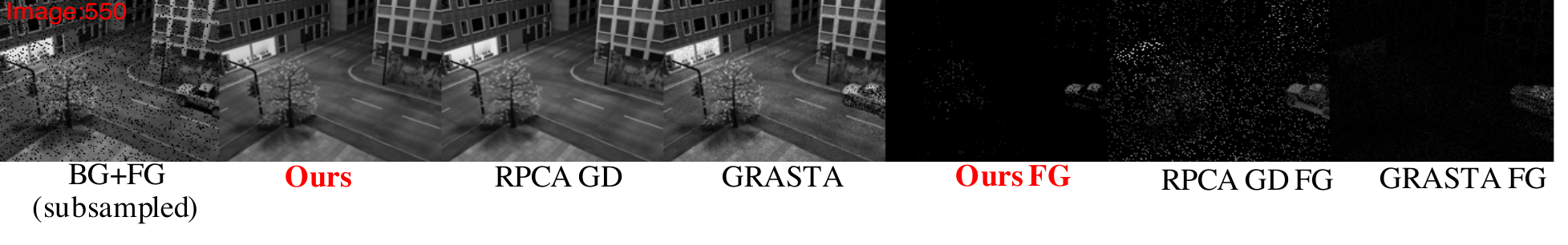}
    \caption{\small{Background and foreground separation on Stuttgart dataset {\tt Basic} video. We used 90\% sample. GRASTA forms a fragmentary background and exhausts around 540 frames to form a stable video. We also note that RPCA GD has more false positives in the foreground.}}
    \label{Qual_basic2}
\end{figure*}

\begin{remark}
One may use the Douglas-Rachford operator splitting method~\cite{artacho2016global} as an alternative to the nonconvex projections. We leave this for future research. 
\end{remark}

\section{Convergence Analysis\label{sec:convergecne}}
In this section, we establish a local convergence analysis of our algorithm by using the basic properties of the alternating projection algorithm~\cite{transversalityMOR2008}. A similar analysis was done for GoDec~\cite{godec}, although Algorithm~\ref{rpca_algo} is vastly different compare to GoDec (we report detailed comparison of these algorithms in Appendix).

Recall that Algorithm~\ref{alg:feas} performs an alternating projection of $[L,S]$ on the sets $\cX_1$ and $\cX_{\cap}\eqdef\cX_2\cap\cX_3$ defined in \eqref{eq:X1_defF}. Before stating the convergence theorem, let us define a (local) angle between sets.

\begin{definition} \label{angle_defn}
Let a point $p$ be in the intersection of set boundaries $\partial K$ and $\partial L$. Define $c(K,L,p)$ be the cosine of an angle between sets $K$ and $L$ at point a $p$ as:
\begin{eqnarray*}
c(K, L, p) \eqdef \cos \angle (\partial K^\top|_p, \partial L^\top|_p),
\end{eqnarray*}
where $\partial K^\top|_p$ denotes a tangent space of set boundary $\partial K$ at $p$ and  $\angle$ returns the angle between subspaces given as arguments. As a consequence, we have $0\le c(K, L, p) \le 1.$
\end{definition}
Let us also define $d_{\cX_1 \cap \cX_{\cap}}(x)$ to be an Euclidean distance of a point $x$ to the set $\cX_1 \cap \cX_{\cap}$.

\begin{theorem}\label{conv_l}\cite{transversalityMOR2008} Suppose that $[\bar{L}\;\;\bar{S}]\in \partial (\cX_1 \cap \cX_{\cap})$. Given any constant $c\in \R$ such that 
$c>c(\cX_1 , \cX_{\cap}, [\bar{L}\;\;\bar{S}])$ there is a starting point $[L_0\;\;S_0]$ close to $[\bar{L}\;\;\bar{S}]$ such that the iterates $L_k, S_k$ of Algorithm~\ref{rpca_algo} satisfy
\[
d_{\cX_1 \cap \cX_{\cap}}([L_k\;\;S_k]) < c^k d_{\cX_1 \cap \cX_{\cap}}([L_0\;\;S_0]).
\]
\end{theorem}

\begin{remark}\label{no_conv}
From Theorem \ref{conv_l} it is clear that the smaller $c(\cX_1 , \cX_{\cap}, [\bar{L}\;\;\bar{S}])$ produces a faster convergence, while $c(\cX_1 , \cX_{\cap}, [\bar{L}\;\;\bar{S}])=1$ can stop the convergence, as described in Example~2 in Appendix. 
\end{remark}

\begin{remark}\label{RMC}
Theorem~\ref{conv_l} is stated for Algorithm~\ref{rpca_algo}, however, one can easily obtain an equivalent result for Algorithm~3 as well. 
\end{remark}

\begin{remark}
Considering the nuclear norm relaxation instead of low rank constraint and $\ell_1$ norm relaxation instead of sparsity constraint, the set $\cX_{\cap}$ becomes convex, and thus the whole problem becomes convex as well. Therefore, Algorithms~\ref{rpca_algo} and~3 converge globally. 
\end{remark}

For completeness, we also derive the exact form of tangent spaces of $\cX_1, \cX_\cap$ mentioned in Definition~\ref{angle_defn}. 
Suppose that $\rank(\bar{L})=r$, and $\bar{S}$ is a matrix of maximal sparsity, that is, $\bar{S} \in \cX_3$ while $\bar{S}+S' \not\in \cX_3$ s.t. $\|S'\|=1$ and $\|\bar{S}+S' \|_0=\|\bar{S}\|_0+1$.
The tangent spaces of $\partial \cX_1$ and $\partial \cX_{\cap}$ at point $[\bar{L}\;\;\bar{S}]$ are given by
\begin{eqnarray*}\label{normal_manifold1}
\partial \cX_1^\top|_{[\bar{L}\;\;\bar{S}]} &=& \cX_1 \\
\partial \cX_\cap^\top|_{[\bar{L}\;\;\bar{S}]} &=& \partial \cX_2^\top|_{\bar{L}} \times \partial \cX_3^\top|_{\bar{S}}, \\
\end{eqnarray*}
where
\begin{eqnarray*}
\partial \cX_2^\top|_{\bar{L}} &=& \{\tilde{L}|\, \tilde{L} =\bar{L}
+ \tilde{U}\tilde{\Sigma} \tilde{V}^\top, \tilde{U}^\top \tilde{U} = \tilde{V}^\top \tilde{V} = I, \\ &&\bar{U}^\top \tilde{U}=0, \bar{V}^\top \tilde{V}=0,  \tilde{\Sigma} =\diag{\tilde{\Sigma}}\} \\
\partial \cX_3^\top|_{\bar{S}} &=& \{
\tilde{S}|\, \tilde{S} = \bar{S}+S', S'_{i,j} = 0 \, \, \forall (i,j)\,\, \mbox{s.t.} \,\, \bar{S}_{i,j}=0
\}.
\end{eqnarray*}

Later in Appendix, we also empirically show that:
\begin{enumerate}
\item Convergence speed is not significantly influenced by starting point.
\item Convergence is usually fastest for small true sparsity level $\alpha$ and small true rank $r$, which is the situation in many practical applications. 
\item Convergence of Algorithm~3 is slower for medium sized number of observable entries, that is, when $|\Omega|\approx0.5(m\cdot n)$, and faster for smaller and bigger sizes.
\item If sparsity and rank levels ($\alpha$ and $r$) are set to be smaller than their true values at the optimum incorrectly, Algorithm~\ref{rpca_algo} does not converge (as in this case, $\cap \cX_i$ might not exist). Moreover, the  performance of the algorithm is  sensitive to the choice of $r$, and this is particularly so if we underestimate the true value (see Figure~8 in Appendix). 
\end{enumerate}

Finally, in Appendix we give two examples for the convex version of the problem~\eqref{eq:feas_prob} with the same block structure; in them, the alternating projection algorithm either converges extremely fast or does not even converge linearly. 
\section{Numerical experiments}\label{sec:experiments}
To explore the strengths and flexibility of our feasibility approach, we performed numerical experiments. First, we work with synthetic data and subsequently apply our method to four real-world problems. 

\subsection{Results on synthetic data}\label{sec:synthetic}
To perform our numerical simulations, first, we construct the test matrix $A$. We follow the seminal work of Wright et al.~\cite{APG} to design our experiment. To this end, we construct $A$ as a low-rank matrix, $L$, corrupted by sparse large noise, $S$, with arbitrary large entries such that $A = L+S$.  We generate $L$ as a product of two independent full-rank matrices of size $m\times r$ whose elements are independent and identically distributed (i.i.d.) $\cN(0, 1)$ random variables and ${\rm rank}(L) = r$. We generate $S$ as a noise matrix whose elements are sparsely supported by using the operator \eqref{eq:T_def} and lie in the range $[-500,500]$. We fix $m=200$ and define $\rho_r={\rm rank}(L)/m$ where ${\rm rank}(L)$ varies. We choose the sparsity level $\alpha\in(0,1)$. 
For each pair of $(\rho_r,\alpha)$ we apply iEALM, APG, and our algorithm to recover the pair $(\hat{L},\hat{S})$ such that $\hat{A}=\hat{L}+\hat{S}$ be the recovered matrix. For both APG and iEALM, we set $\lambda=1/\sqrt{m}$ and for iEALM we use $\mu=1.25/\|A\|_2$ and $\rho=1.5$, where $\|A\|_2$ is the spectral norm (maximum singular value) of $A$. If the recovered matrix pair $(\hat{L},\hat{S})$ satisfies the relative error $\frac{\|L-\hat{L}\|_F+\|S-\hat{S}\|_F}{\|{A}\|_F}<0.01$ then we consider the construction is viable. In Figure \ref{syntheticdata} we show the fraction of perfect recovery, where white denotes {\it success} and black denotes {\it failure}. As mentioned in \cite{APG}, the success of APG is approximately below the line $\rho_r+\alpha=0.35.$ However, the success of iEALM is not as good as APG. To conclude, when the sparsity level $\alpha$ is low, our feasibility approach can provide a feasible reconstruction for any $\rho_r$. We note that for low sparsity level, the RPCA algorithms can only provide a feasible reconstruction for $\rho_r\le0.4$. On the other hand, for low $\rho_r$, our feasibility approach can tolerate sparsity level approximately up to 63\%. In contrast, RPCA algorithms can tolerate sparsity up to 50\% for low $\rho_r$. Therefore, taken together, we can argue that our method can be proved useful to solve real-world problems when one wants to recover a moderately sparse matrix having {\it any} inherent low-rank structure present in it or in case of a low-rank matrix corrupted with dense outliers of arbitrary large magnitudes. 
\begin{figure}
    \centering
    \begin{subfigure}{0.32\textwidth}
    \includegraphics[width=\textwidth]{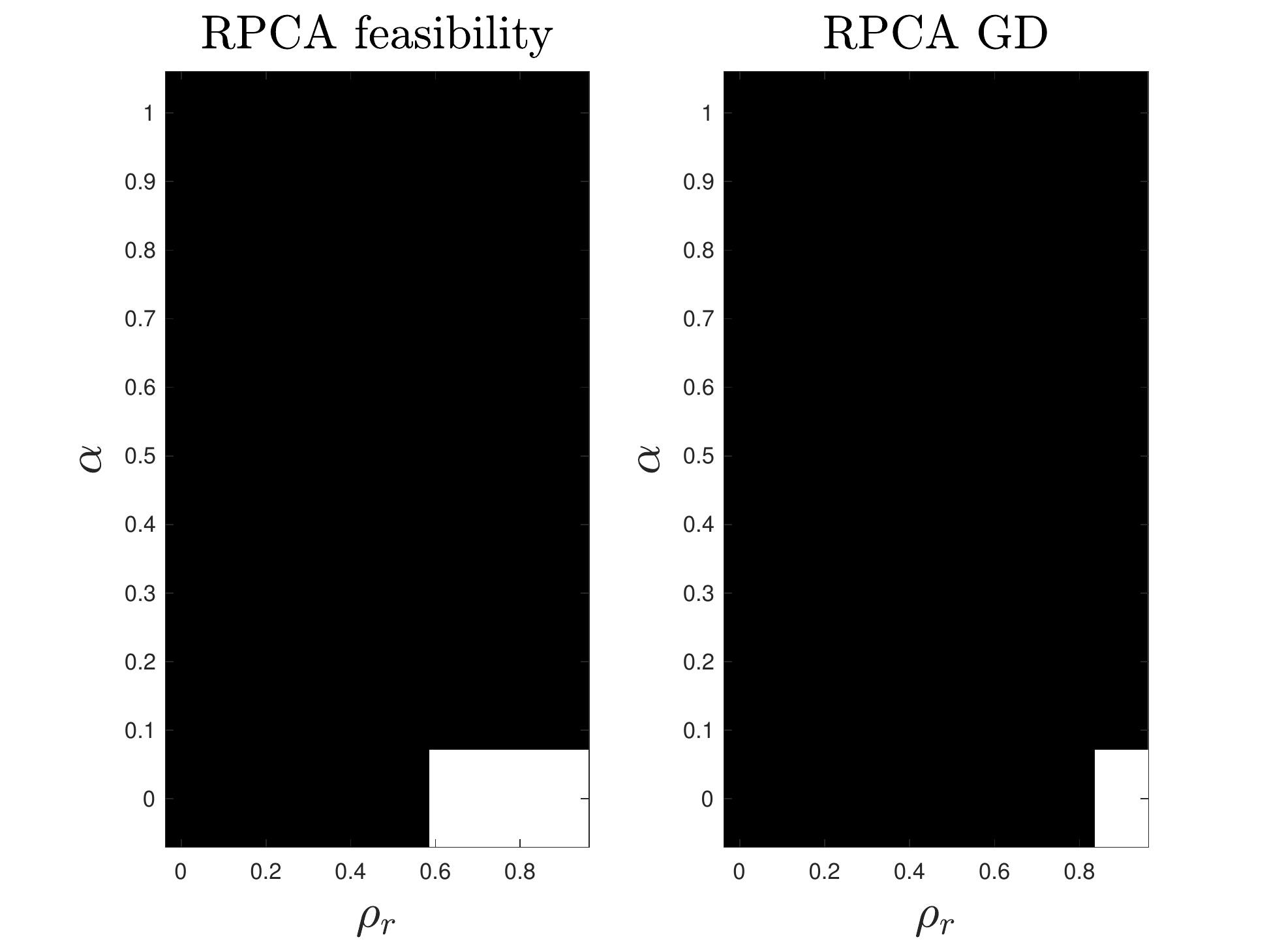}
        \caption{}
      \label{a}
      \end{subfigure}
    \begin{subfigure}{0.32\textwidth}
    \includegraphics[width = \textwidth]{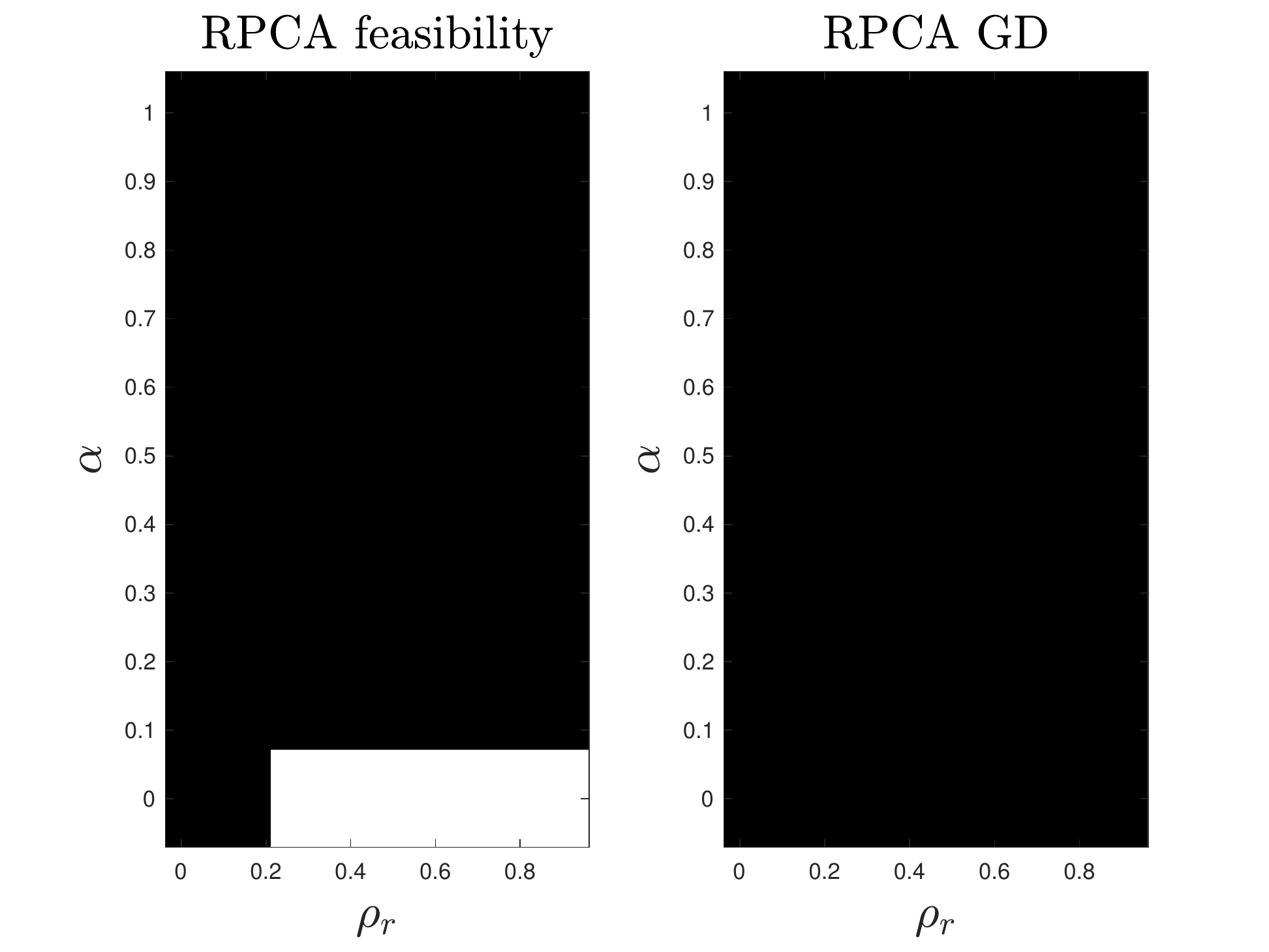}
    \caption{}
    \label{b}
  \end{subfigure} 
   \begin{subfigure}{0.32\textwidth}
    \includegraphics[width = \textwidth]{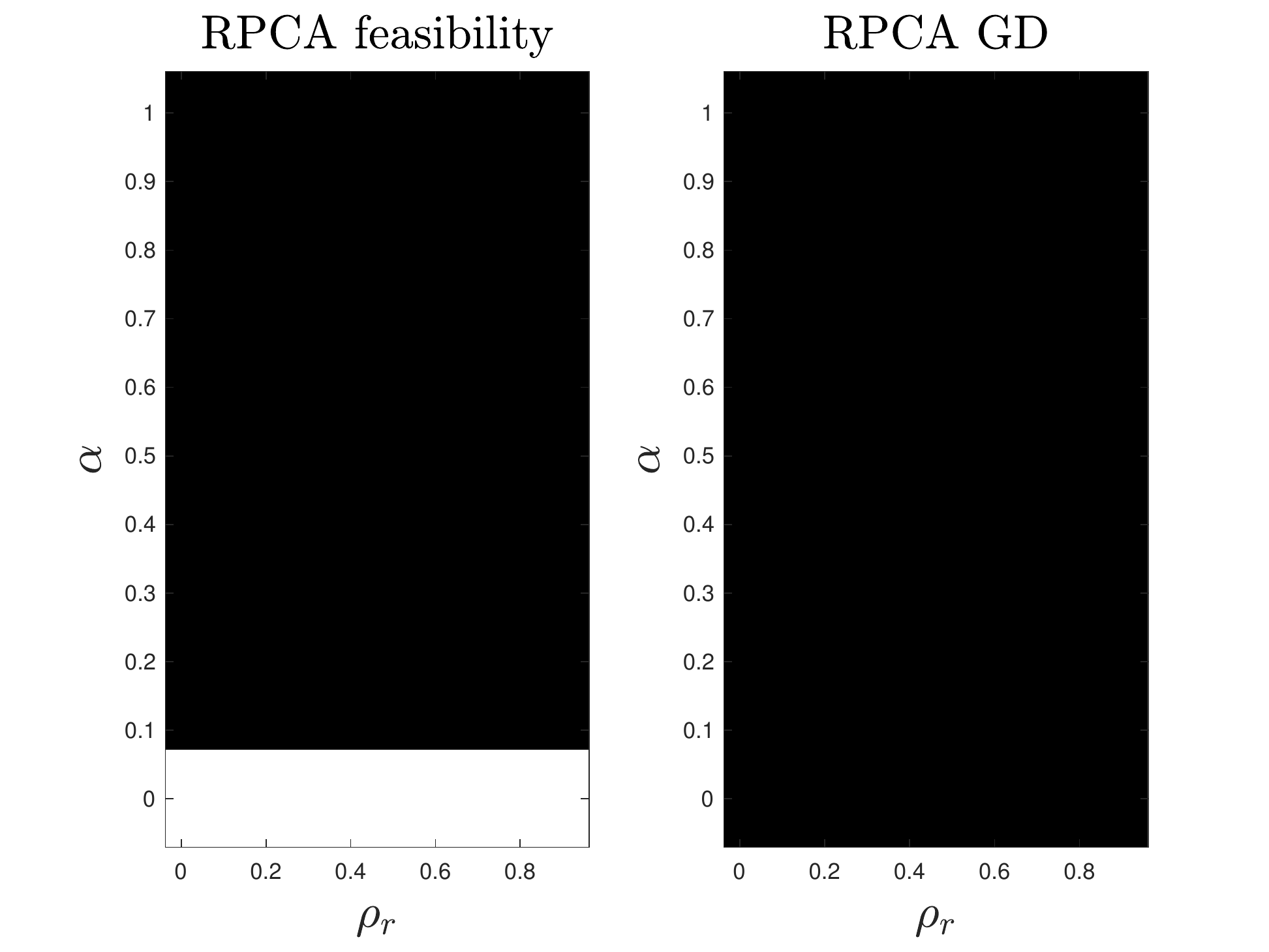}
    \caption{}
    \label{c}
  \end{subfigure} 
  \vspace{-0.1in}
 \caption{\small{Phase transition diagram for Relative error for RMC problems: (a) $|\Omega^{C}|=0.5 (m.n)$, (b) $|\Omega^{C}|=0.75 (m.n)$, (c) $|\Omega^{C}|=0.9 (m.n)$. Here, $\rho_r={\rm rank}(L)/m$ and $\alpha$ is the sparsity measure. We have $(\rho_r,\alpha)\in (0.025,1]\times(0,1)$ with $r=5:25:200$ and $\alpha = {\tt linspace}(0,0.99,8)$.}}
    \label{syntheticdata_rel_error}
\end{figure}
\begin{figure}
    \centering
    \includegraphics[width = \textwidth]{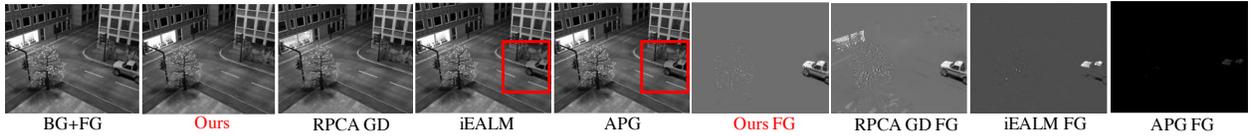}
    \caption{\small{Background and foreground separation on Stuttgart dataset {\tt Basic} video. Except RPCA GD and our method, all other methods fail to remove the static foreground object.}}
    \label{Qual_basic}
\end{figure}

\subsection{Results on synthetic data: RMC problem}\label{sec:rmc}
For experiments in this section, we used a similar technique as in Section \ref{sec:synthetic} to generate the test matrix $A$. We fixed $m=200$  and denote $\rho_r$ and $\alpha$ same as in Section \ref{sec:synthetic}. We randomly select the set of observable entries in $A$. We compare our method against the RPCA gradient descent (RPCA GD) by Yi et al.\ \cite{RPCAgd} and use the relative error for the low-rank component recovered as performance measure, that is, if $\|L-\hat{L}\|_F/\|L\|_F<\tilde{\epsilon}$ then we consider the construction is viable. Note that $L$ is the original low-rank matrix and $\hat{L}$ is the low-rank matrix recovered. For $|\Omega^{C}|=0.5 (m.n), 0.75 (m.n),$ and $0.9 (m.n)$ we consider $\tilde{\epsilon}=0.2, 0.6$, and $1$, respectively. In Figure \ref{syntheticdata_rel_error}, for the phase transition diagram white denotes {\it success} and black denotes {\it failure}. From Figure \ref{syntheticdata_rel_error} we observe that irrespective of the cardinalities of the set of the observed entries our feasibility approach outperforms RPCA GD. However, as the cardinality of the set of the observable entries, that is, $|\Omega|$ decreases, the performance of our feasibility approach gets better (see Figure \ref{c}). 

\begin{figure}
    \centering
    \begin{subfigure}{0.49\textwidth}
    \includegraphics[width=\textwidth]{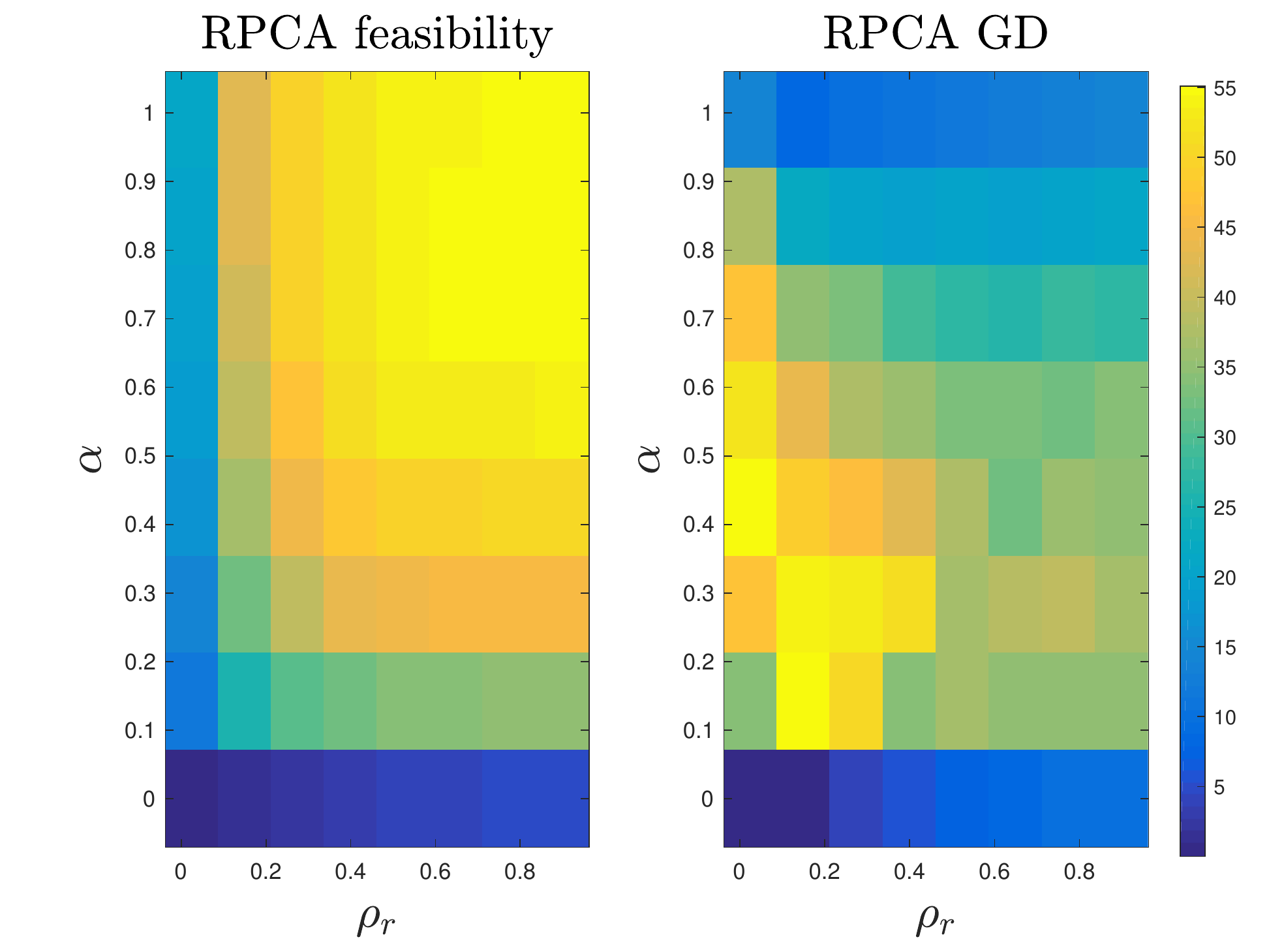}
        \caption{}
      \label{1a}
      \end{subfigure}
    \begin{subfigure}{0.49\textwidth}
    \includegraphics[width = \textwidth]{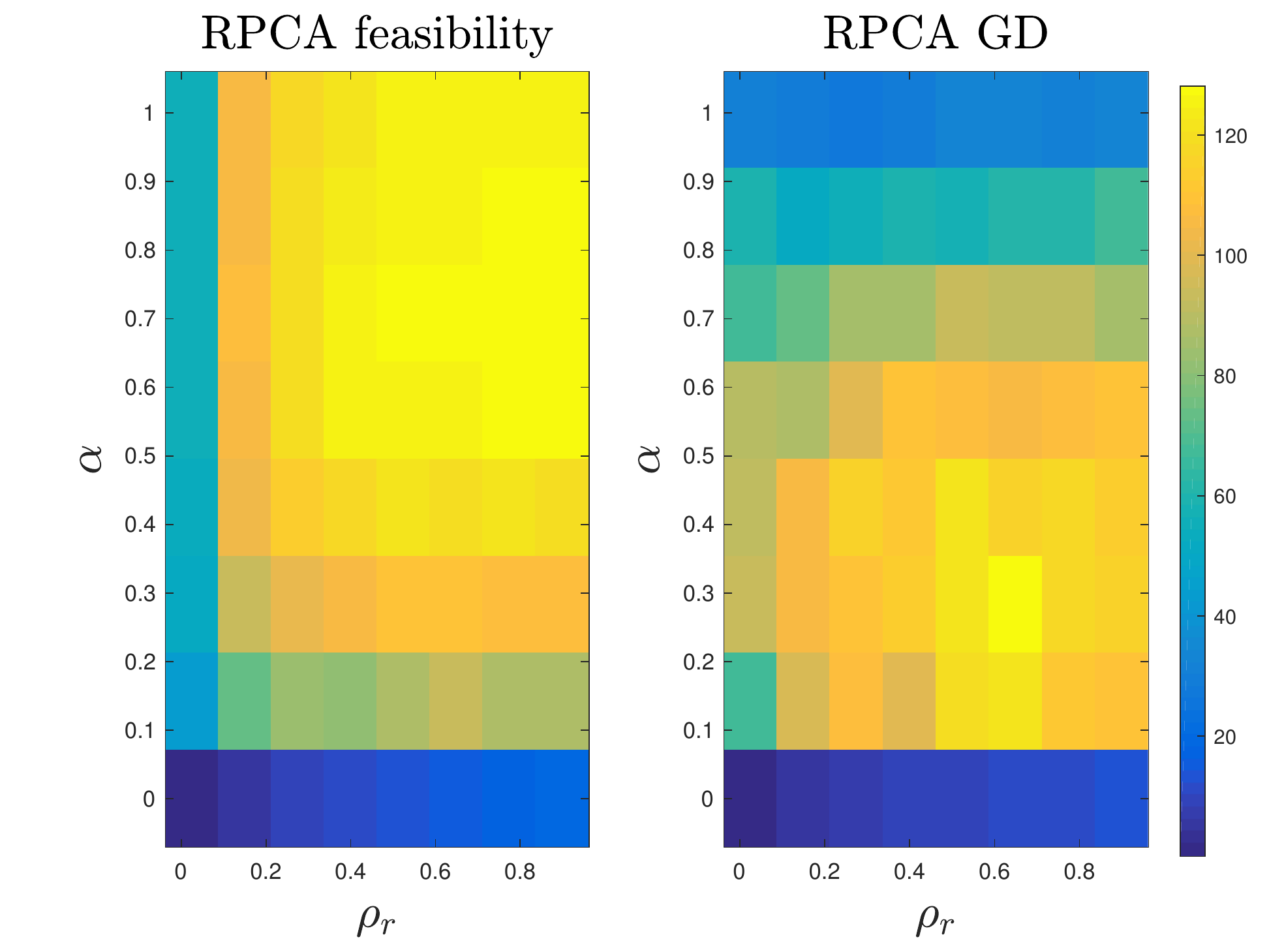}
    \caption{}
    \label{1b}
  \end{subfigure} 
   \begin{subfigure}{0.49\textwidth}
    \includegraphics[width = \textwidth]{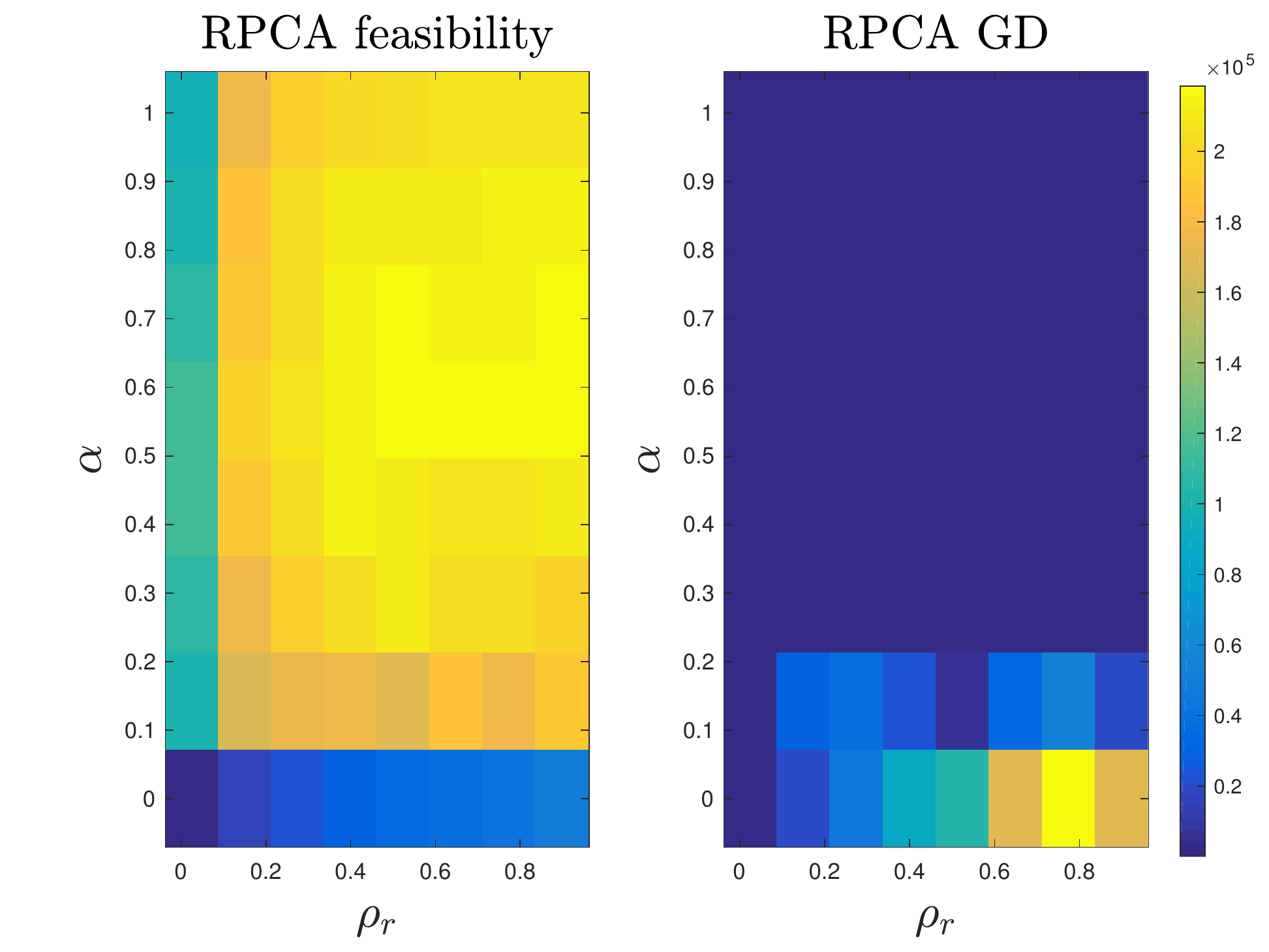}
    \caption{}
    \label{1c}
  \end{subfigure} 
  \begin{subfigure}{0.49\textwidth}
    \includegraphics[width = \textwidth]{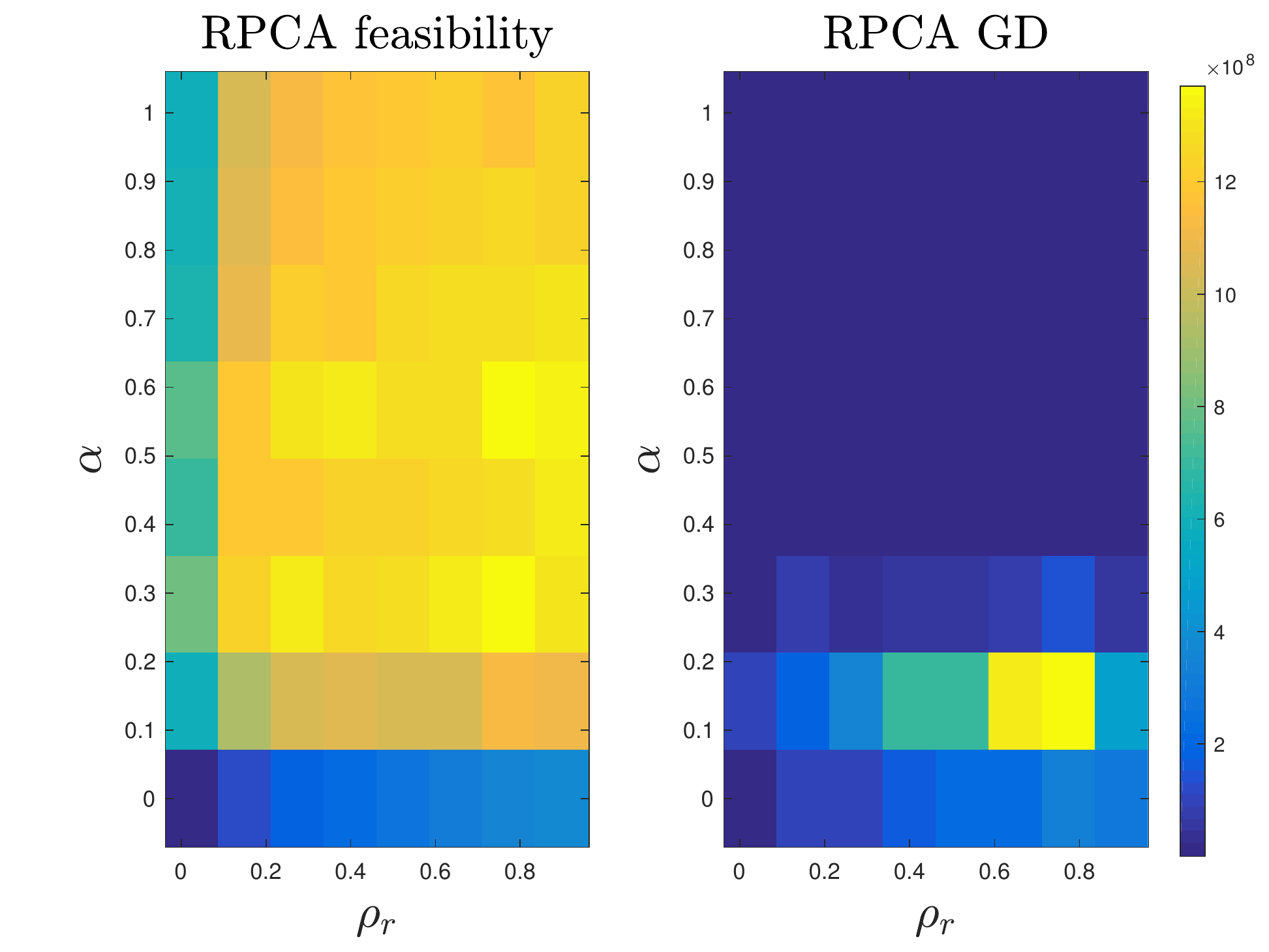}
    \caption{}
    \label{1d}
  \end{subfigure} 
  \vspace{-0.1in}
 \caption{\small{RMSE for RMC problems: (a) $|\Omega^{C}|=0.5 (m.n)$, (b) $|\Omega^{C}|=0.75 (m.n)$, (c) $|\Omega^{C}|=0.9 (m.n)$, (d) $|\Omega^{C}|=0.95 (m.n)$. Here, $\rho_r={\rm rank}(L)/m$ and $\alpha$ is the sparsity measure. We have $(\rho_r,\alpha)\in (0.025,1]\times(0,1)$ with $r=5:25:200$ and $\alpha = {\tt linspace}(0,0.99,8)$.}}
    \label{syntheticdata_rmse}
\end{figure}
\begin{figure}
    \centering
    \includegraphics[width = \textwidth]{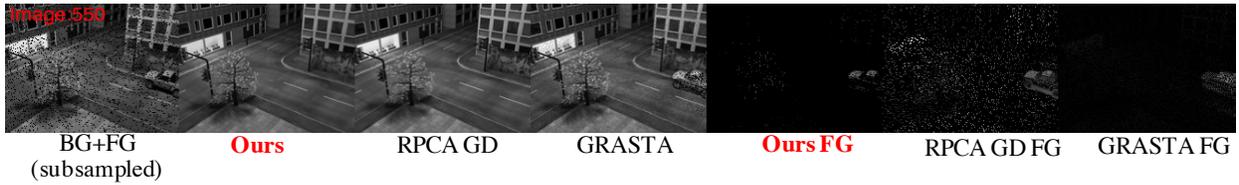}
    \caption{\small{Background and foreground separation on Stuttgart dataset {\tt Basic} video. We used 90\% sample. GRASTA forms a fragmentary background and exhausts around 540 frames to form a stable video. We also note that RPCA GD has more false positives in the foreground.}}
    \label{Qual_basic2}
\end{figure}
Next, we use the root mean square error (RMSE), that is, $\|L-\hat{L}\|_F/\sqrt{mn}$ as a performance measure for these set of results. Note that $L$ is the original low-rank matrix and $\hat{L}$ is the low-rank matrix recovered. From Figure \ref{syntheticdata_rmse} we observe that when the cardinality of the set of the observable entries $\Omega$ is 50\% and 75\% of $[m]\times [n]$, respectively, RPCA GD has slightly better RMSE than our method as $\rho_r$ increases. However, as the cardinality of the set of the observable entries, that is, $|\Omega|$ decreases, we outperform RPCA GD (see Figure \ref{1c}-\ref{1d}). Therefore, we further validate that for RMC problems, when $|\Omega|$ is small the feasibility approach is better to recover a low-rank matrix. 

\subsection{Applications to real-world problem}
In this section we demonstrate the robustness of our feasibility approach to solve four classic real-world problems: i) background and foreground estimation from fully and partially observed data, ii) shadow removal from face images captured under varying illumination and camera position, iii) inlier subspace detection, iv) processing astronomical data.

\subsubsection{Background and foreground estimation from fully observed data}\label{sec:bg}
In this section, we show our results on the background estimation problem. In the past decade, one of the most prevalent approaches used to solve background estimation problem is to treat it as a low-rank and sparse matrix decomposition problem \cite{Bouwmans2016,Sobral20144,duttaligongshah,RPCA-BL,prmf,grasta,gosus,dutta_thesis,inWLR,duttalirichtarik_modeling,duttali_bg}. Given a sequence of $n$ video frames with each frame mapped into a vector ${a}_i\in {\mathbb R}^m$, $i=1,2,...,n$, the data matrix $A\in {\mathbb R}^{m\times n}$ in the collection of all the frame vectors is expected to be split into $L+S$. By using the above idea, RPCA \cite{candeslimawright,LinChenMa,APG} was introduced by considering the background frames, $L$, having a low-rank structure and the foreground, $S$, sparse. The convex relaxation of the problem is  \eqref{rpca}. 
\begin{figure}
    \centering
    \begin{subfigure}{0.49\textwidth}
    \includegraphics[width=\textwidth]{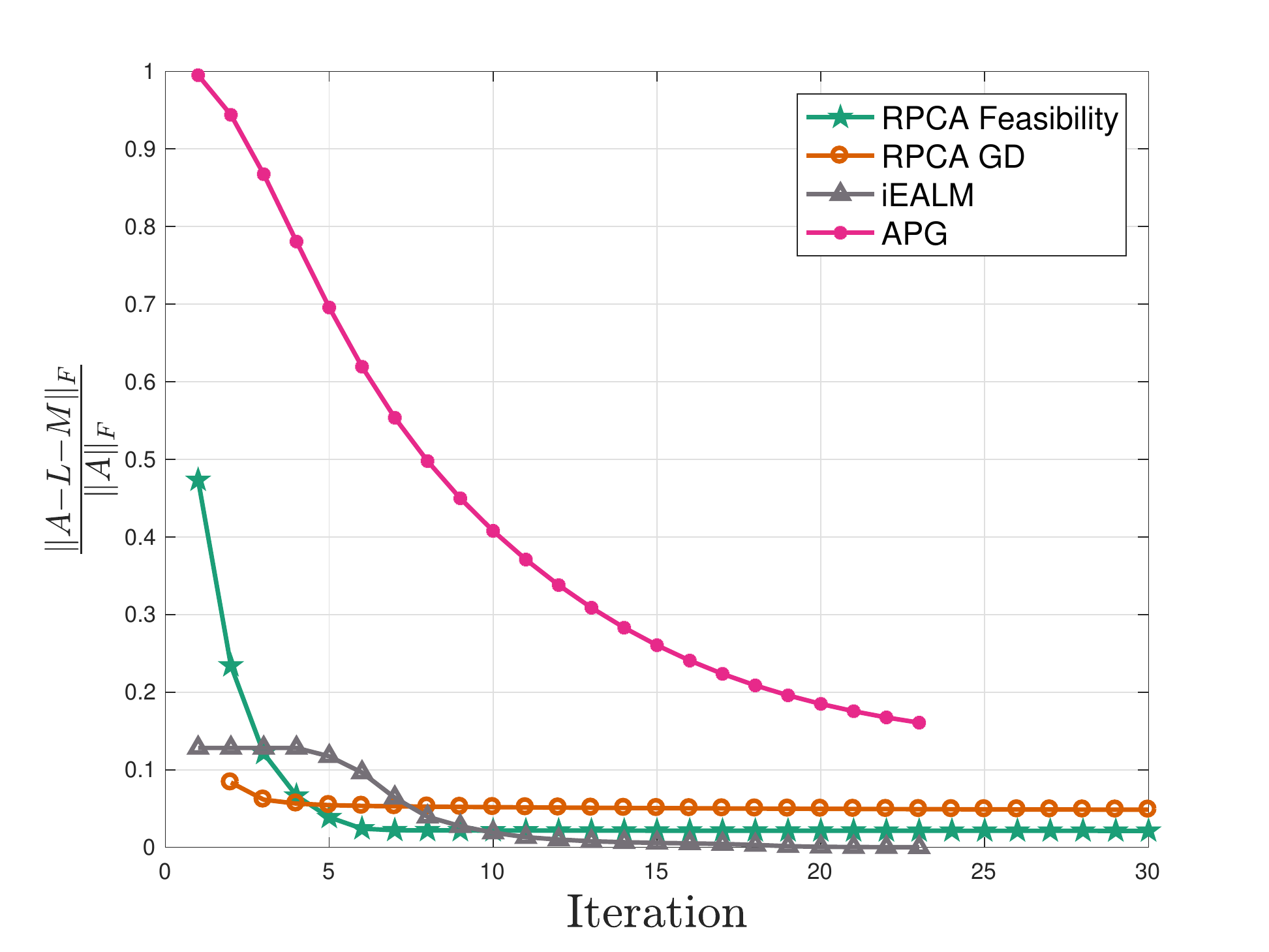}
        \caption{}
      \label{Rel_bg}
      \end{subfigure}
    \begin{subfigure}{0.49\textwidth}
    \includegraphics[width = \textwidth]{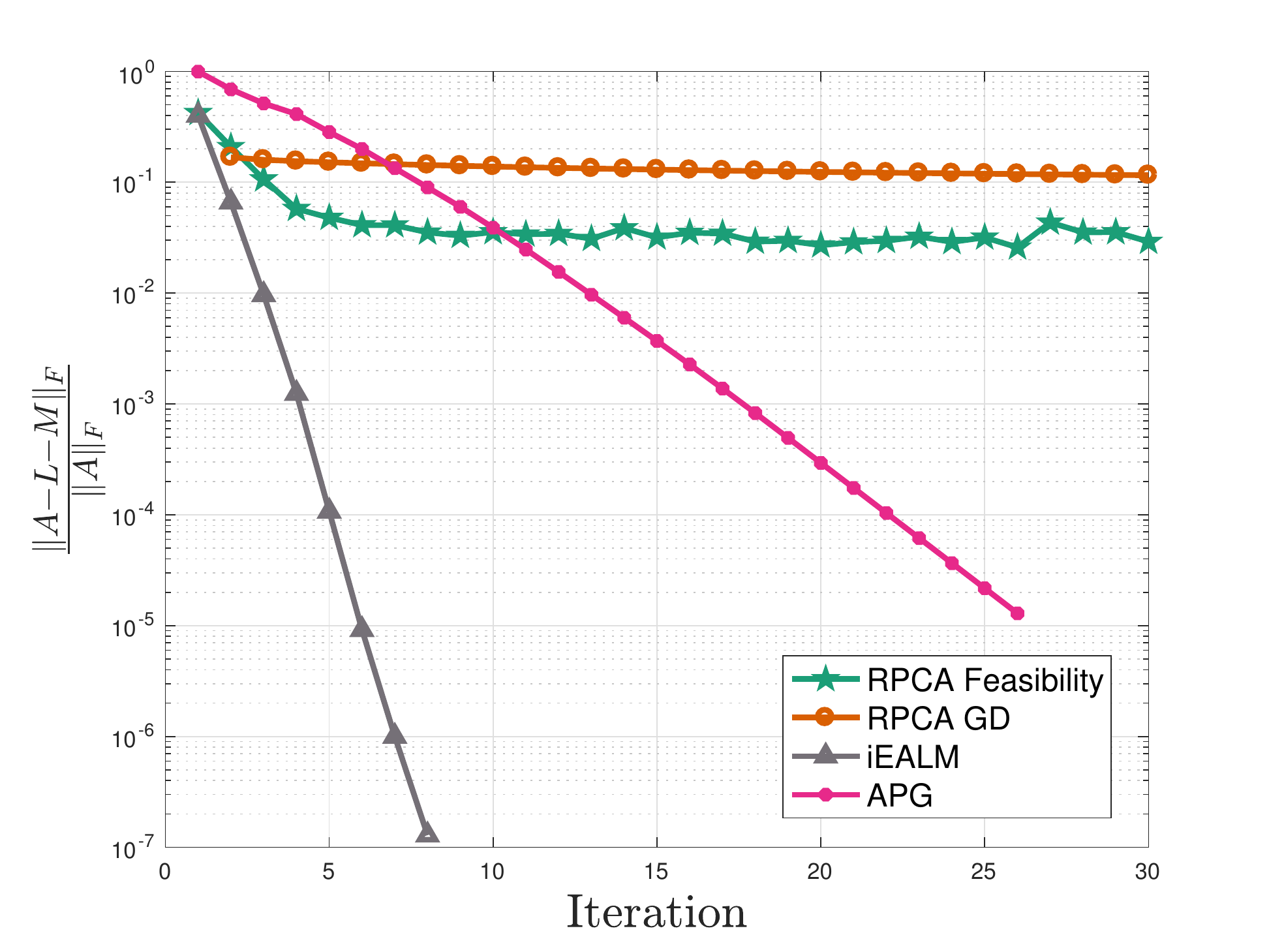}
    \caption{}
    \label{Rel_shadow}
  \end{subfigure} 
  \vspace{-0.1in}
  \caption{\small{(a) Comparison of relative error vs. iteration between RPCA F, iEALM, APG, and RPCA GD on {\tt Basic} video, frame size $144\times176$. iEALM takes 55.41 seconds, RPCA GD takes 36.08 seconds (30 iterations), APG takes 51.14 seconds, and RPCA CF takes 42.72 seconds (30 iterations). The threshold $\epsilon$ for all algorithms is set to $2\times10^{-4}.$ (b) Comparison of relative error (log scale) vs. iteration between RPCA F, iEALM, APG, and RPCA GD on Shadow removal, {\tt Yale Extended Face dataset}, subject B12. iEALM takes 2.05 seconds (threshold $10^{-7}$), RPCA GD takes 4.73 seconds (30 iterations, threshold $2\times10^{-4}$), APG takes 10.9 seconds (threshold $10^{-7}$), and RPCA CF takes 4.71 seconds (30 iterations, threshold $2\times10^{-4}$). For APG and iEALM we plot every fifth iteration.}}
  \label{rel_err}
\end{figure}

For simulations, we used the {\tt Basic} sequence of the Stuttgart artificial dataset \cite{cvpr11brutzer}. We compare our methods against inexact augmented Lagrange methods of multiplier (iEALM) of Lin et al.\ \cite{LinChenMa}, accelerated proximal gradient (APG) of Wright et al.\ \cite{APG}, and RPCA GD. We downsampled the video frames to a resolution of $144\times 176$ and for iEALM we use $\mu=1.25/\|A\|_2$ and $\rho=1.5$. For both APG and iEALM we set $\lambda=1/\sqrt{\max\{m,n\}}$. For RPCA GD and our method we use target rank $r=2$, sparsity $\alpha = 0.1$. The threshold $\epsilon$ for all methods are kept to $2\times 10^{-4}$. 
The qualitative analysis on the background and foreground recovered on the sample frame of the {\tt Basic} sequence in Figure \ref{Qual_basic} suggests that our method and RPC GD recover a visually better quality background and foreground compare with the other methods. We also note that RPCA GD recovers a foreground with more false positives compare to our method and iEALM and APG cannot remove the static foreground object. 
\begin{figure}
    \centering
    \includegraphics[width = \textwidth]{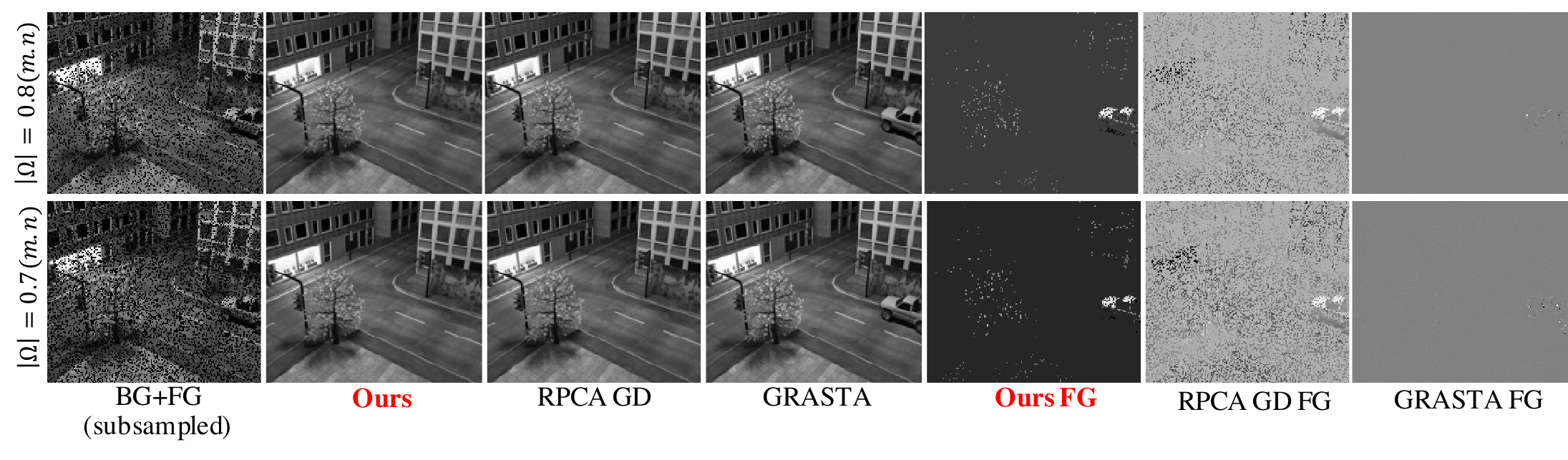}
    \caption{\small{Background and foreground separation on Stuttgart dataset {\tt Basic} video. We used 80\% and 70\% sample, respectively. GRASTA forms a fragmentary background and exhausts around 540 frames to form a stable video. We also note that RPCA GD has more false positives in the foreground. In contrast our feasibility approach recovers superior quality background and foreground.}}
    \label{Qual_basic3}
\end{figure}
\begin{figure}
    \centering
    \begin{subfigure}{0.32\textwidth}
    \includegraphics[width=\textwidth]{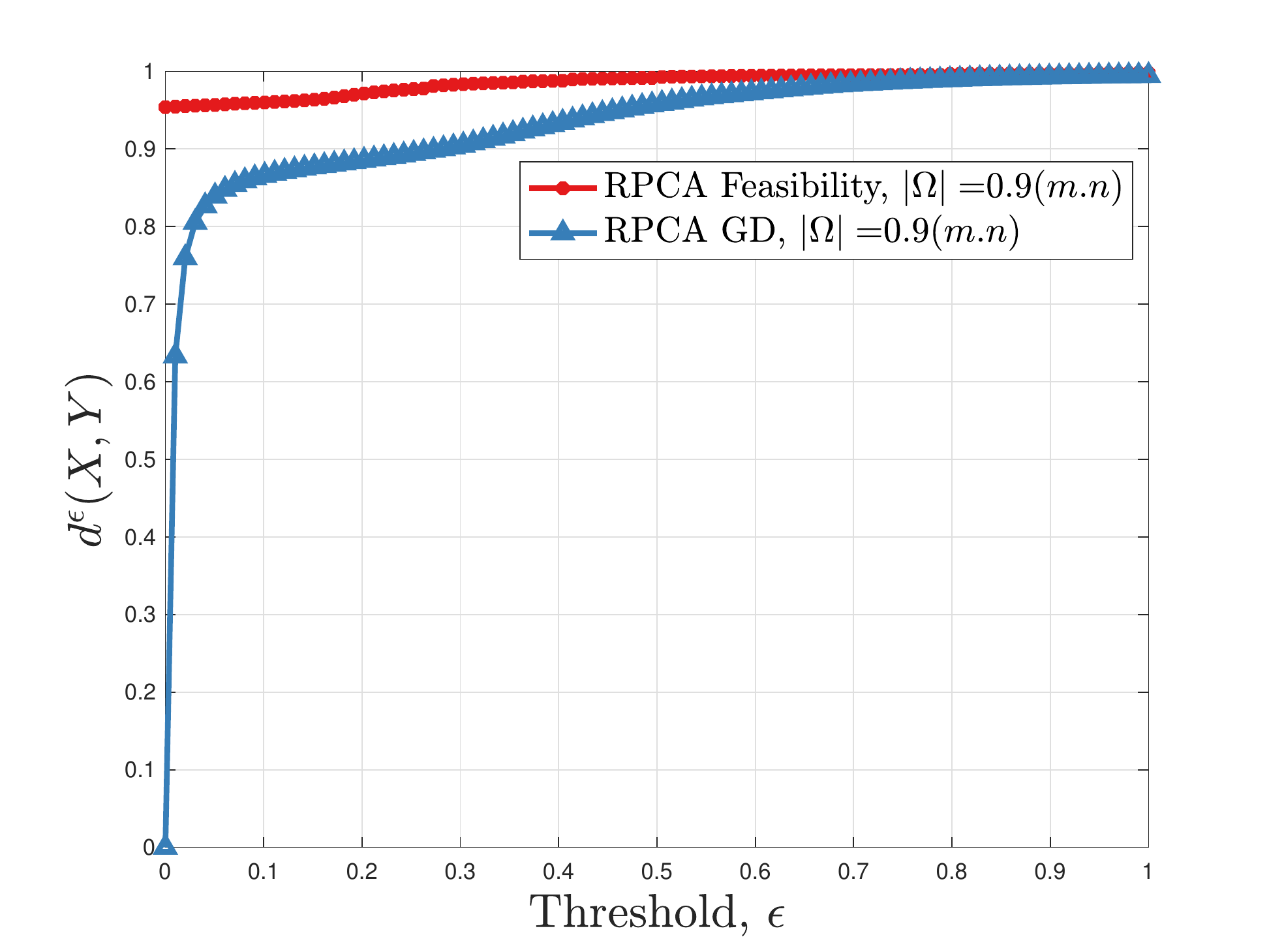}
        \caption{}
      \label{1aa}
      \end{subfigure}
    \begin{subfigure}{0.32\textwidth}
    \includegraphics[width = \textwidth]{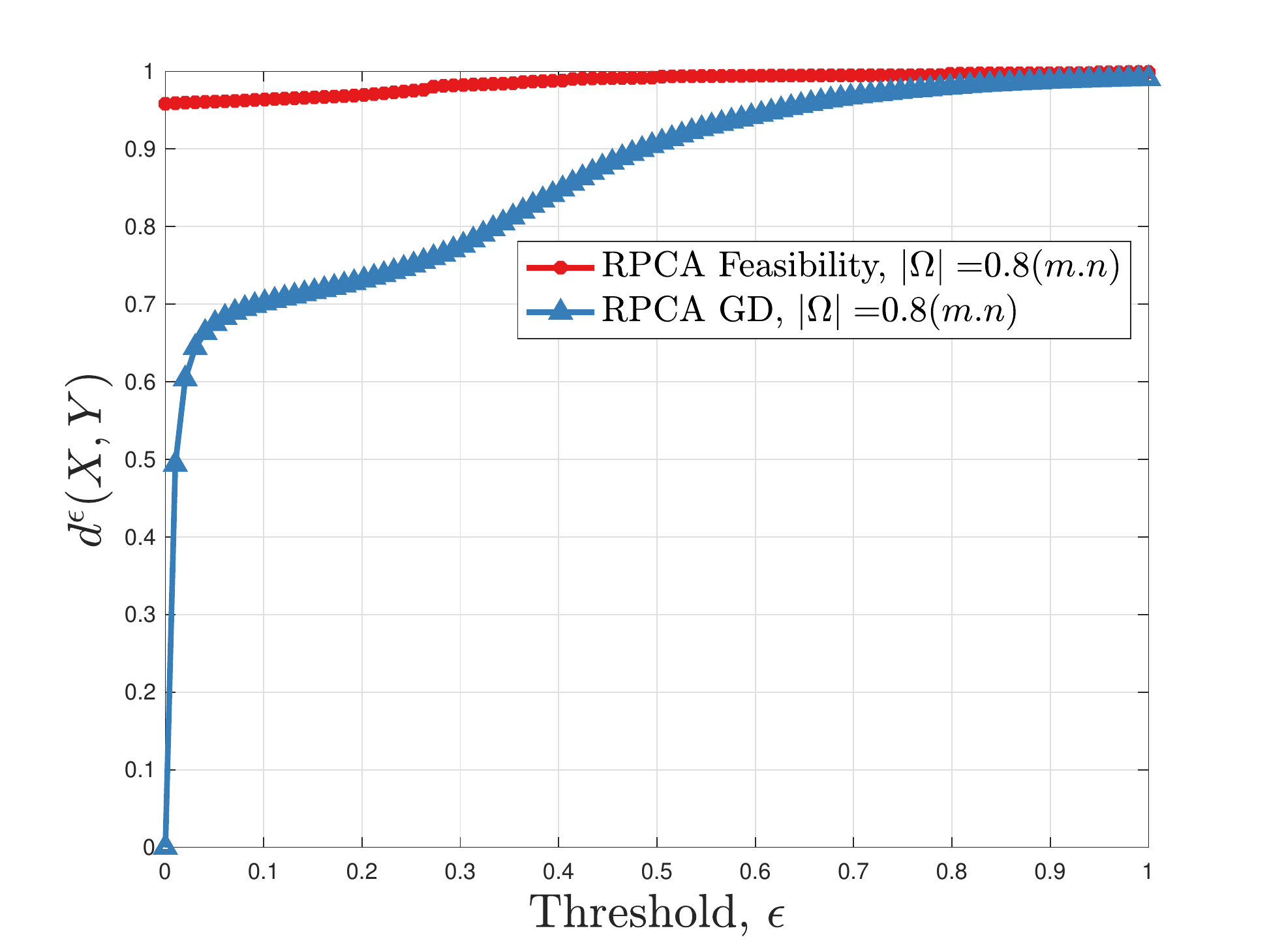}
    \caption{}
    \label{1bb}
  \end{subfigure} 
   \begin{subfigure}{0.32\textwidth}
    \includegraphics[width = \textwidth]{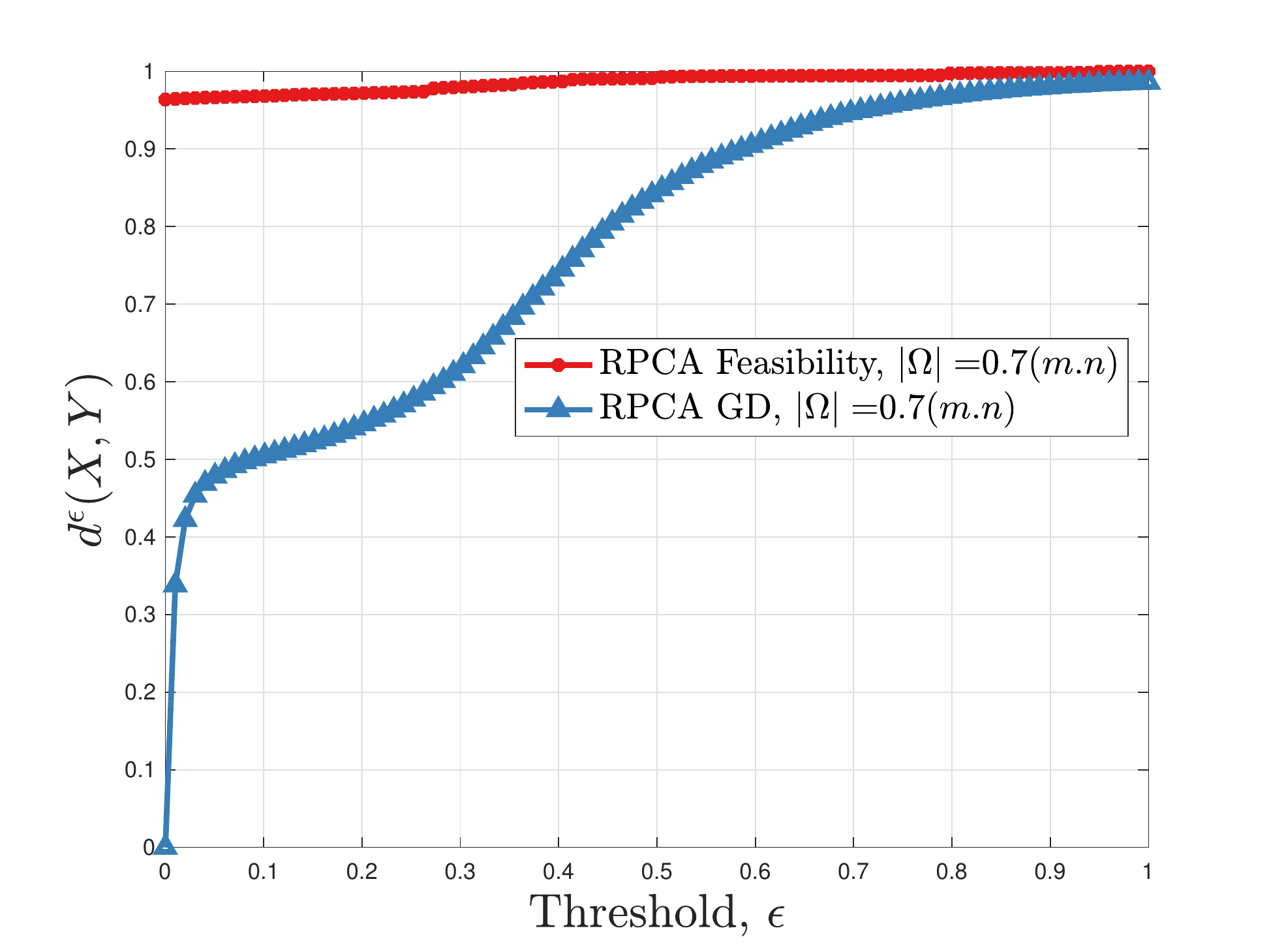}
    \caption{}
    \label{1cc}
  \end{subfigure} 
  \begin{subfigure}{0.32\textwidth}
    \includegraphics[width = \textwidth]{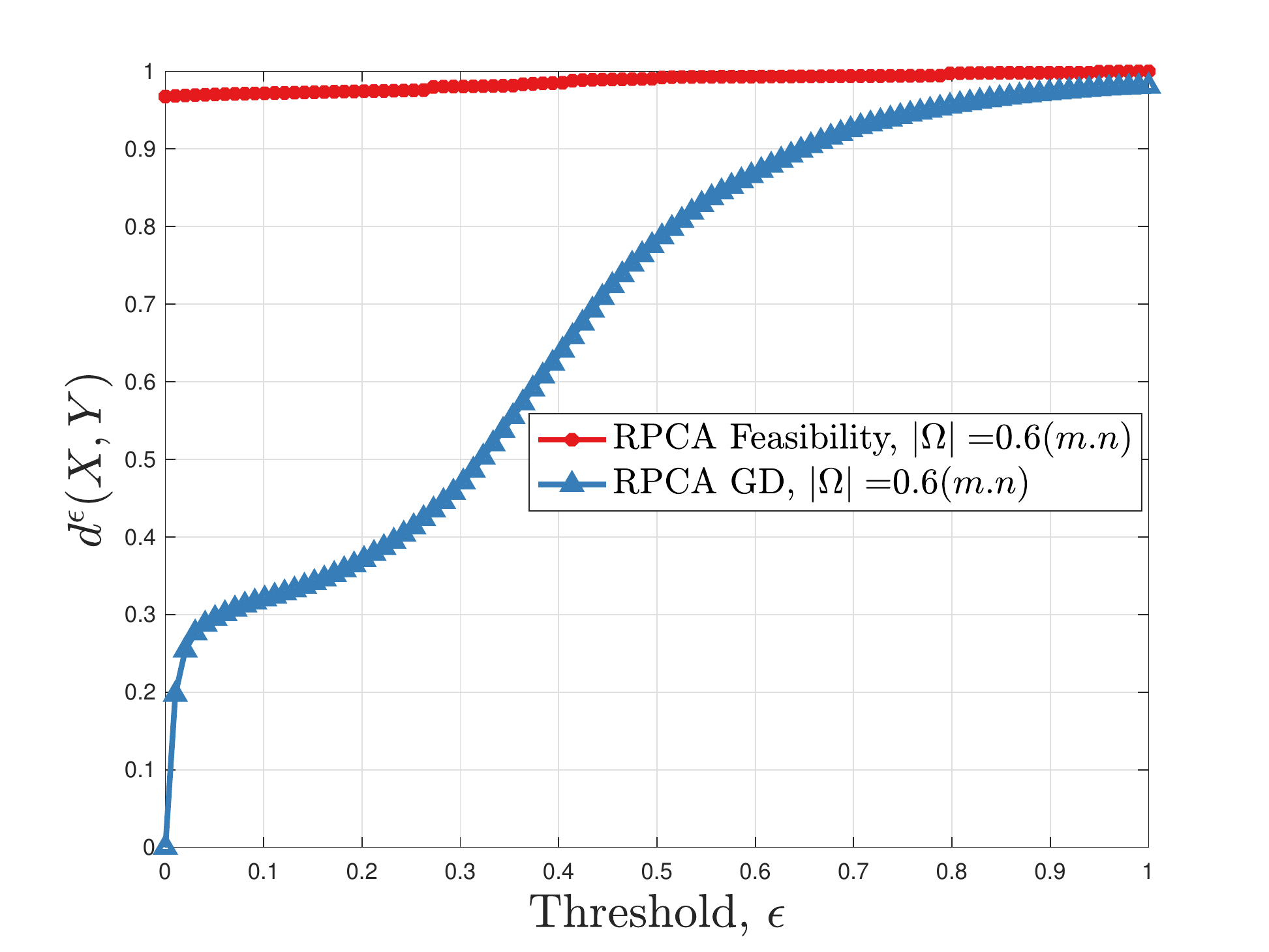}
    \caption{}
    \label{1dd}
      \end{subfigure} 
     \begin{subfigure}{0.32\textwidth}
    \includegraphics[width = \textwidth]{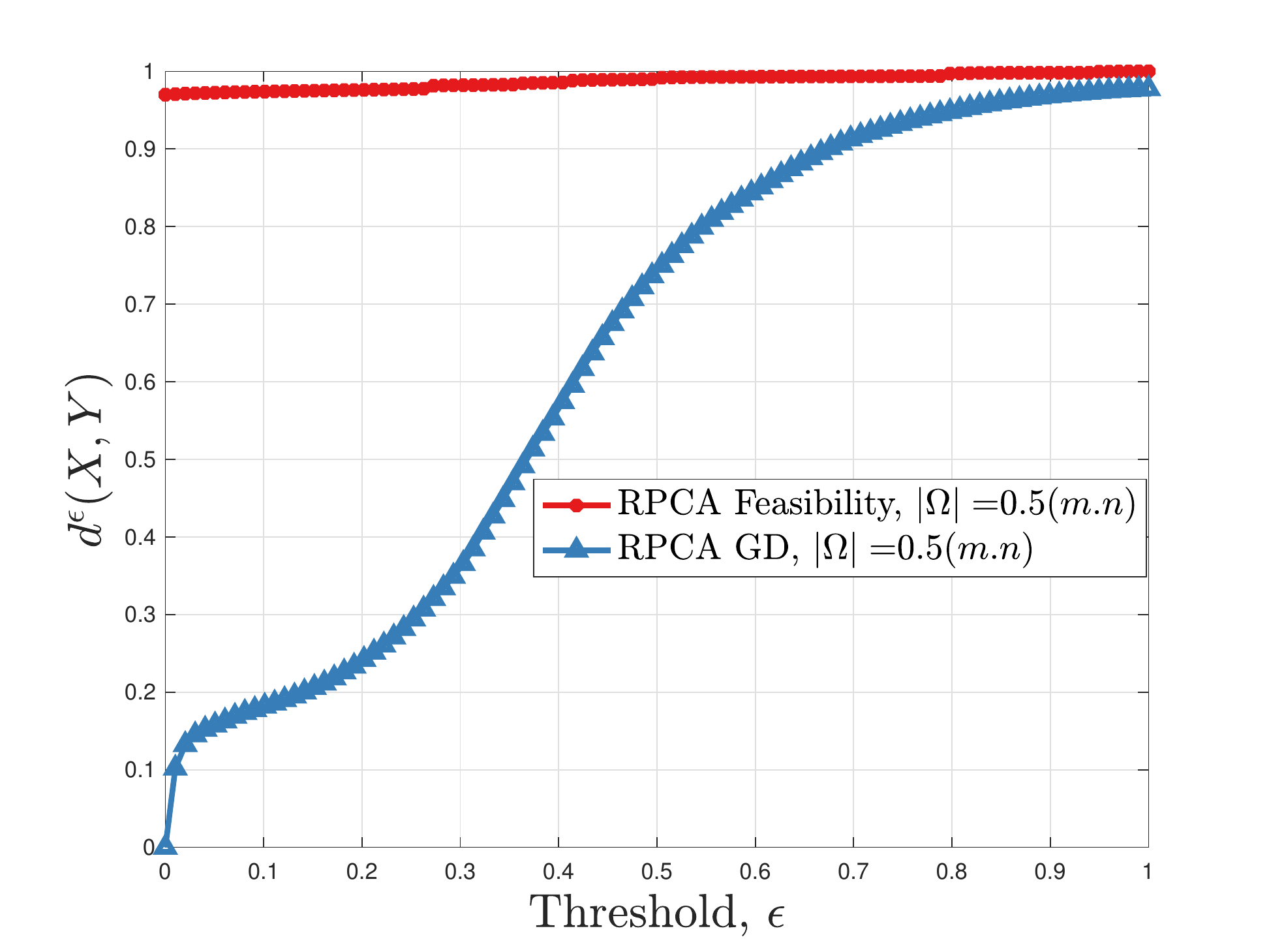}
    \caption{}
    \label{1ee}
      \end{subfigure} 
     \begin{subfigure}{0.32\textwidth}
    \includegraphics[width = \textwidth]{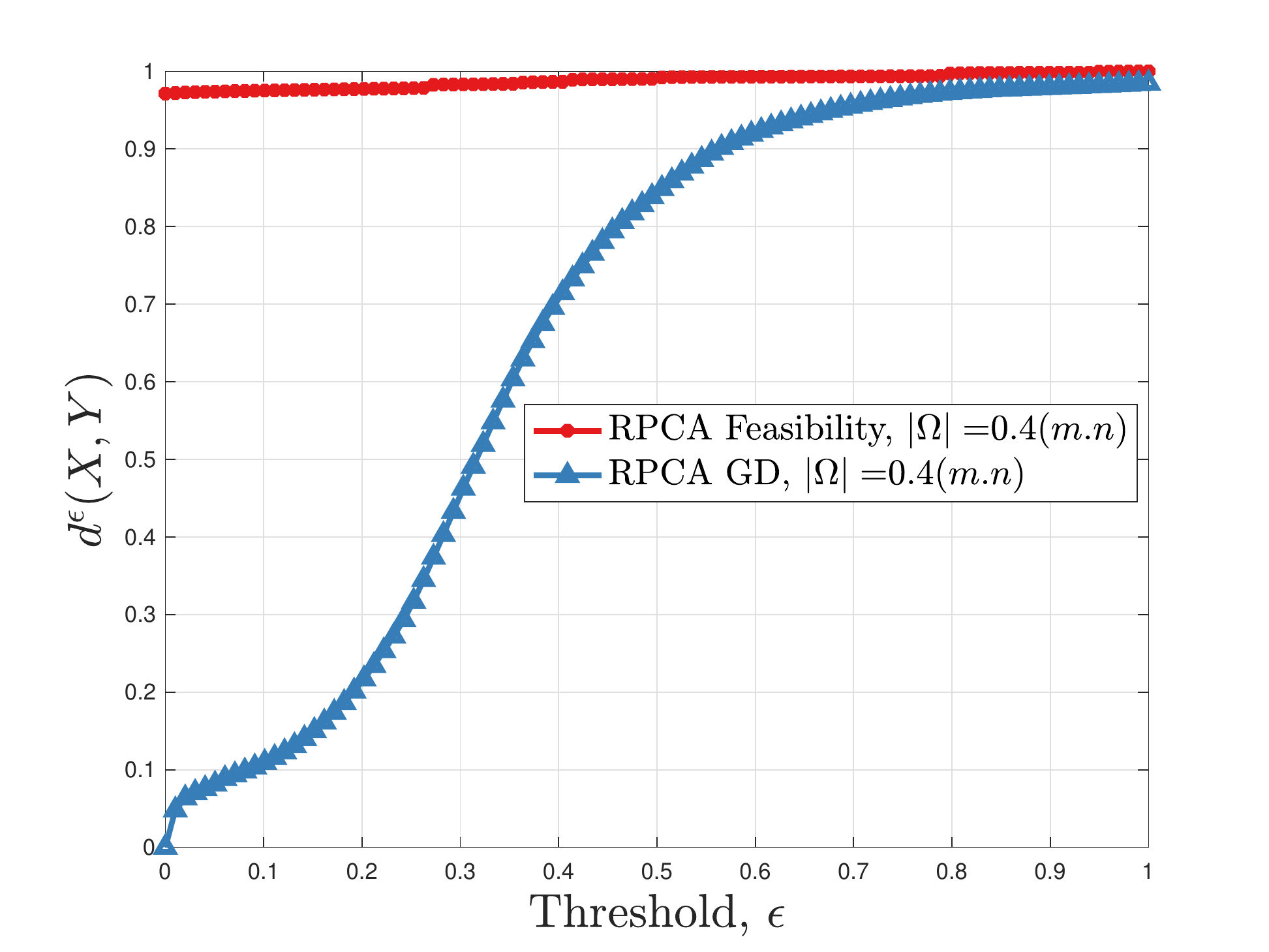}
    \caption{}
    \label{1ff}
  \end{subfigure} 
\caption{\small{Quantitative comparison of foreground recovered by RPCA GD and RPCA F on {\tt Basic} video, frame size $144\times176$ with observable entries: (a) $|\Omega|=0.9(m.n)$, (b) $|\Omega|=0.8(m.n),$ (c) $|\Omega|=0.7(m.n)$, (d) $|\Omega|=0.6(m.n)$, (e) $|\Omega|=0.5(m.n)$, and (f) $|\Omega|=0.4(m.n)$. The performance of RPCA GD drops significantly as $|\Omega|$ decreases. In contrast, the performance of RPCA F stays stable irrespective of the size of $|\Omega|$.}}
  \label{FP_rmc}
\end{figure}
\begin{figure}
    \centering
    \includegraphics[width = \textwidth]{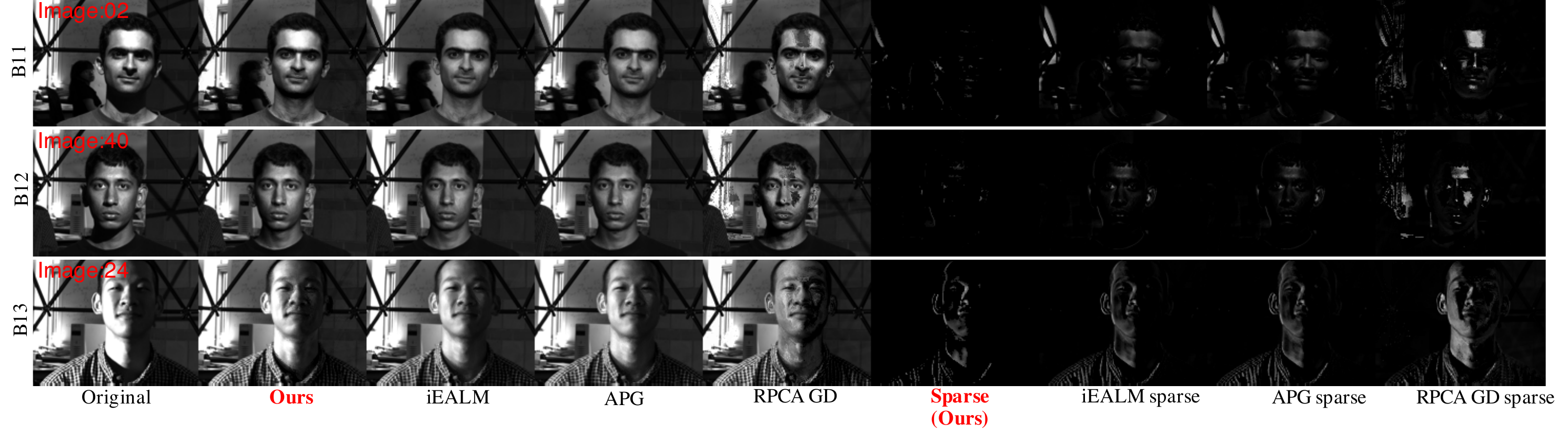}
    \caption{\small{Shadow and specularities removal from face images captured under varying illumination and camera position. Our feasibility approach provides comparable reconstruction to that of iEALM and APG.}}
    \label{shadow_removal}
\end{figure}

\subsubsection{Background and foreground estimation from partially observed data}\label{sec:bg1}
We randomly select the set of observable entries in the data matrix $A$ and tested our algorithm against Grassmannian Robust Adaptive Subspace Tracking Algorithm (GRASTA) \cite{grasta} and RPCA GD. In Figure \ref{Qual_basic2}, we demonstrate the performance on the {\tt Basic} sequence of the Stuttgart dataset with $|\Omega|=0.9(m.n).$ The parameters for our algorithm and RPCA GD are set as same as in Section~\ref{sec:bg}. For GRATSA we set the parameters same as mentioned in the authors' website\footnote{\url{https://sites.google.com/site/hejunzz/grasta}}. Next, in Figure \ref{Qual_basic3} we show the background and foreground separated by different methods on the same sample frame of the {\tt Basic} sequence of Stuttgart dataset for different subsample rate.  
It is evident that RPCA GD and our approach has the best background reconstruction. However, when compared with the foreground ground truth our method has a better quantitative measure as RPCA GD has a higher number of false positives in the foreground~(see Figure \ref{FP_rmc}).

\paragraph{Background and foreground estimation from partially observed data: Quantitative measure.}\label{sec:bg3}
Let $X=(X_1,\dots,X_n)$ and $y=(Y_1,\dots,Y_n)$ be  two video sequences (reconstructed foreground and ground truth foreground), where $X_i, Y_i\in \R^m$ are vectors corresponding to frame $i$, each containing $m$ pixels.  We scale all pixel values to $[0,1].$ To compare the video sequences, we define  an $\epsilon$--proximity measure of $X$ and $Y$ as
\[d^{\epsilon}(X,Y)\eqdef \frac{1}{nm} \sum_{i=1}^n \sum_{k=1}^{m} d^{\epsilon}(X_{ik},Y_{ik}),\] where  \[d^{\epsilon}(u,v) \eqdef \begin{cases} 1&\quad  |u-v|\le\epsilon,\\
0 & \quad \text{otherwise,}
\end{cases}\] 
and $\epsilon \in [0,1]$ is a threshold. Clearly, $0\leq d^\epsilon(X,Y) \leq 1$, $\epsilon \mapsto d^\epsilon(X,Y)$ is  increasing, and $d^1(X,Y) = 1$. If $d^\epsilon (X,Y) =\alpha$, then $\alpha\times 100  \%$ of pixels in the recovered video are within $\epsilon$ distance, in absolute value, from the ground truth.

In Figure \ref{FP_rmc}, we plot $d^\epsilon(X,Y)$ as a function of $\epsilon$ for our method and  RPCA GD. We use the {\tt Basic} sequence of the Stuttgart dataset and vary the cardinality of the set of observable entries $\Omega$. Our feasibility approach outperforms RPCA GD for all values of $|\Omega|$ and $\epsilon$, and the difference is striking; in particular, our method recovers more than 95\% pixels correctly even for under accuracy (i.e., small $\epsilon$) requirements.
\subsubsection{Shadow removal}
The images of a face exposed to a wide variety of lighting conditions can be approximated accurately by a low-dimensional linear subspace. More specifically, the images under distant, isotropic lighting lie close to a 9-dimensional linear subspace which is known as the harmonic plane \cite{basri_jacobs}. We used the {\tt Extended Yale Face Database} for our experiments \cite{yale_face}. We used iEALM, APG, and RPCA GD to compare against our algorithm. We downsampled each image to a resolution of $120\times 160$ and use 63 images of a subject in each test. For APG and iEALM, we set the parameters same as in Section \ref{sec:bg}. For RPCA GD and our method, we set target rank $r=9$ and sparsity level $\alpha = 0.1$. 
The qualitative analysis on the recovered images shows that our feasibility approach provides a comparable reconstruction similar to that of iEALM and APG (see Figure \ref{shadow_removal}). In contrast, the reconstructed face images by RPCA GD are of poor visual quality. 

\subsubsection{Inlier detection}\label{inlier detection}
Our next set of experiments demonstrate the power of our method in detecting the inliers and the outliers from a composite dataset. For this purpose, we artificially create a dataset that contains both inliers and outliers. We used the {\tt Yale Extended Face Database} to construct a data set that contains images of faces under different illuminations. We denote this as inliers. With these inliers, we infused 400 random natural images from the BACKGROUND/Google folder of the Caltech101 database \cite{caltech101} that serve as outliers. Both inlier and outlier images were converted to grayscale and the resolution is downsampled to $20\times20$ pixels. For the inliers, we are looking for the 9-dimensional linear subspace where the images of the same face lie. That is, similar to \cite{rspca} we consider a low-dimensional model to the set of all faces aka inliers. We note that the seven algorithms proposed in \cite{rspca} are designed to explicitly find a low-rank subspace. In \cite{rspca} the authors used different objective functions and used SGD, incremental approach, and mirror descent algorithms to find the low-dimensional subspace. However, we approach the problem slightly differently. We split the dataset, $A$, into a 9-dimensional low-rank subspace $L$ and expect the specularities and outliers to be in the sparse set $S$. Once we find $L$, we extract the basis of $L$ and project the faces on it.  In Figure \ref{inlier}, we show the qualitative results of our experiments\footnote{The codes and datasets for experiments in Section \ref{inlier detection} and \ref{telescopic data} are obtained from \url{https://github.com/jwgoes/RSPCA}}.

As proposed in \cite{rspca}, we use the error term $\|P_L-P_{L^*}\|_F/{3\sqrt{2}}$, where $L$ is subspace fitted by the PCA to the set of inliers and $L^*$ be the subspace fitted by different algorithms. We normalize the quantity $\|P_L-P_{L^*}\|_F$ because when $L\perp L^*$, $\|P_L-P_{L^*}\|_F/{3\sqrt{2}}\approx 1.$ Therefore, $\|P_L-P_{L^*}\|_F/{3\sqrt{2}}$ is expected to lie between 0 and 1 where the smaller is the better. 
We refer to Table \ref{inlier_data} for our quantitative results. 
\begin{figure}
    \centering
    \includegraphics[width =\textwidth]{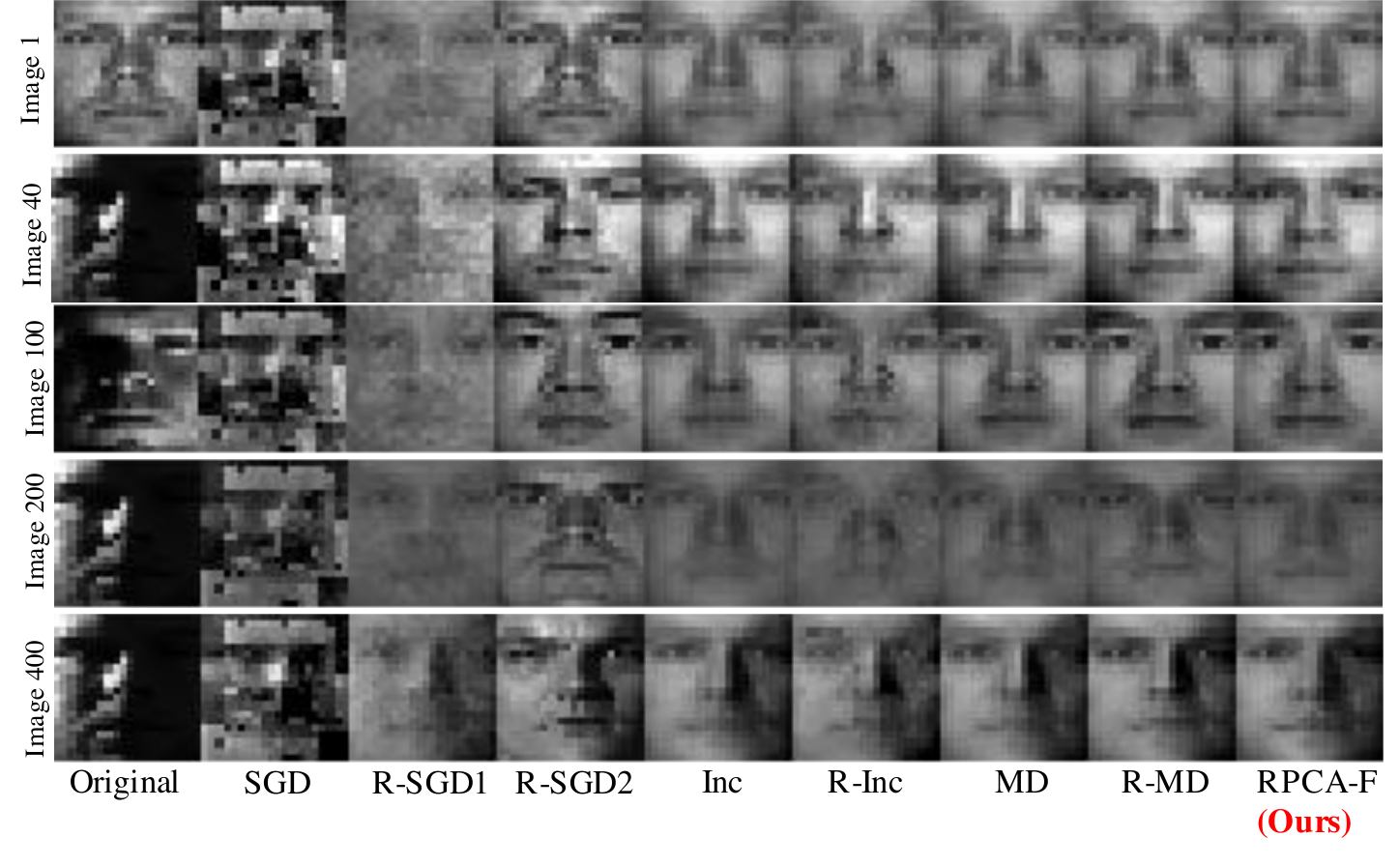}
    \caption{\small{Inliers and outliers detection. Face images captured in different lighting conditions are inliers. We project different faces to 9 dimensional subspaces found by different methods.}}\label{inlier}
\end{figure}
\begin{table}
\begin{center}
\begin{tabular}{|l|c|c|c|c|c|c|c|c|}
\hline
\;\;Metric used & SGD & R-SGD1 & R-SGD2 & Inc & R-Inc &MD & R-MD & RPCA-F\\
\hline
$\frac{\|P_L-P_{L^*}\|_F}{3\sqrt{2}}$&0.6985   & 0.8603   & 4.6607   & 0.7703  &  0.7214 &   0.6711  &  0.6679 &   0.7764\\
\hline
\end{tabular}
\end{center}
\caption{\small{Quantitative performance of different algorithms in inlier detection experiment. Except R-SGD2 all methods are highly competitive.}}\label{inlier_data}
\end{table}

\subsubsection{Processing astronomical data}\label{telescopic data}
In this experiment, we use the VIMOS Very Large Telescope (VIMOS-VLT) Deep Survey dataset \cite{vimos} to understand the evolution of the galaxies. We compare the first 4 eigenspectra obtained by our feasibility approach with those of the state-of-the-art methods, such as RE-PCA of \cite{vimos1}, online PCA, and robust online PCA of \cite{rspca}. Similar to Section \ref{inlier detection}, we split the dataset, $A$, into a 4-dimensional low-rank subspace $L$ and after we find $L$, we extract the orthogonal basis of $L$ and plot them. From Figure \ref{eigenspec}, visually, robust online PCA and our method are close relatively best fit to the ground truth RE-PCA of \cite{vimos1}. For details of the data, motivation, and experimental setup we refer the readers to \cite{vimos,vimos1,rspca}. 
\begin{figure}
    \centering
    \includegraphics[width =\textwidth]{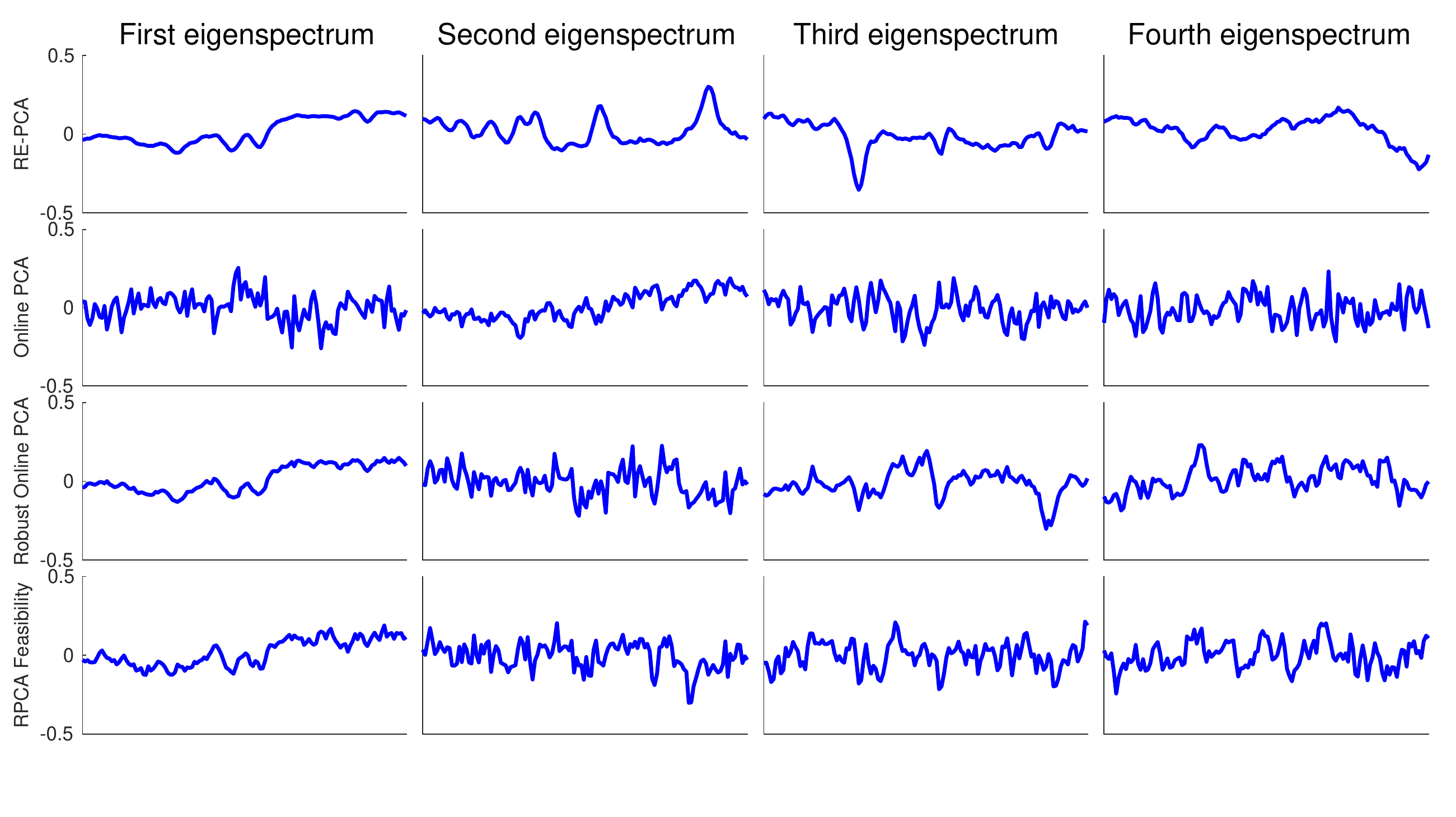}
    \caption{\small{The top four eigenspectra for the VVDS galaxies. The top three rows are the state-of-the art algorithms RE-PCA of \cite{vimos1}, online PCA, and robust online PCA of \cite{rspca}, respectively. The last row is our feasibility approach.}}\label{eigenspec}
\end{figure}

\subsection{Further experiments}

We show many of our experimental results in the Appendix. On synthetic data, we empirically validate the sensitivity of Algorithm \ref{rpca_algo} with respect to the initialization, the choices of $r$, and sparsity level $\alpha$; and the effect of the cardinality of $\Omega$ for Algorithm \ref{rmc_algo} (see  \ref{sec:init}, \ref{sec:init2}, and \ref{sec:init3}). 
\section{Conclusion}

In this paper, we propose a simplistic and novel approach to solve the classic RPCA and RMC problems. We consider an alternating projection algorithm based on the set feasibility approach to solve these problems in their crude form, without considering any further heuristics, such as loss functions, convex and surrogate constraints. Although we did not rigorously study convergence of our method theoretically; we investigated the convergence through numerical simulations on synthetic and real-world data and extensively compared with the current state-of-the-art methods. Our feasibility approach can open a new direction of potential research on online algorithms based on RPCA framework \cite{incpcp, grasta, gosus} that vastly used in video analysis, segmentation, subspace detection, and only a few to mention.

\bibliographystyle{plain}{
\small
\bibliography{egbib_updated}}

\appendix


\section{Additional Numerical Experiments \label{sec:app_exp}}
In this section we empirically study the  convergence of Algorithm~\ref{rpca_algo}. 

\subsection{Algorithm~\ref{rpca_algo}: Sensitivity to  initialization \label{sec:init}}

First, we examine how the starting point influences the convergence. We construct $A\in \R^{100\times 100}$ and perform 50 runs of Algorithm~\ref{rpca_algo} for various values of $\alpha$ and $r$. In all cases, we set \[A\eqdef \mathcal{T}_{\alpha}(S')+H_r(L')\] for $S',L'$ with independent random entries from $\cN(0,1)$, and run Algorithm~\ref{rpca_algo} with (correct) parameters $\alpha, r$. 
Figure~\ref{fig:initial} shows the worst, the best, and the median case for each iteration, and illustrates that the convergence (and convergence speed) of the algorithm for the vast majority of cases is independent of the initial point. Moreover, when both rank and sparsity are not too big (sparsity level is 10\% or less and rank is 15\% or less), we observe very fast convergence.

\begin{figure}[h]
\centering
\begin{minipage}{0.25\textwidth}
  \centering
\includegraphics[width =  \textwidth ]{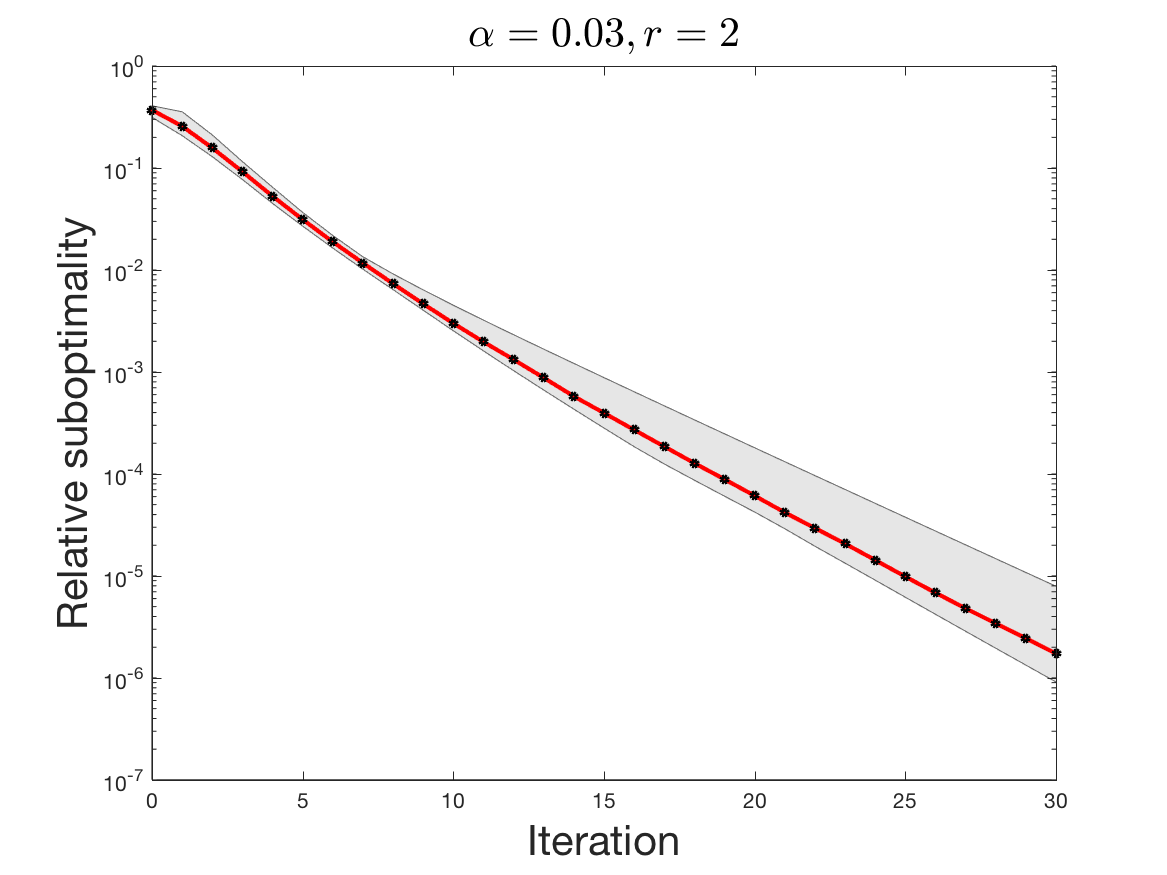}
\end{minipage}%
\begin{minipage}{0.25\textwidth}
  \centering
\includegraphics[width =  \textwidth ]{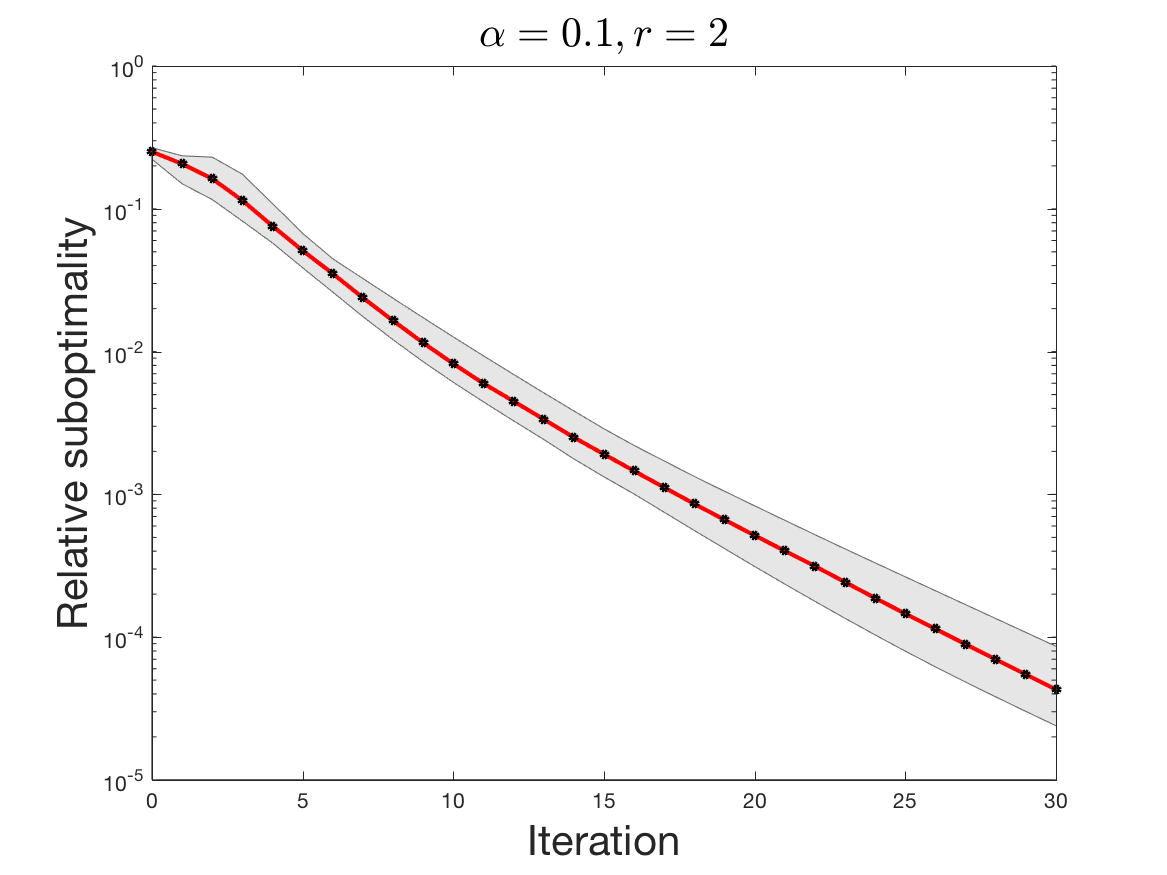}
\end{minipage}%
\begin{minipage}{0.25\textwidth}
  \centering
\includegraphics[width =  \textwidth ]{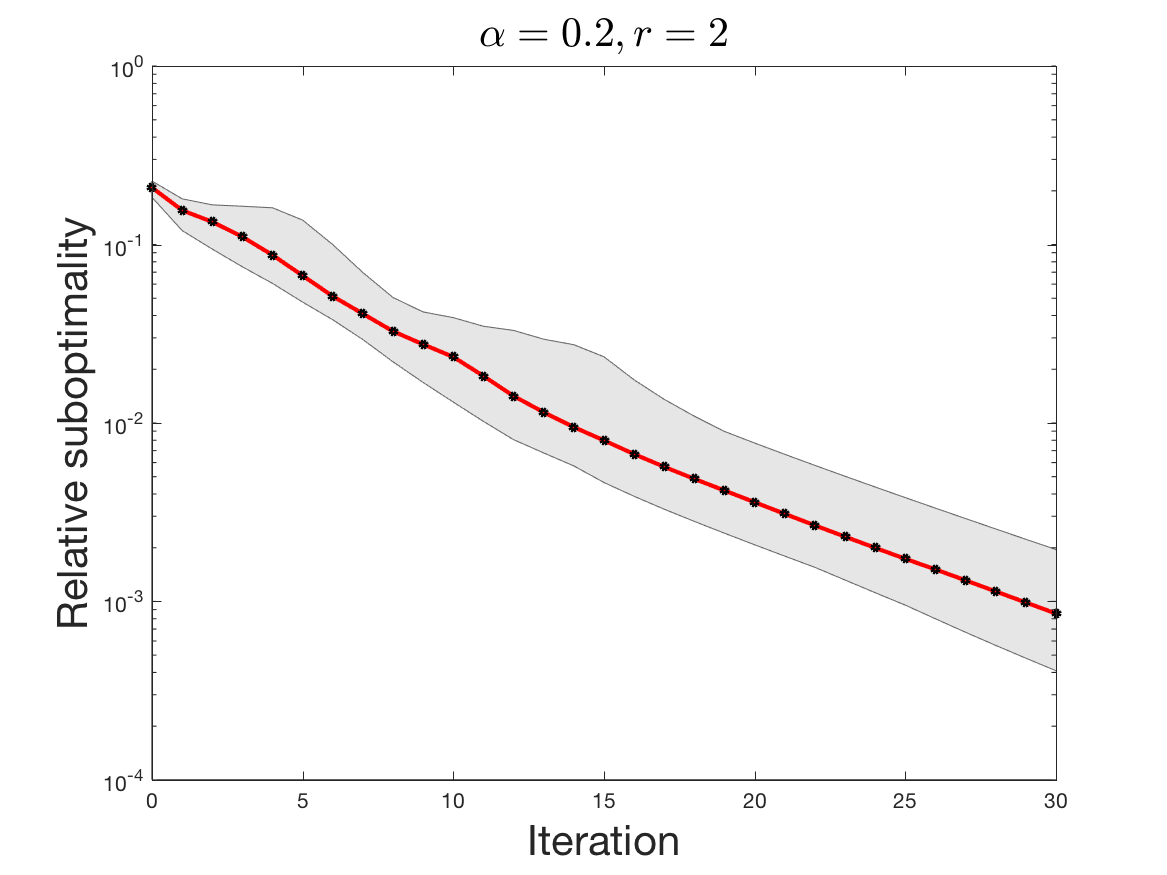}
\end{minipage}%
\begin{minipage}{0.25\textwidth}
  \centering
\includegraphics[width =  \textwidth ]{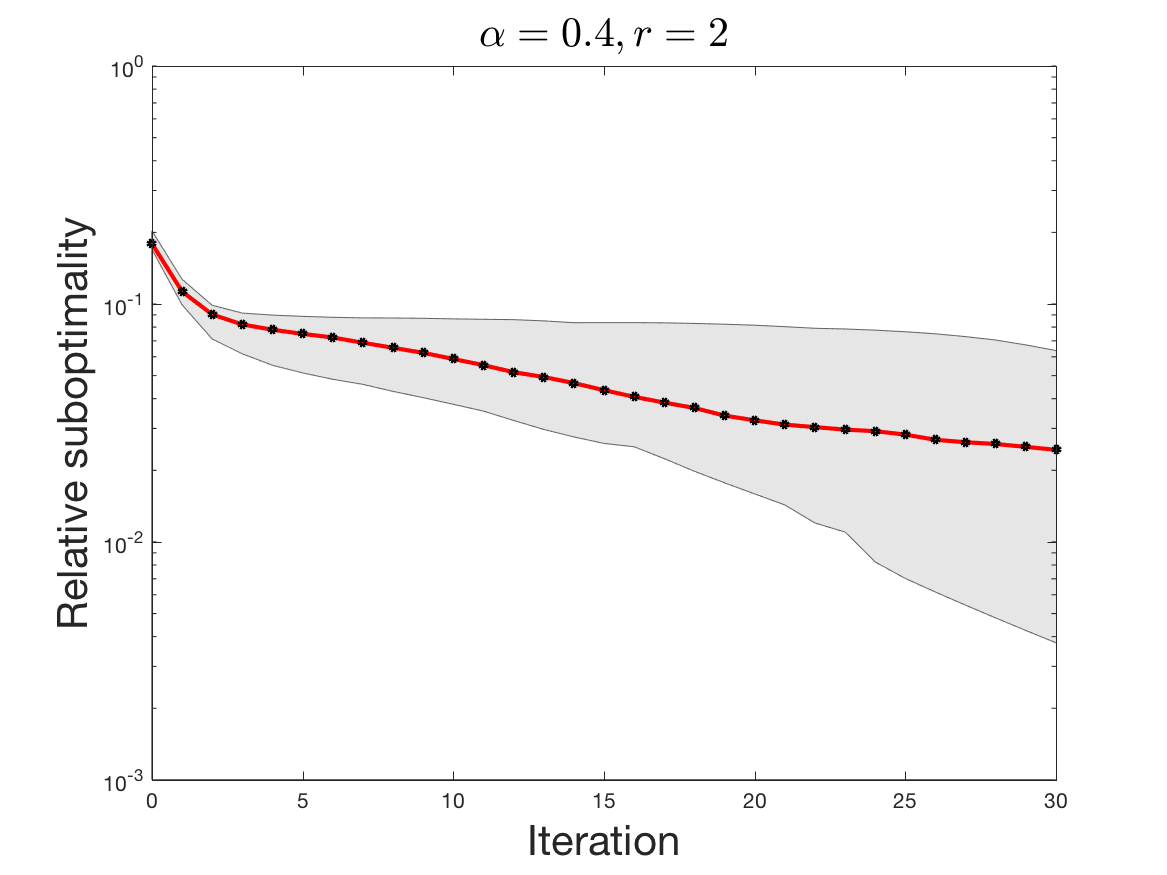}
\end{minipage}%
\\
\begin{minipage}{0.25\textwidth}
  \centering
\includegraphics[width =  \textwidth ]{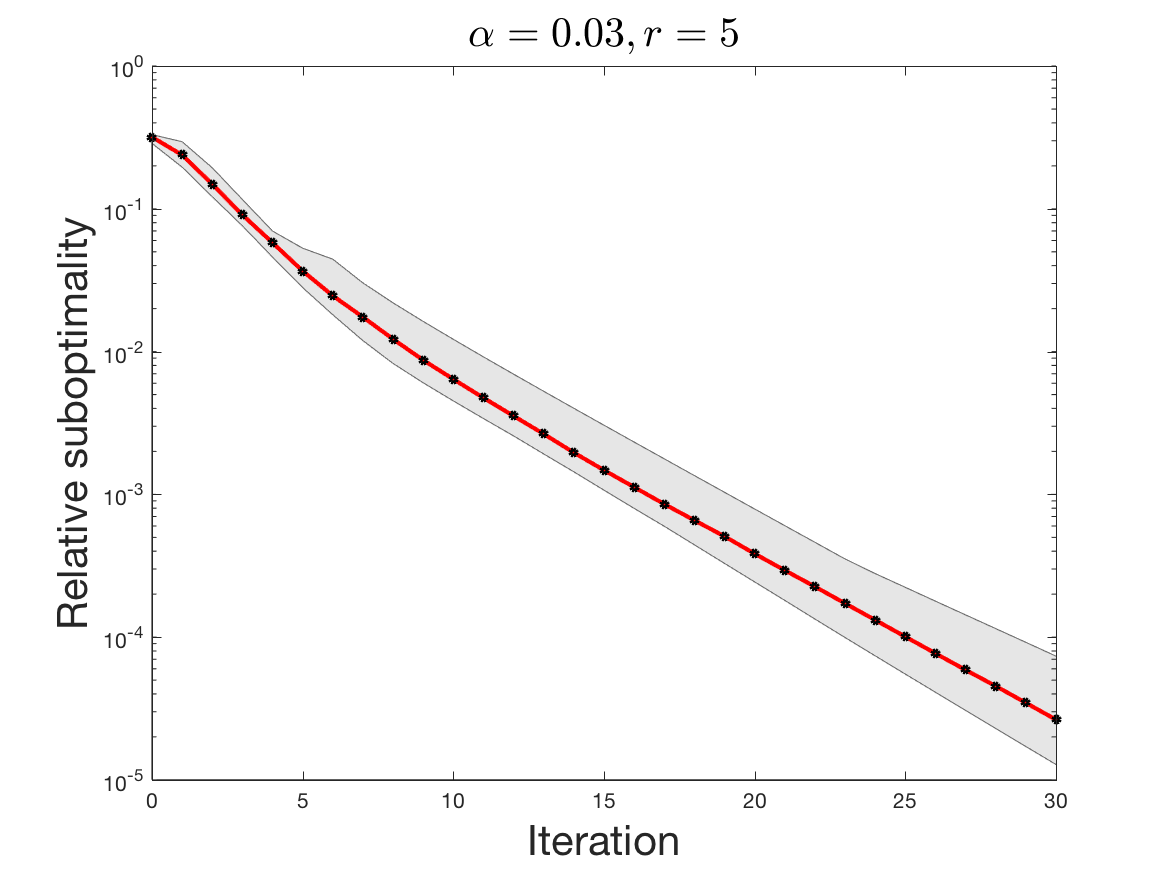}
\end{minipage}%
\begin{minipage}{0.25\textwidth}
  \centering
\includegraphics[width =  \textwidth ]{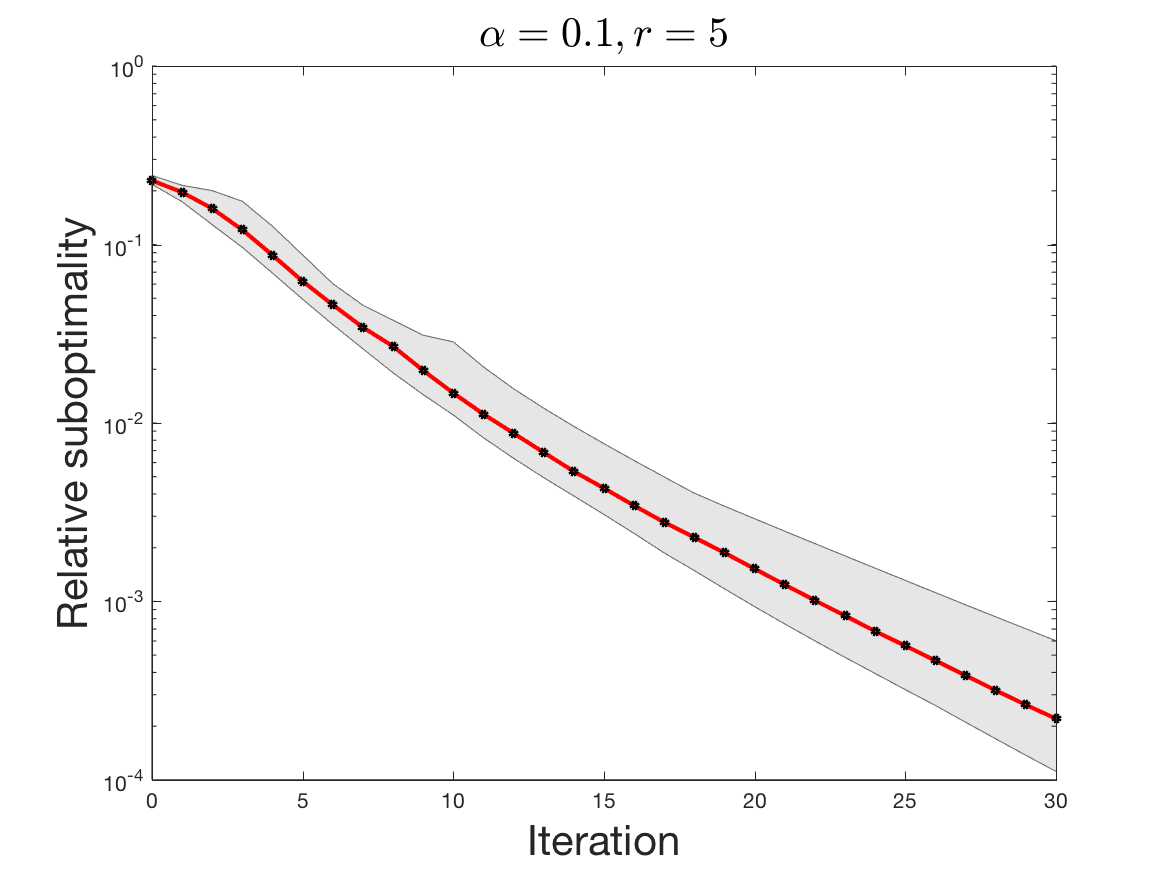}
\end{minipage}%
\begin{minipage}{0.25\textwidth}
  \centering
\includegraphics[width =  \textwidth ]{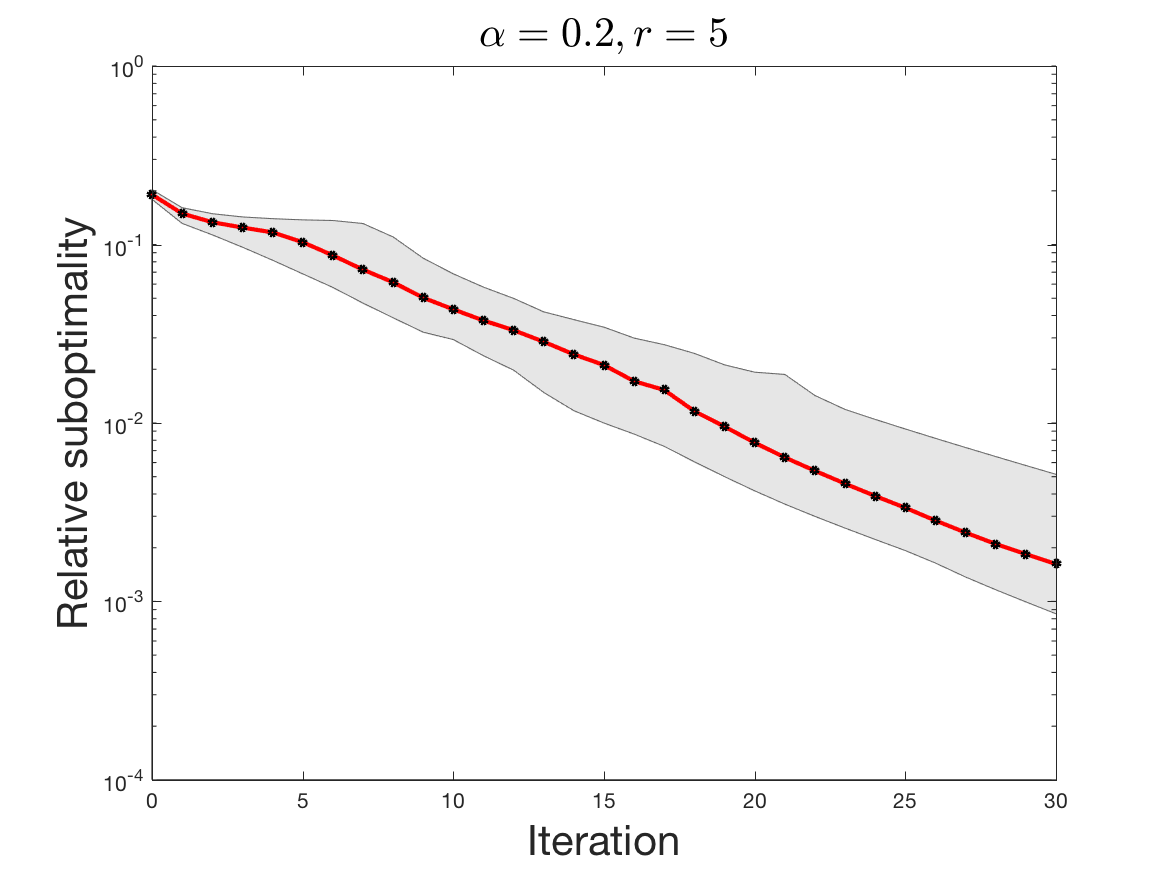}
\end{minipage}%
\begin{minipage}{0.25\textwidth}
  \centering
\includegraphics[width =  \textwidth ]{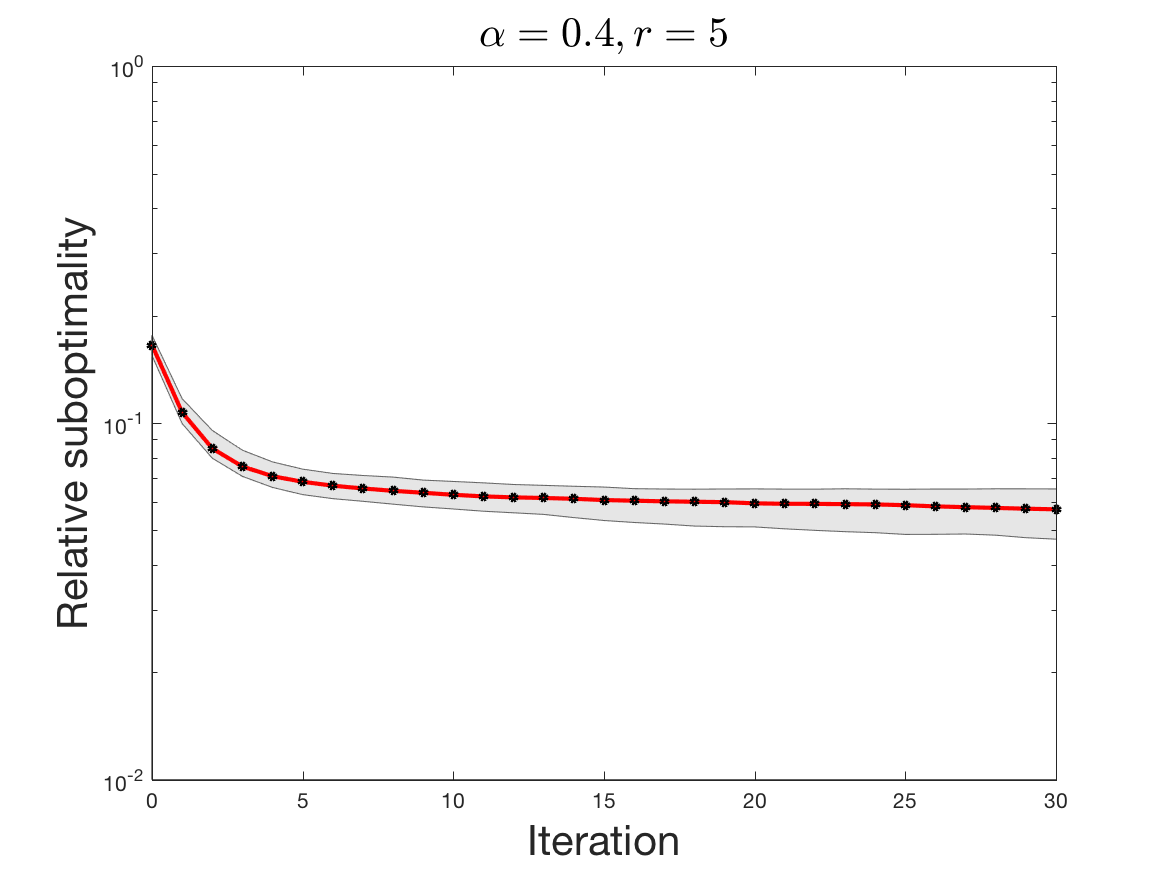}
\end{minipage}%
\\
\begin{minipage}{0.25\textwidth}
  \centering
\includegraphics[width =  \textwidth ]{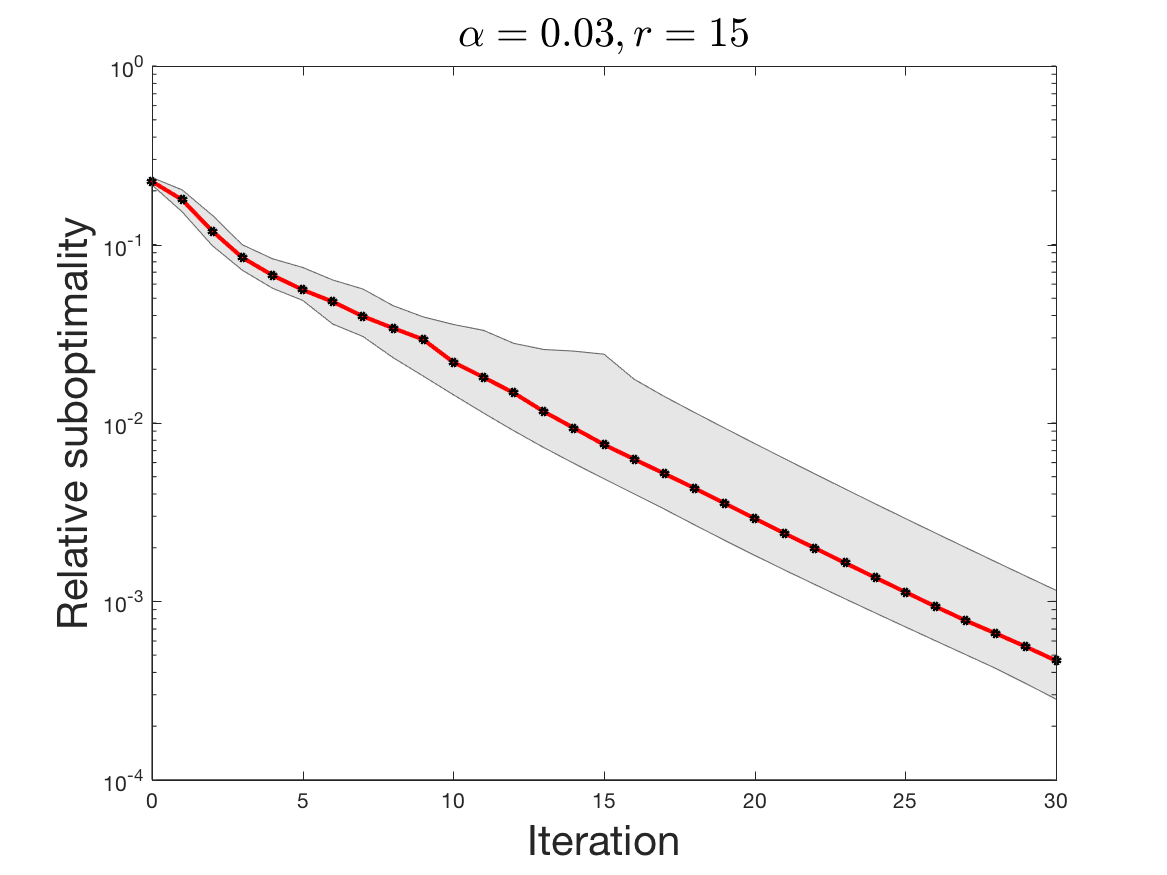}
\end{minipage}%
\begin{minipage}{0.25\textwidth}
  \centering
\includegraphics[width =  \textwidth ]{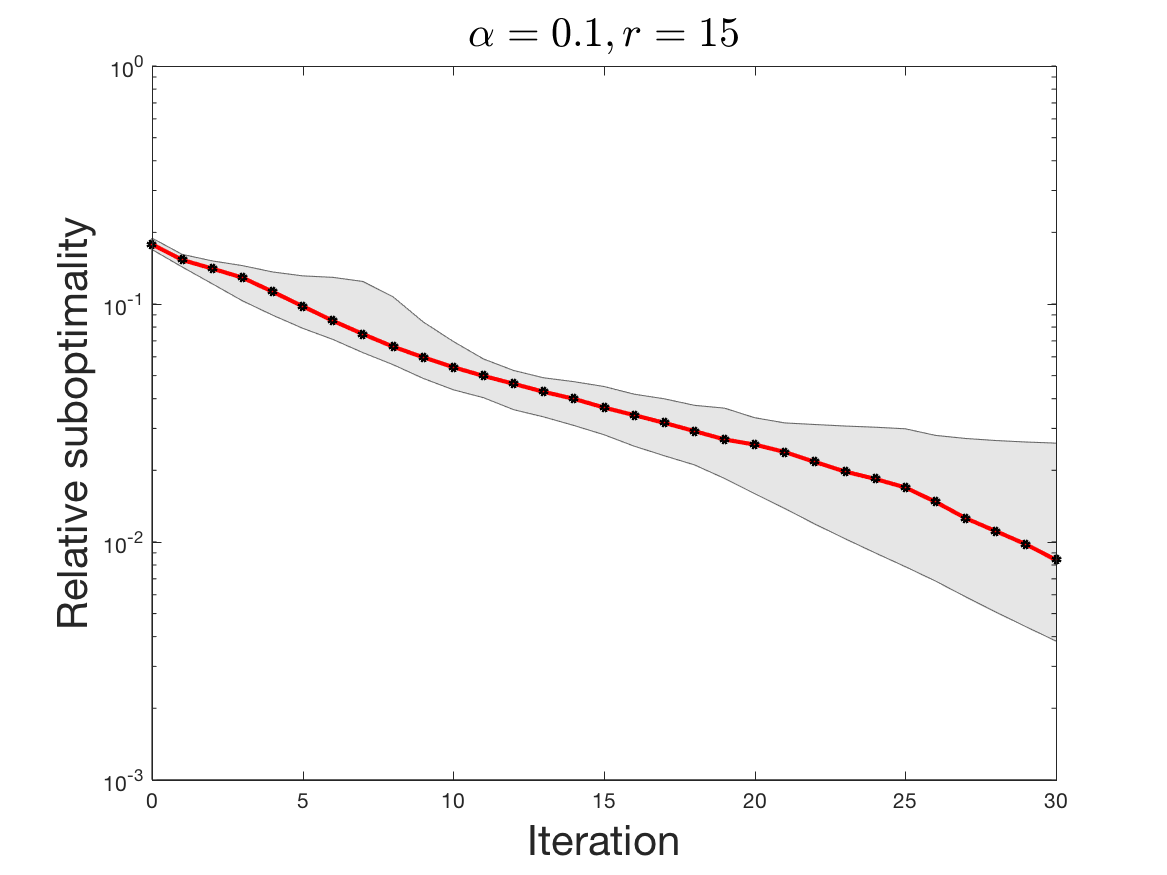}
\end{minipage}%
\begin{minipage}{0.25\textwidth}
  \centering
\includegraphics[width =  \textwidth ]{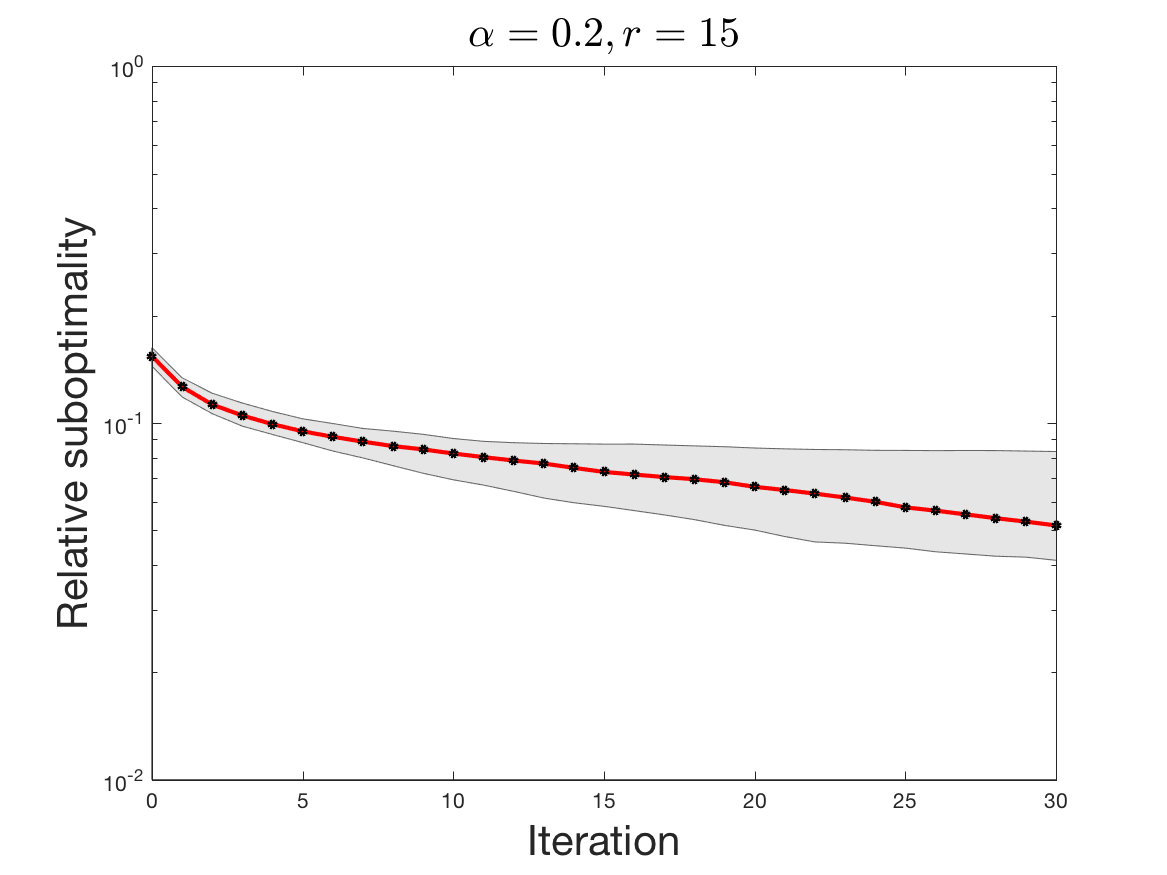}
\end{minipage}%
\begin{minipage}{0.25\textwidth}
  \centering
\includegraphics[width =  \textwidth ]{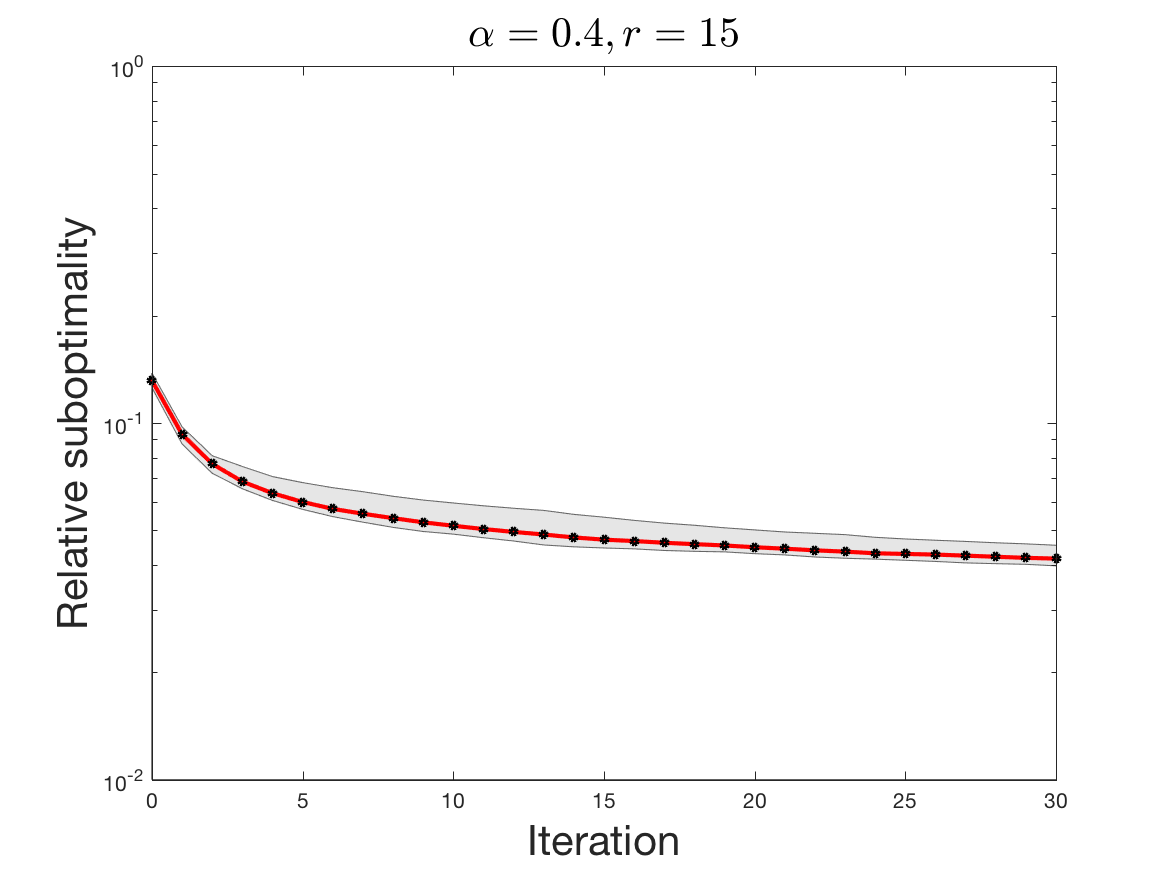}
\end{minipage}%
\\
\begin{minipage}{0.25\textwidth}
  \centering
\includegraphics[width =  \textwidth ]{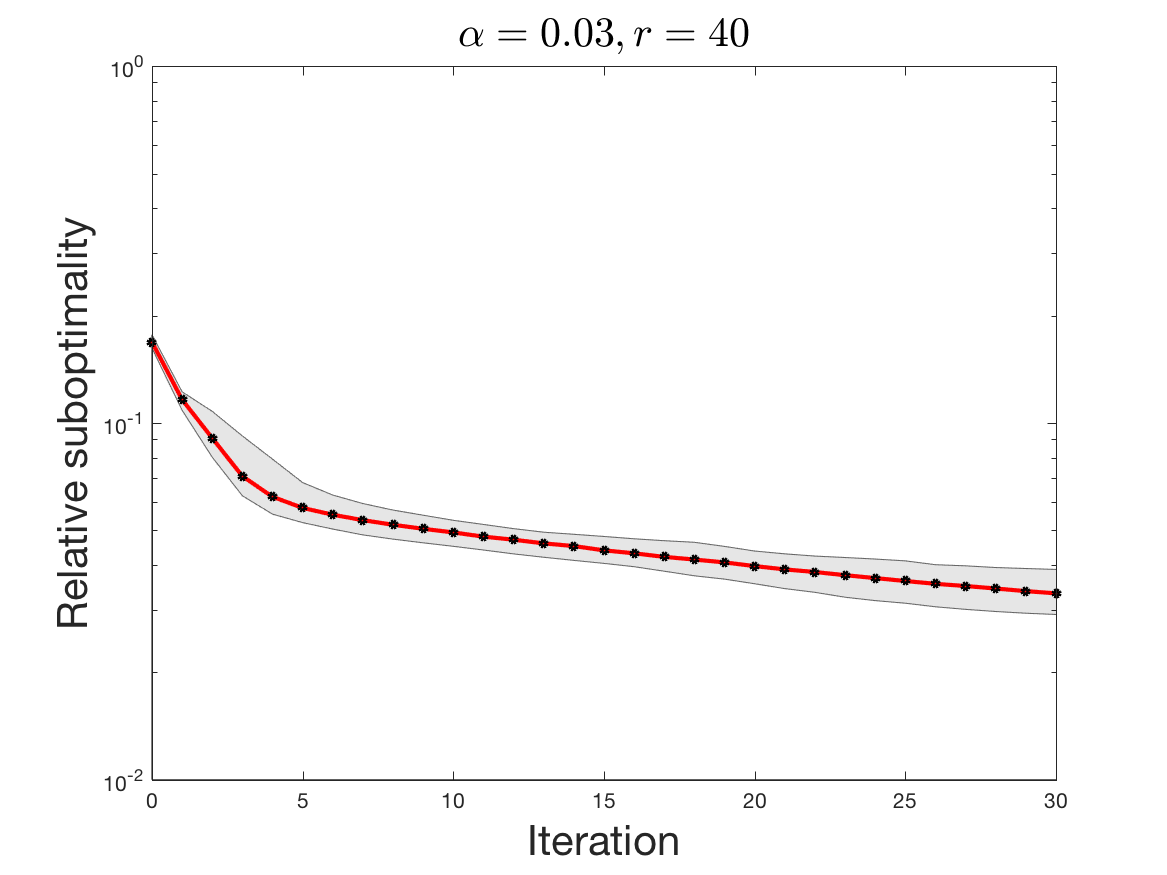}
\end{minipage}%
\begin{minipage}{0.25\textwidth}
  \centering
\includegraphics[width =  \textwidth ]{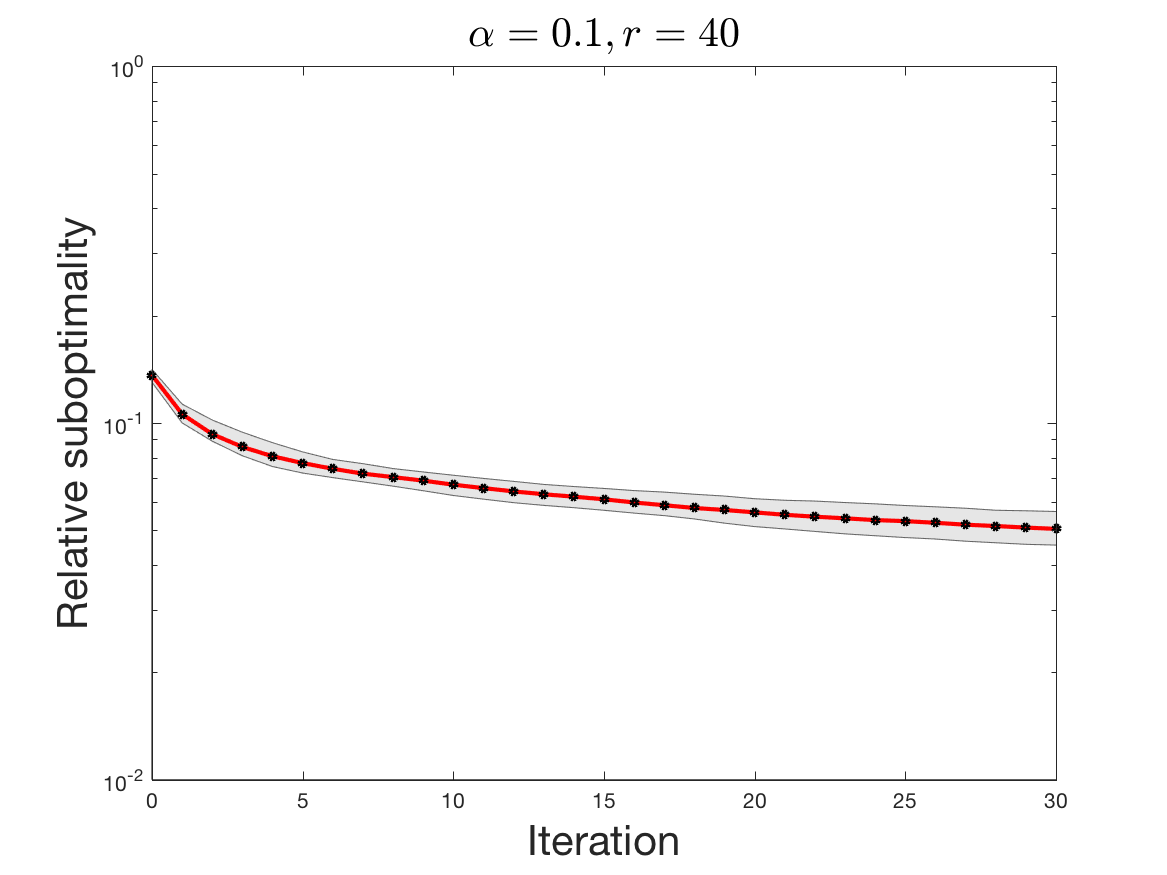}
\end{minipage}%
\begin{minipage}{0.25\textwidth}
  \centering
\includegraphics[width =  \textwidth ]{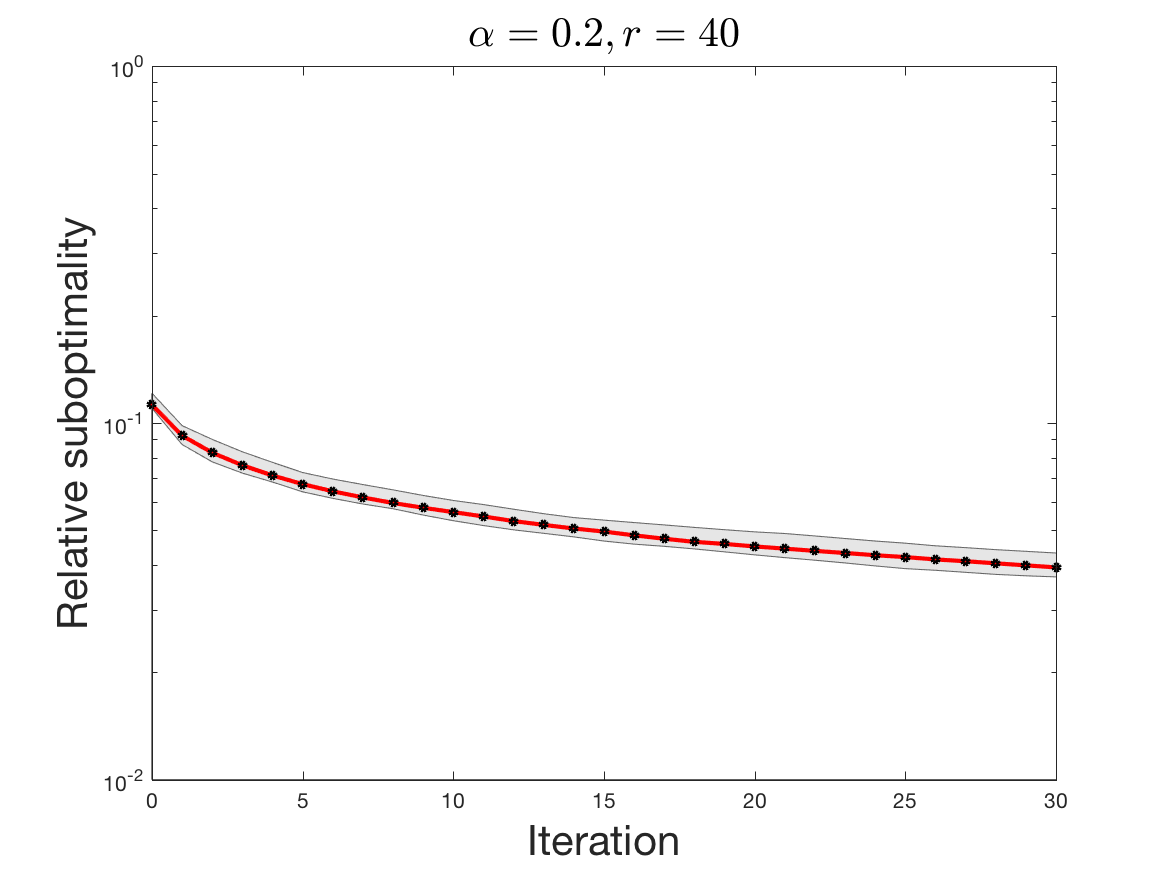}
\end{minipage}%
\begin{minipage}{0.25\textwidth}
  \centering
\includegraphics[width =  \textwidth ]{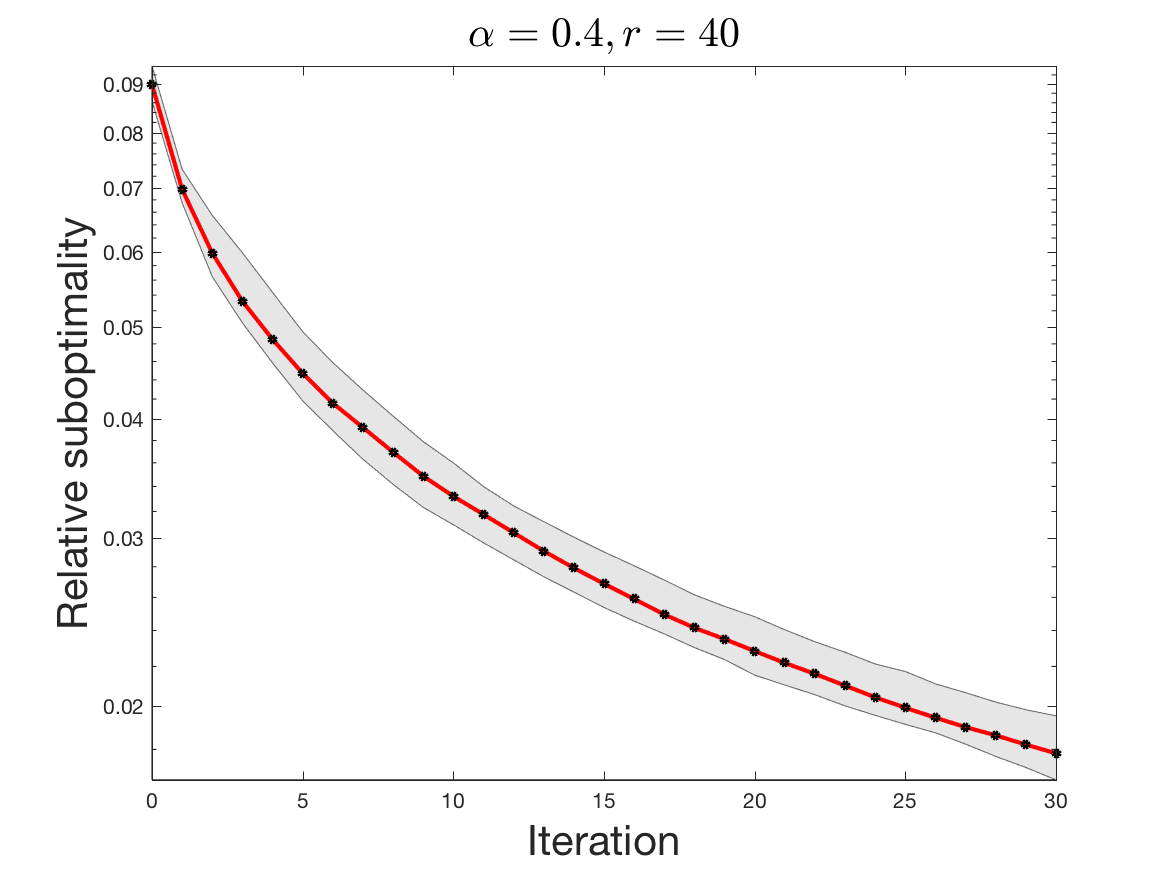}
\end{minipage}%
\caption{Sensivity of Algorithm~\ref{rpca_algo} to initialization. The best, the worst, and the median case are plotted for each iteration.}\label{fig:initial}
\end{figure}

\subsection{Algorithm~\ref{rmc_algo}: The effect of the number of observable entries on convergence}\label{sec:init2}

In this section, we study convergence properties of Algorithm~\ref{rmc_algo}. For different choices of $\alpha$ and $r$, Figure~\ref{fig:omega} shows how fast does Algorithm~\ref{rmc_algo} converge to the optimum. We observe  extremely fast convergence for both small ($<0.1$) and large ($\approx 1$) fraction of observable entries. However, for medium fractions of observable entries, Algorithm~\ref{rmc_algo} seems to often do not converge. This is an interesting phenomenon that could be studied more deeply in future research. However, for example, as Figure~\ref{FP_rmc} shows, we demonstrate that Algorithm~\ref{rmc_algo} still outperforms the other methods even for the critical medium sized $\Omega$.
\begin{figure}
\centering
\begin{minipage}{0.25\textwidth}
  \centering
\includegraphics[width =  \textwidth ]{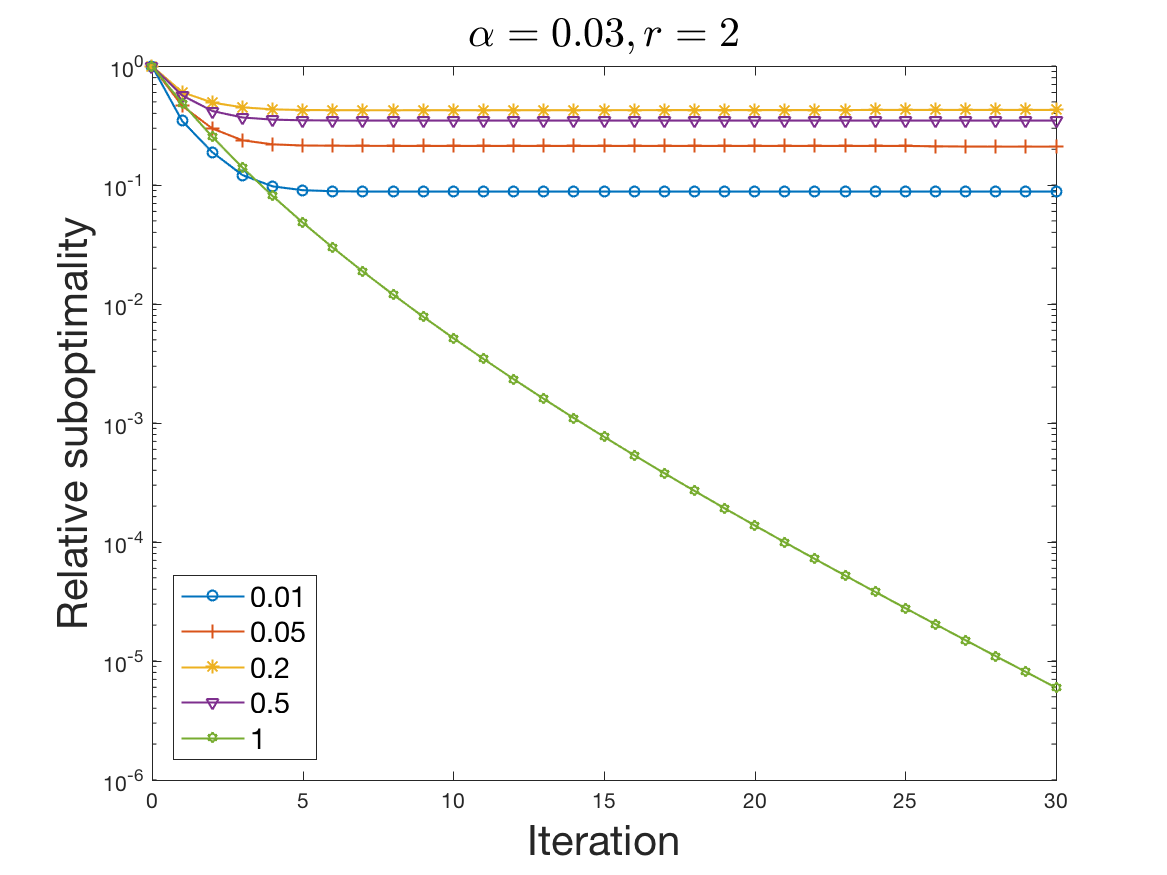}
\end{minipage}%
\begin{minipage}{0.25\textwidth}
  \centering
\includegraphics[width =  \textwidth ]{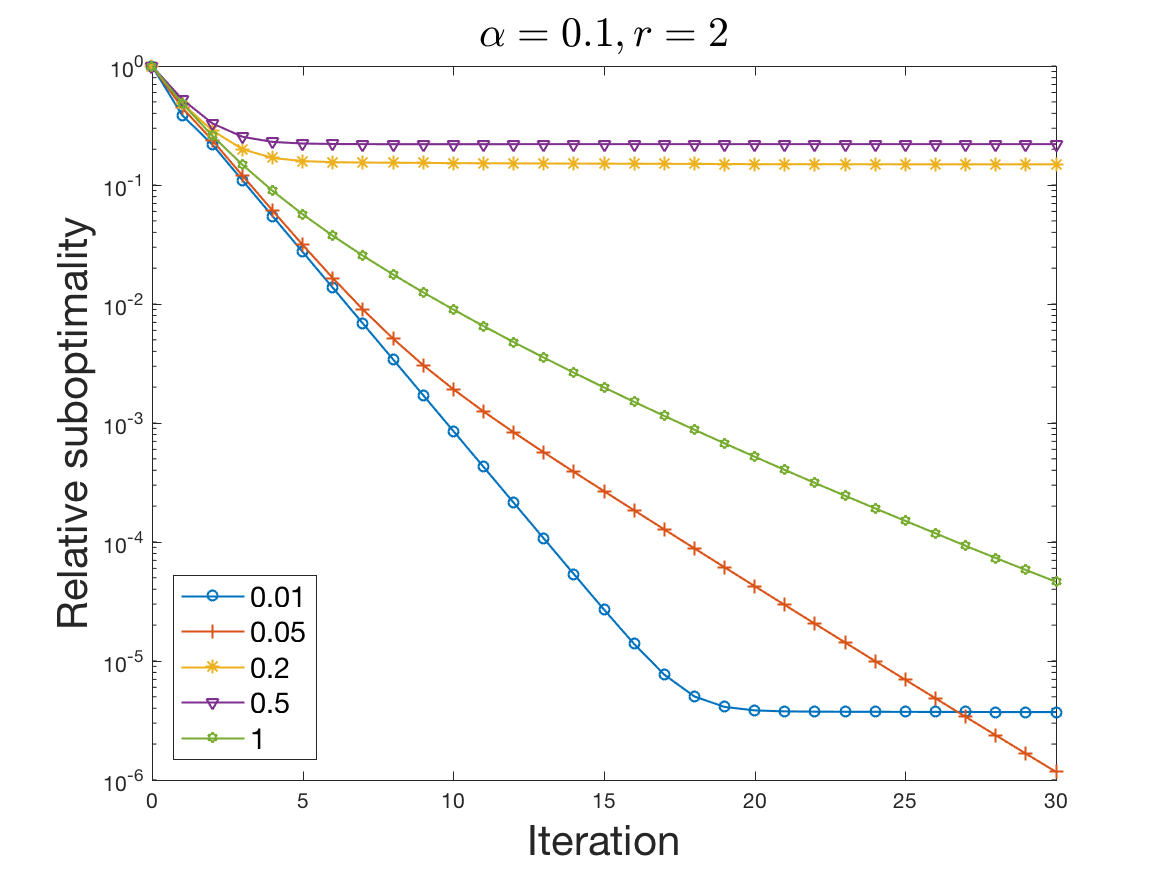}
\end{minipage}%
\begin{minipage}{0.25\textwidth}
  \centering
\includegraphics[width =  \textwidth ]{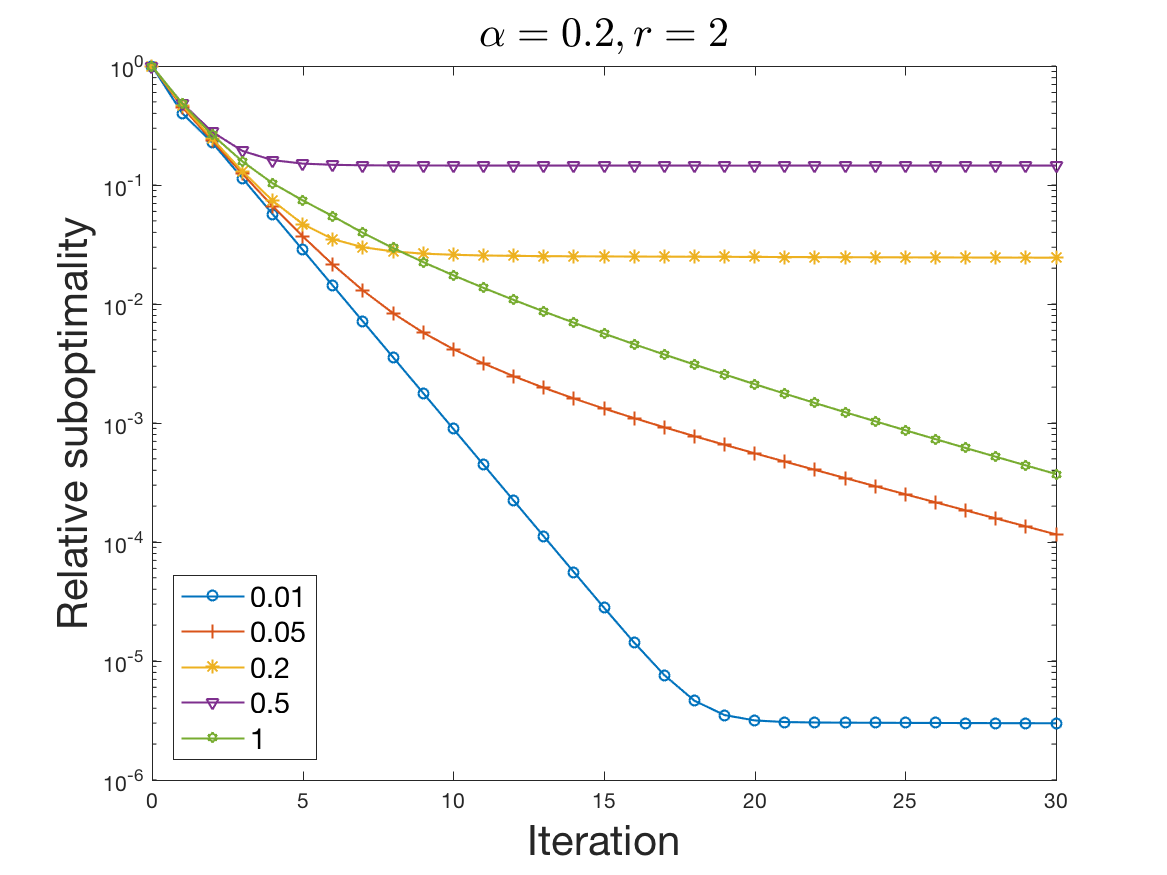}
\end{minipage}%
\begin{minipage}{0.25\textwidth}
  \centering
\includegraphics[width =  \textwidth ]{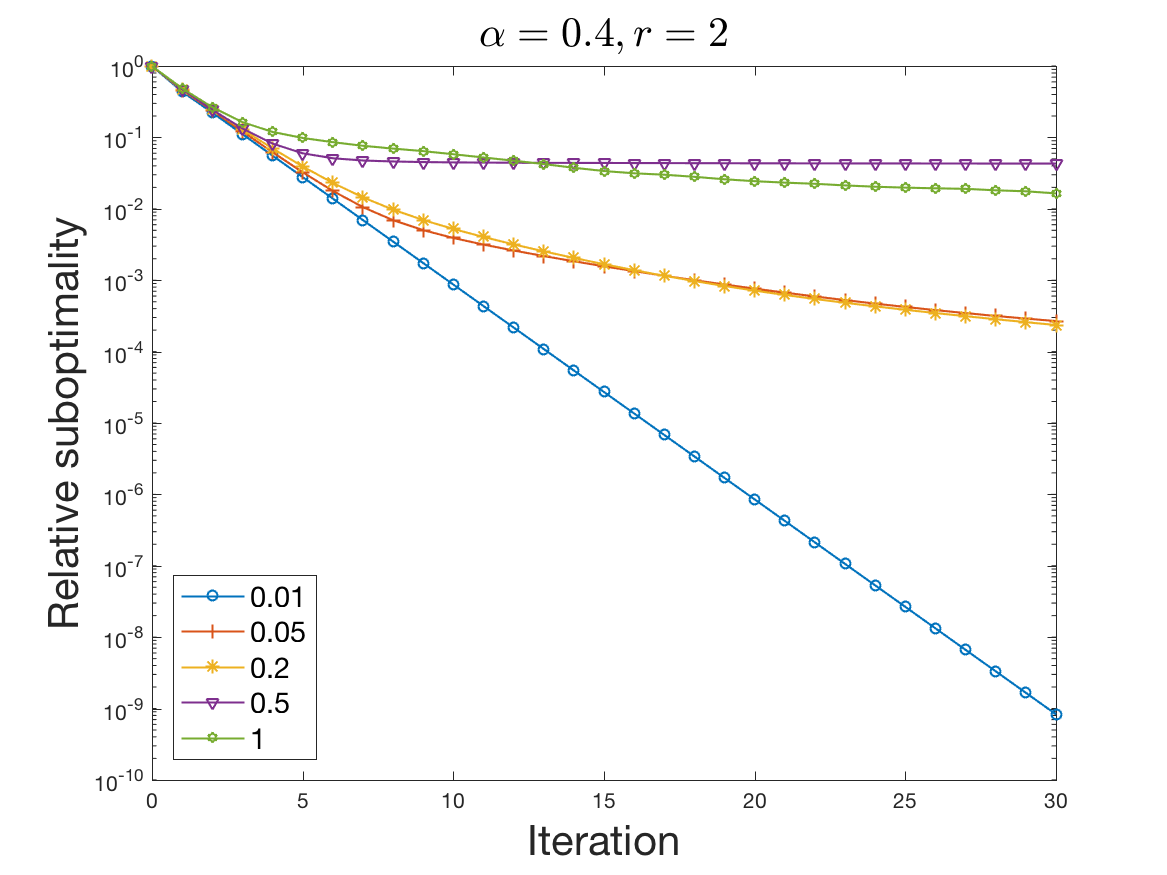}
\end{minipage}%
\\
\begin{minipage}{0.25\textwidth}
  \centering
\includegraphics[width =  \textwidth ]{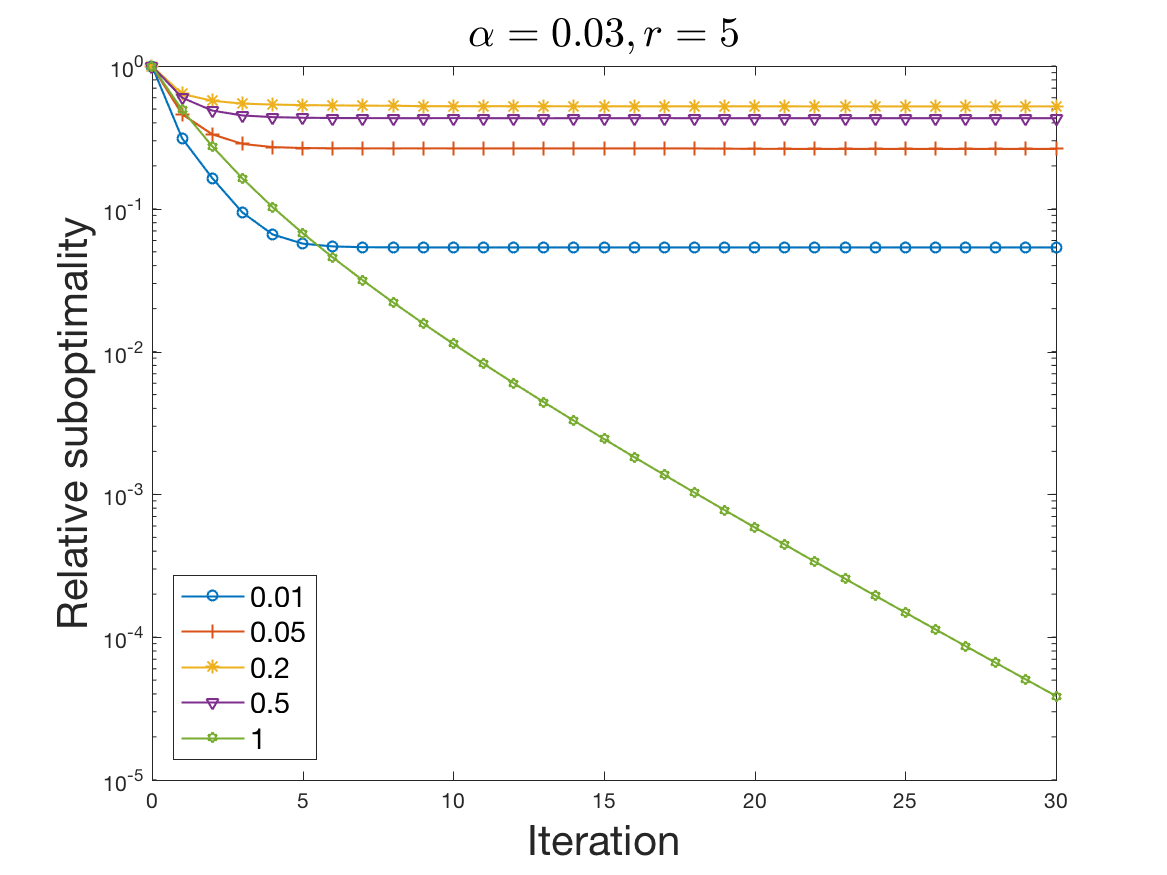}
\end{minipage}%
\begin{minipage}{0.25\textwidth}
  \centering
\includegraphics[width =  \textwidth ]{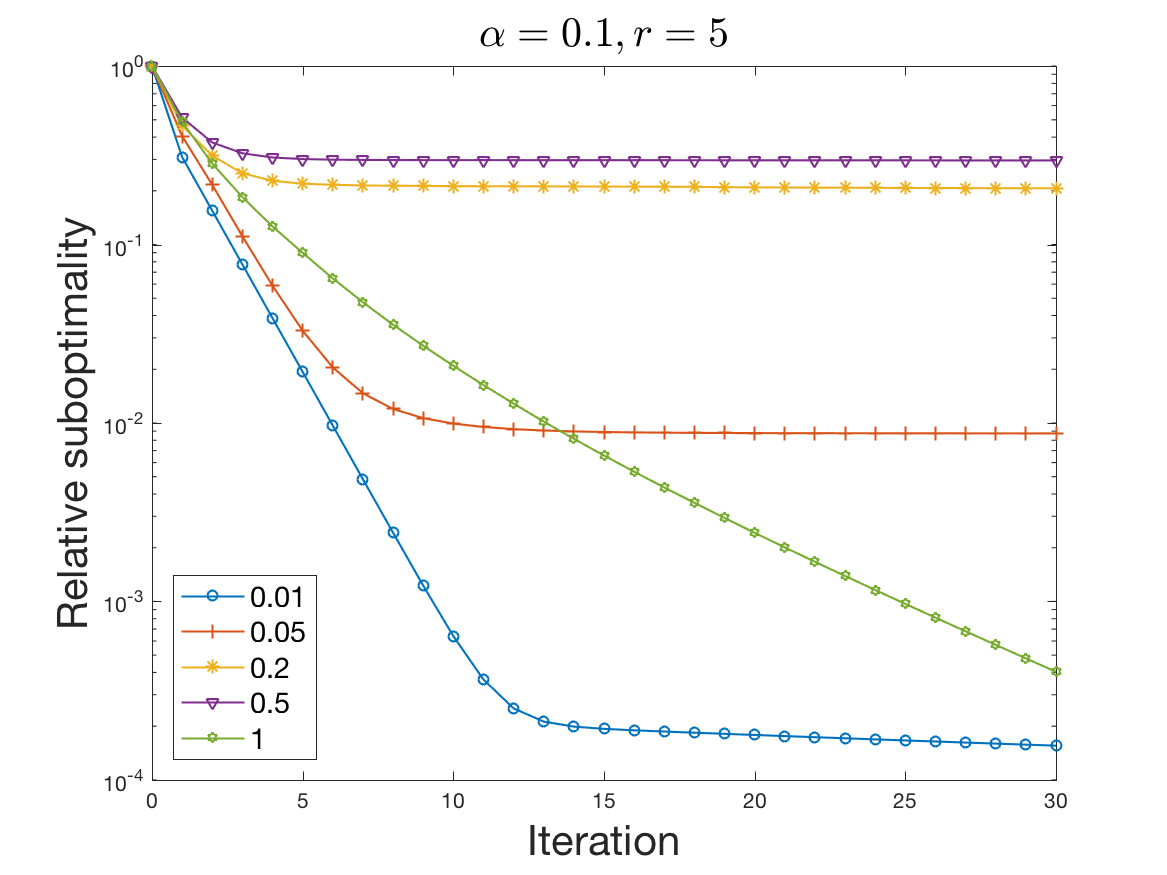}
\end{minipage}%
\begin{minipage}{0.25\textwidth}
  \centering
\includegraphics[width =  \textwidth ]{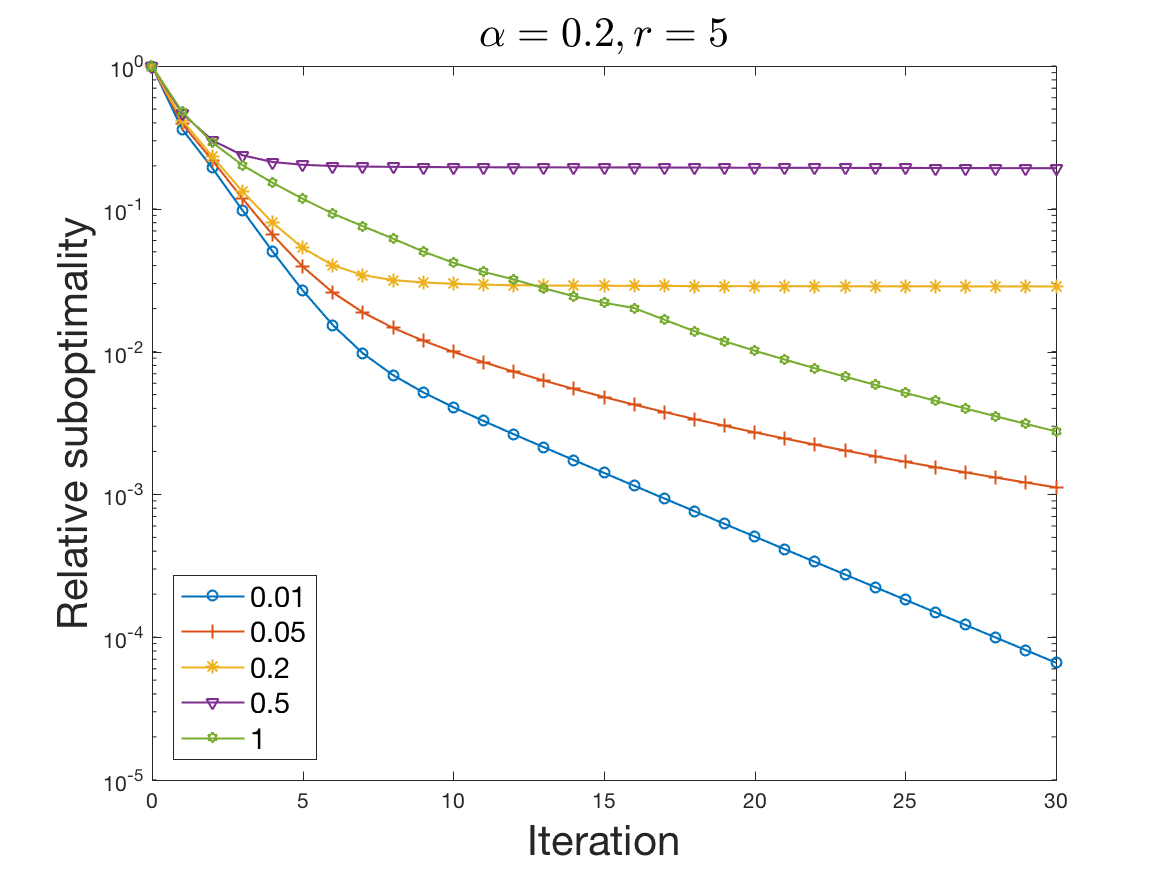}
\end{minipage}%
\begin{minipage}{0.25\textwidth}
  \centering
\includegraphics[width =  \textwidth ]{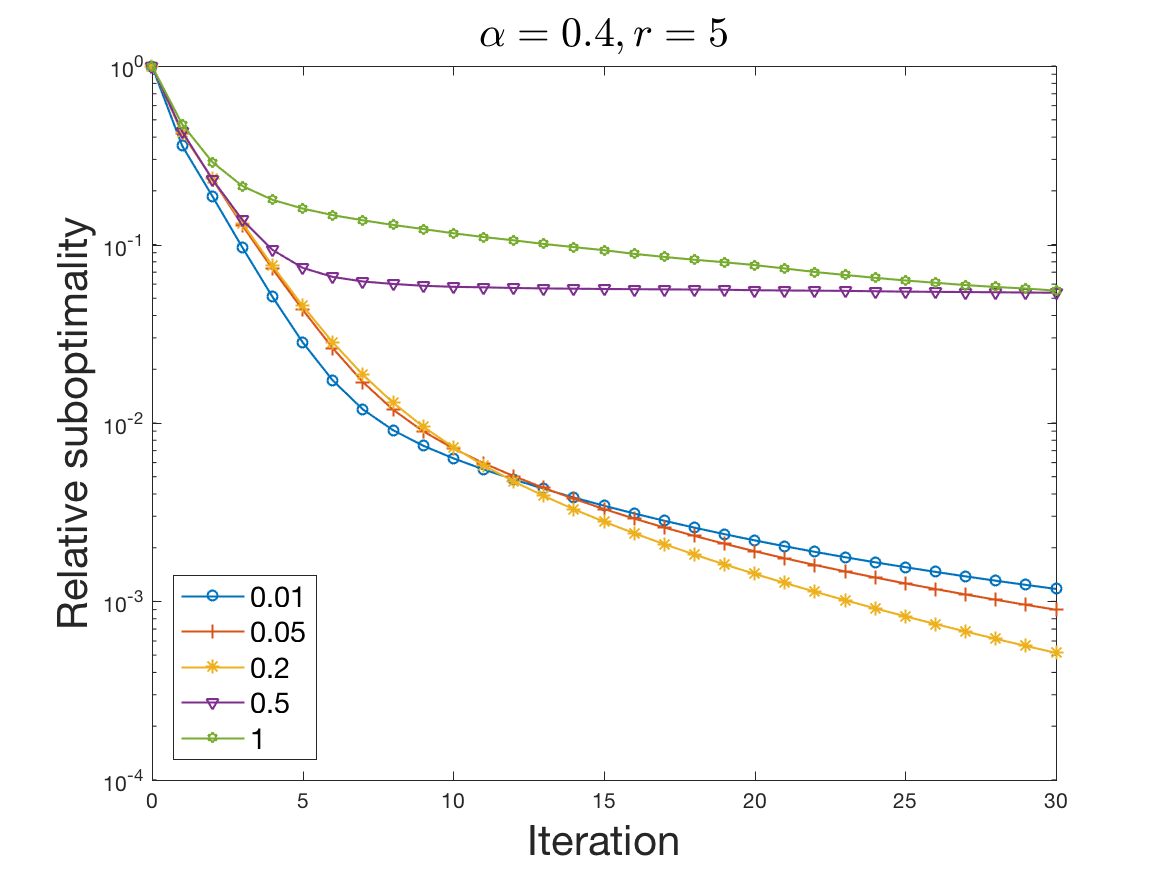}
\end{minipage}%
\\
\begin{minipage}{0.25\textwidth}
  \centering
\includegraphics[width =  \textwidth ]{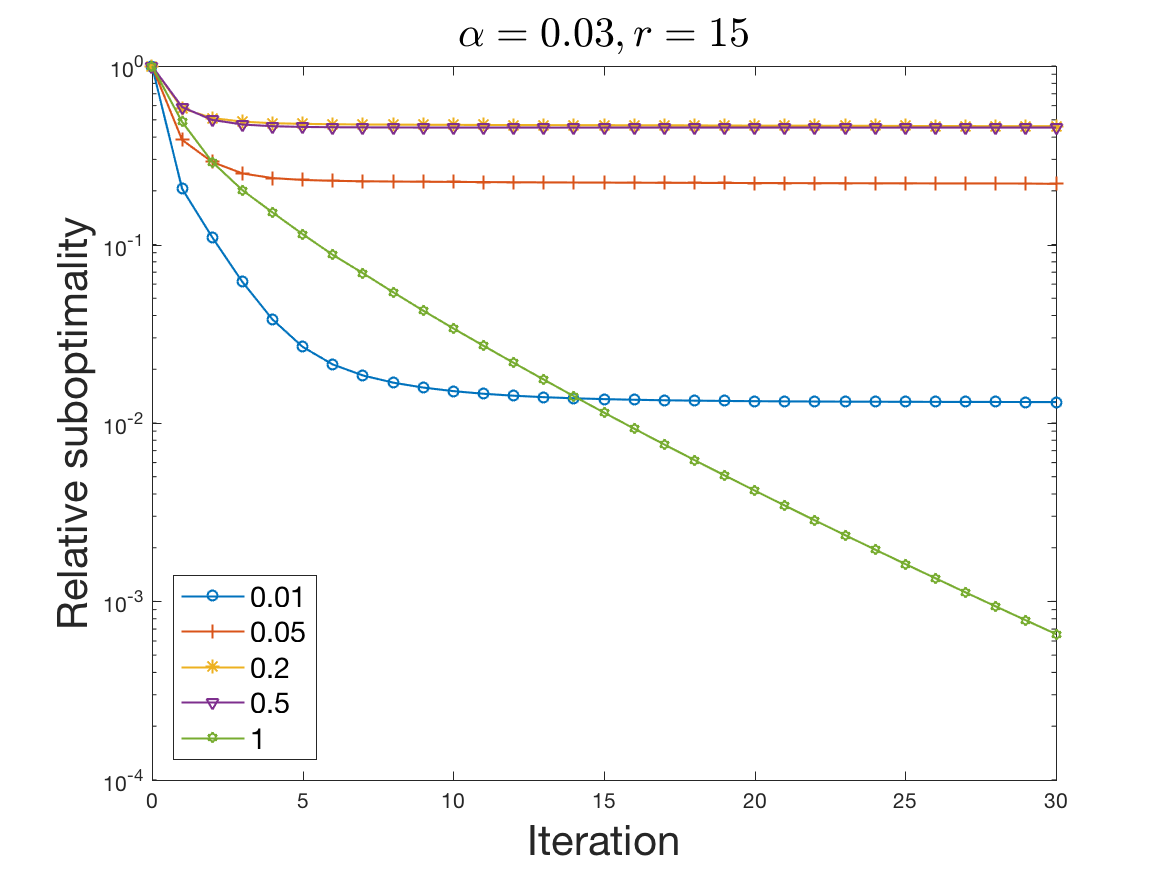}
\end{minipage}%
\begin{minipage}{0.25\textwidth}
  \centering
\includegraphics[width =  \textwidth ]{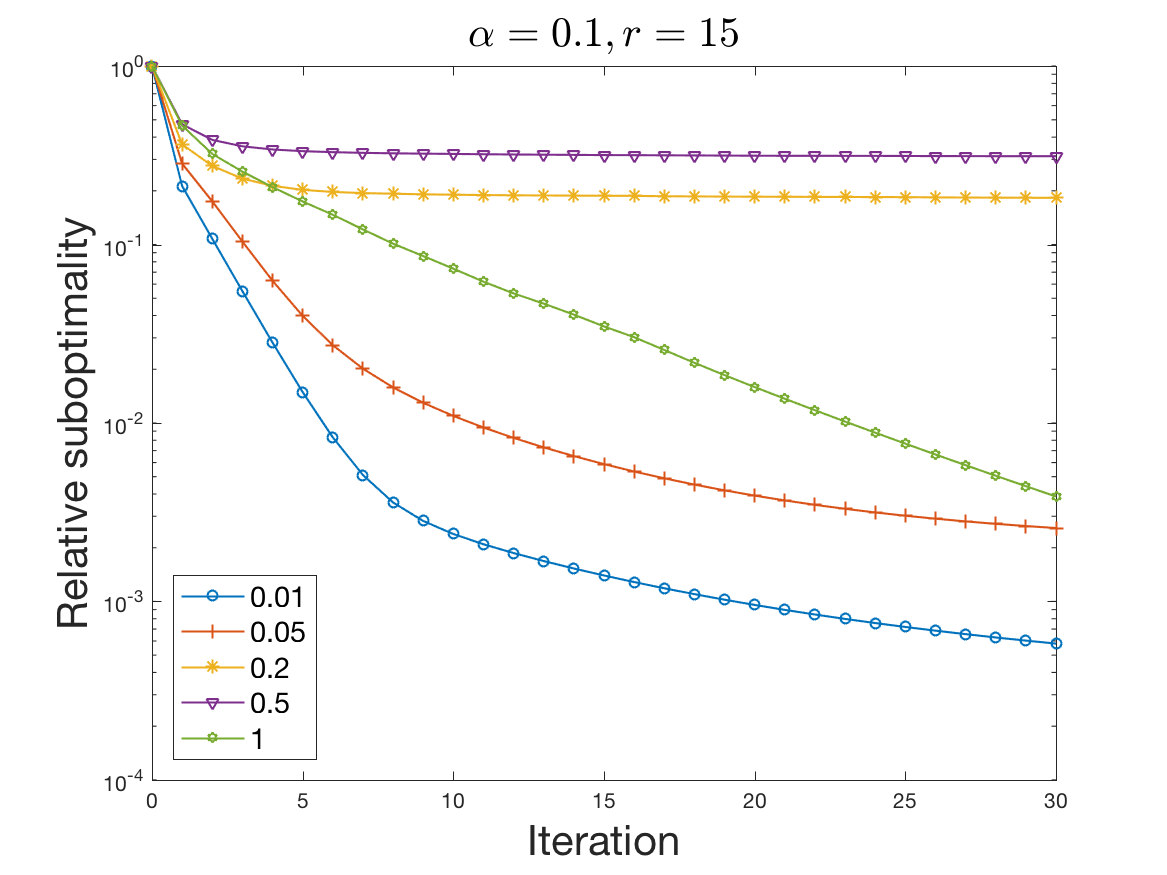}
\end{minipage}%
\begin{minipage}{0.25\textwidth}
  \centering
\includegraphics[width =  \textwidth ]{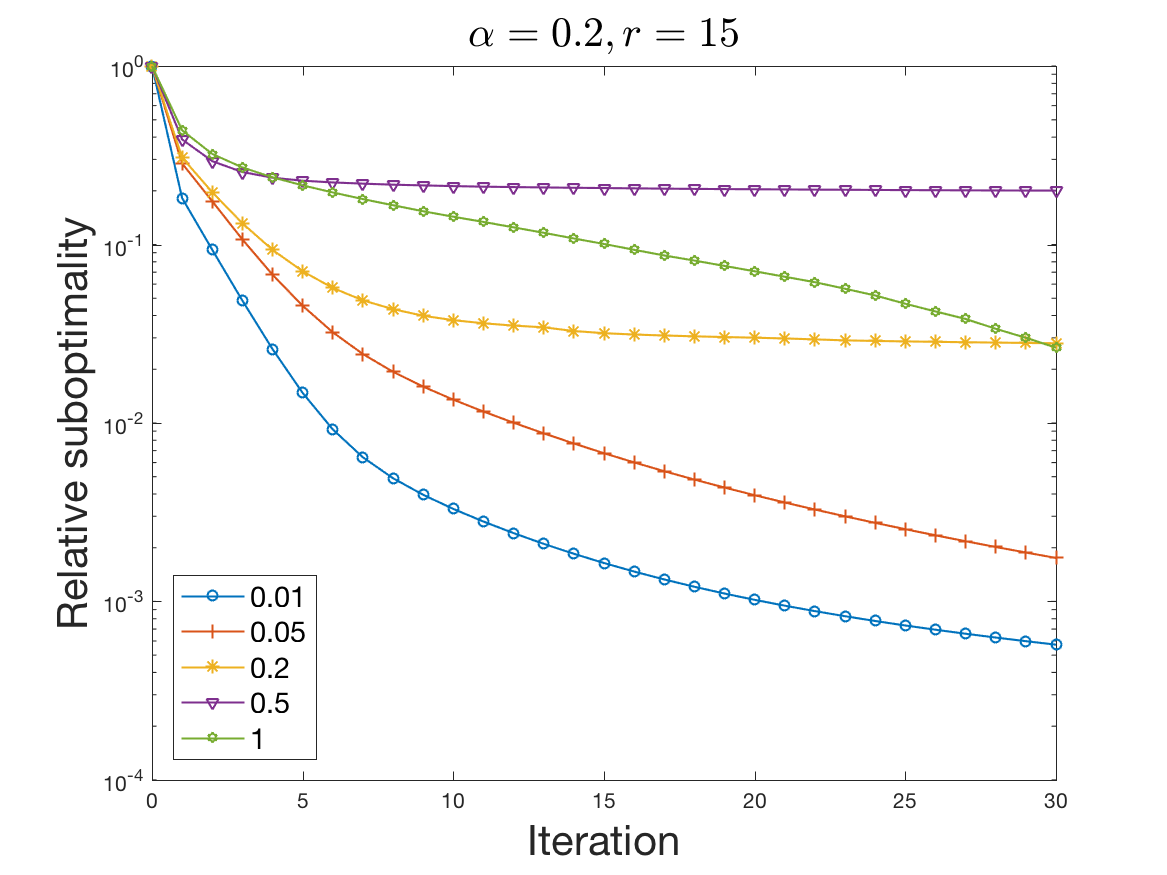}
\end{minipage}%
\begin{minipage}{0.25\textwidth}
  \centering
\includegraphics[width =  \textwidth ]{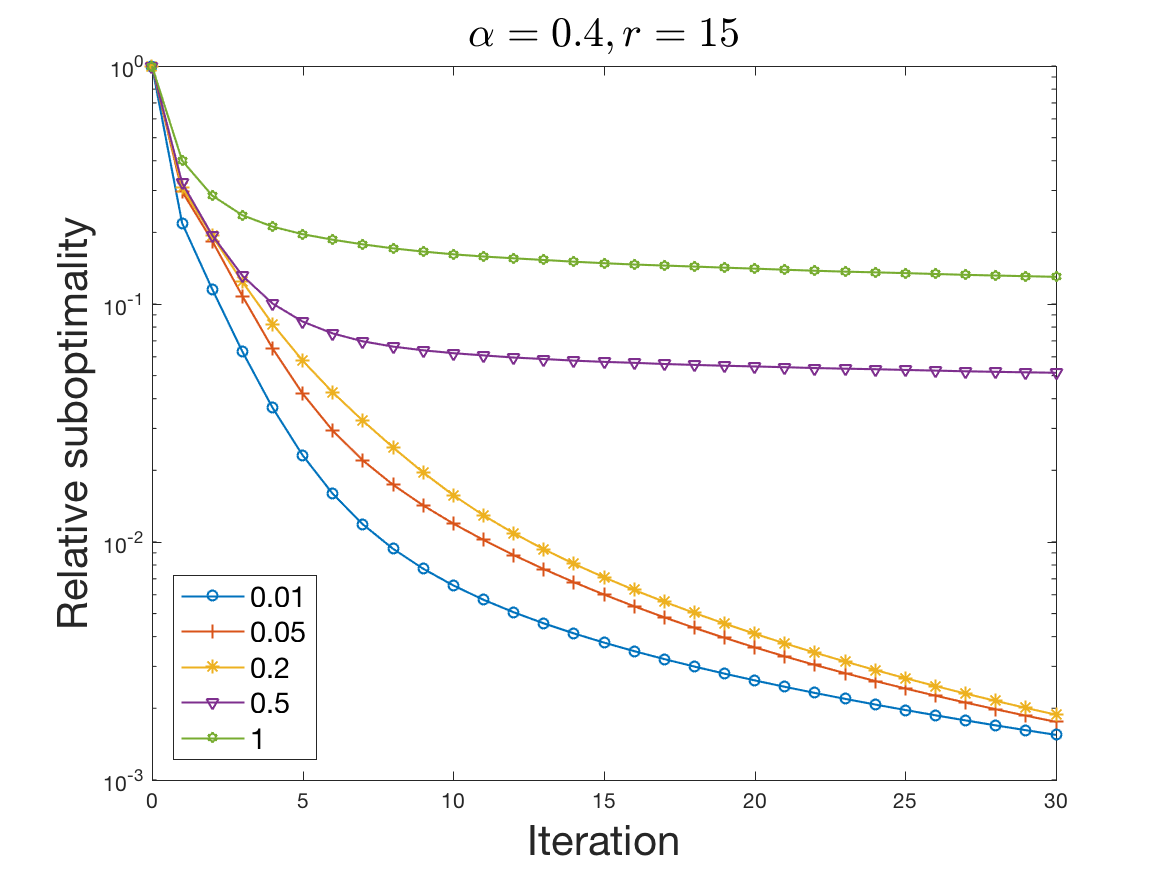}
\end{minipage}%
\\
\begin{minipage}{0.25\textwidth}
  \centering
\includegraphics[width =  \textwidth ]{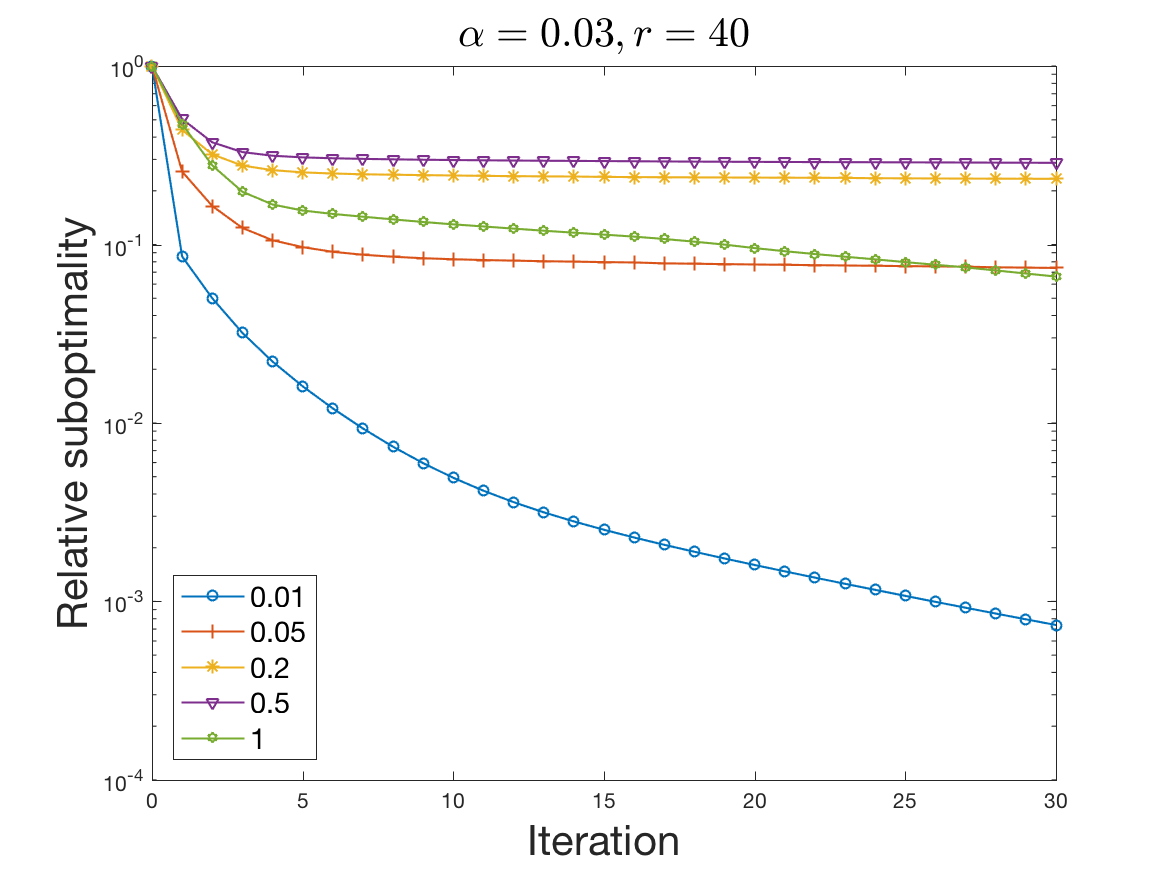}
\end{minipage}%
\begin{minipage}{0.25\textwidth}
  \centering
\includegraphics[width =  \textwidth ]{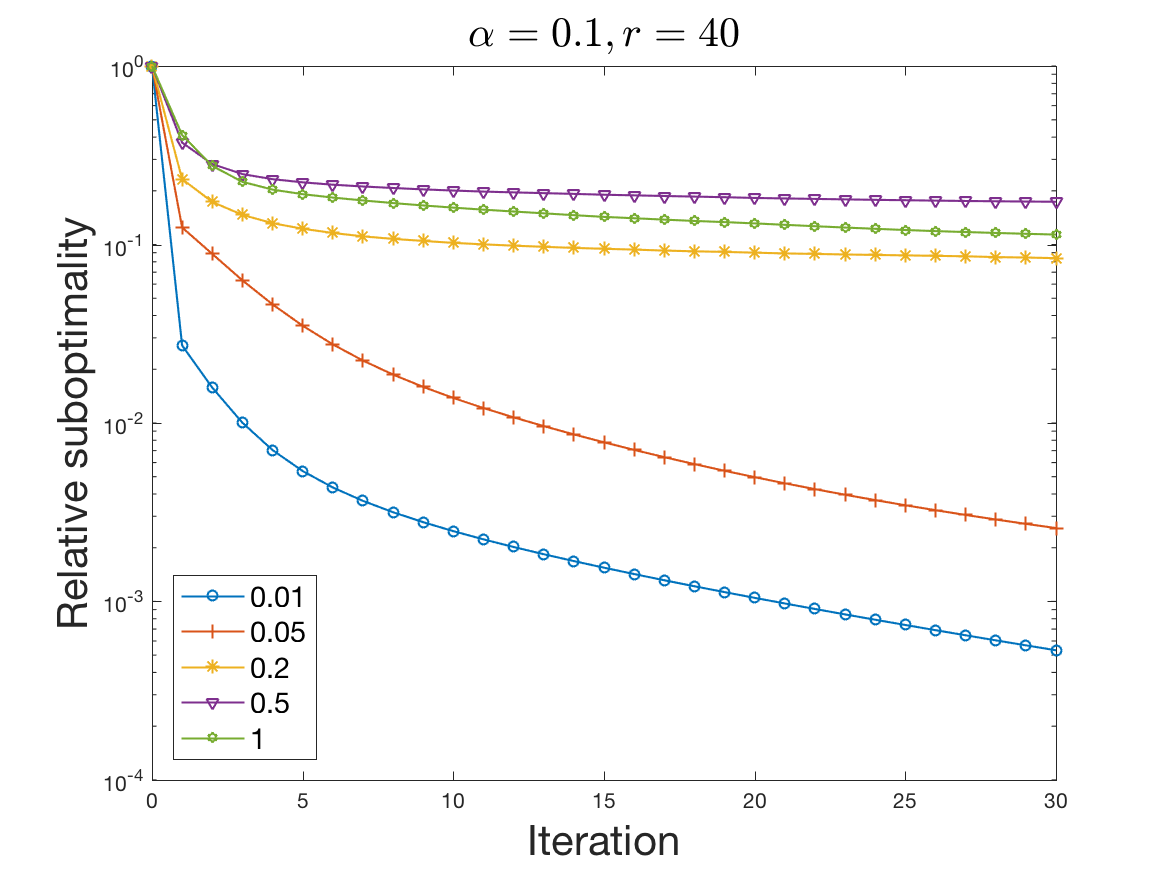}
\end{minipage}%
\begin{minipage}{0.25\textwidth}
  \centering
\includegraphics[width =  \textwidth ]{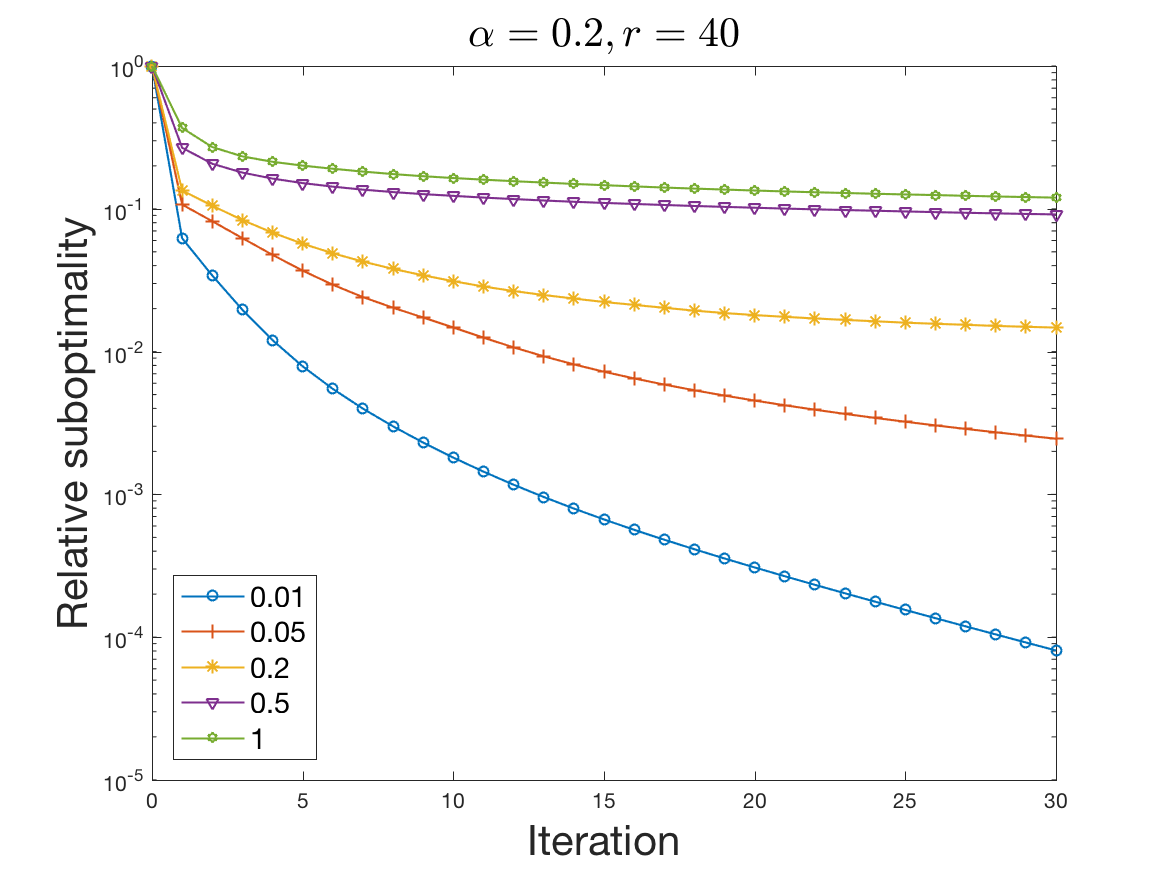}
\end{minipage}%
\begin{minipage}{0.25\textwidth}
  \centering
\includegraphics[width =  \textwidth ]{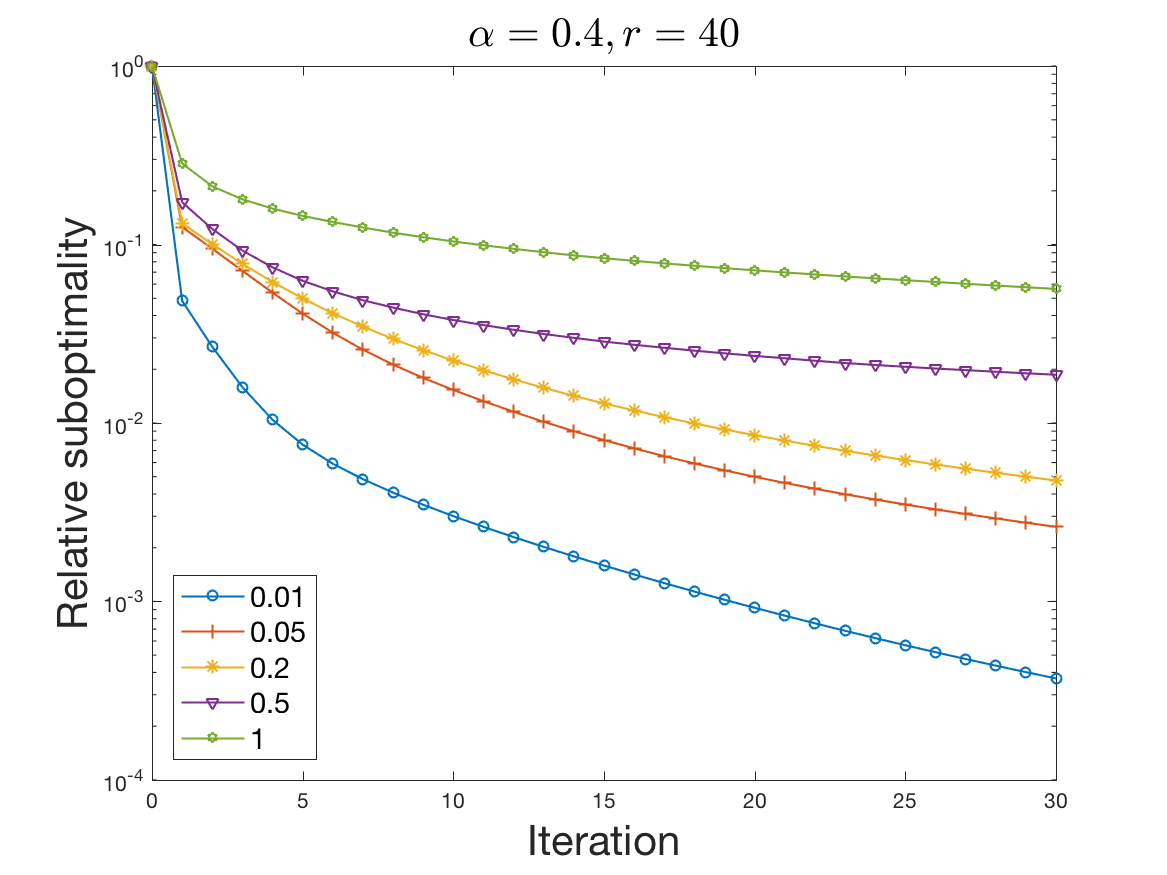}
\end{minipage}%
\caption{Convergence of Algorithm~\ref{rmc_algo} for several  values  of $\Omega$. Each line corresponds to a random percentage of the observable entries and shows normalized $\ell_2$ norm of $(A-L_k-S_k)_{\Omega}$. }\label{fig:omega}
\end{figure}

\subsection{Algorithm~\ref{rpca_algo}: Sensitivity to the choice of $\alpha$ and $r$ }\label{sec:init3}
In this section we study the sensitivity of Algorithm~\ref{rpca_algo} to the degree at which we choose the rank and sparsity level parameters correctly, compare with their true values at the optimum. We first generate a matrix $A$ as described in Section~\ref{sec:init} for a fixed choice of $\hat{\alpha}, \hat{r}$. We then run  Algorithm~\ref{rpca_algo} various choices of $\alpha,r$, including the correct choice. Figure~\ref{fig:sens} shows the results.  
If sparsity and rank levels ($\alpha$ and $r$) are set to be smaller than their true values at the optimum incorrectly, Algorithm~\ref{rpca_algo} does not converge (as in this case, $\cap \cX_i$ might not exist). Moreover, the  performance of the algorithm is sensitive to the choice of $r$, particularly so if we underestimate the true value (see Figure~\ref{fig:omega}). However, overestimating the parameters only leads to a slower convergence. 

\section{Proof of Lemma~\ref{sp} and Lemma~\ref{lem:rmc_linproj}}
We start with proof of Lemma~\ref{sp}.
Note that for $I\in\mathbb{R}^{mn\times mn}$, $L+S=A$ can be rewritten as
\[
\begin{pmatrix}I&I \end{pmatrix}
\begin{pmatrix}
L\\S 
\end{pmatrix}=A. 
\]
Define $x\eqdef \text{vec}\begin{pmatrix}L\\S \end{pmatrix}$ and $ \text{vec}(A)=a$. Therefore, the above is equivalent to
\[
\left( I \otimes \begin{pmatrix}I&I \end{pmatrix}\right)x=a,
\]
which is just a projection on a particular linear system. Recall that the projection of $x_0$ in the Frobenius norm (for the vectors it is equivalent to the $\ell_2$ norm) onto $\cA x=b$ is given as $x_0-\cA^\top (\cA \cA^\top)^\dagger (\cA x-b)$. Therefore,
\begin{eqnarray*}
x
&=&
x_0-\left( I \otimes \begin{pmatrix}I&I \end{pmatrix}\right)^\top
\left( \left( I \otimes \begin{pmatrix}I&I \end{pmatrix}\right)\left( I \otimes \begin{pmatrix}I&I \end{pmatrix}\right)^\top\right)^\dagger
\left(
\left( I \otimes \begin{pmatrix}I&I \end{pmatrix}\right)x-a\right)
\\
&=&
x_0-\frac12\left( I \otimes \begin{pmatrix}I&I \end{pmatrix}\right)^\top
\left( I \otimes I\right)
\left(
\left( I \otimes \begin{pmatrix}I&I \end{pmatrix}\right)x-a\right)
\\
&=&
x_0-\frac12\left( I \otimes \begin{pmatrix}I&I \\ I&I \end{pmatrix}\right)x + \frac12\left( I \otimes \begin{pmatrix}I&I \end{pmatrix}\right)^\top a
\\
&=&
\text{vec}\left( \begin{pmatrix}L_0\\S_0 \end{pmatrix} - \frac12 
\begin{pmatrix}L_0+S_0 \\L_0+S_0 \end{pmatrix}+\frac12 
\begin{pmatrix}A \\A\end{pmatrix} \right) ,
\end{eqnarray*}
which is equivalent to $L^*=\frac12 (L_0-S_0+A)$ and $S^*=\frac12 (S_0-L_0+A)$. Therefore, Lemma~\ref{sp} is established. 

To get Lemma~\ref{lem:rmc_linproj}, it remains to note that the problem is coordinate wise separable. Therefore, the solution behaves as in Lemma~\ref{sp}. on set $\Omega$, otherwise the coordinates of $L,S$ remain unchanged.

\begin{figure}
\centering
\begin{minipage}{0.25\textwidth}
  \centering
\includegraphics[width =  \textwidth ]{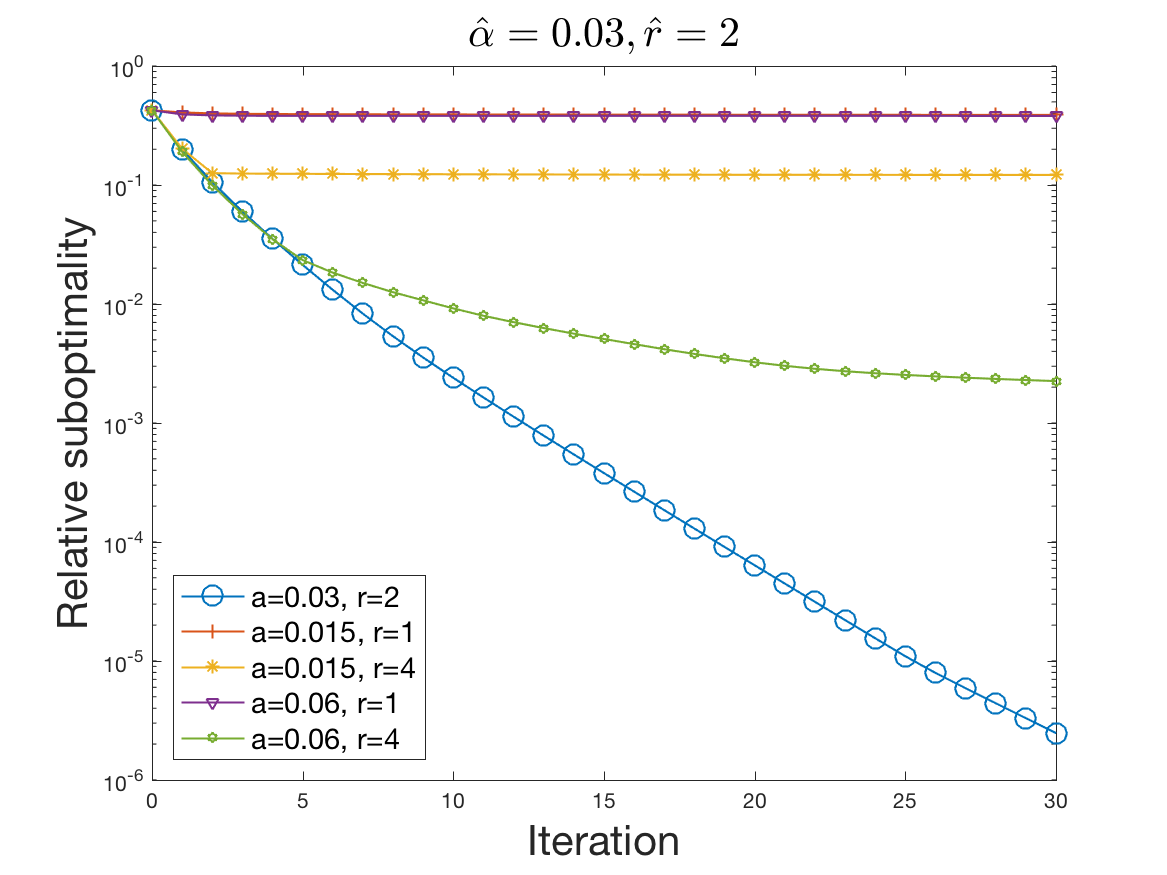}
\end{minipage}%
\begin{minipage}{0.25\textwidth}
  \centering
\includegraphics[width =  \textwidth ]{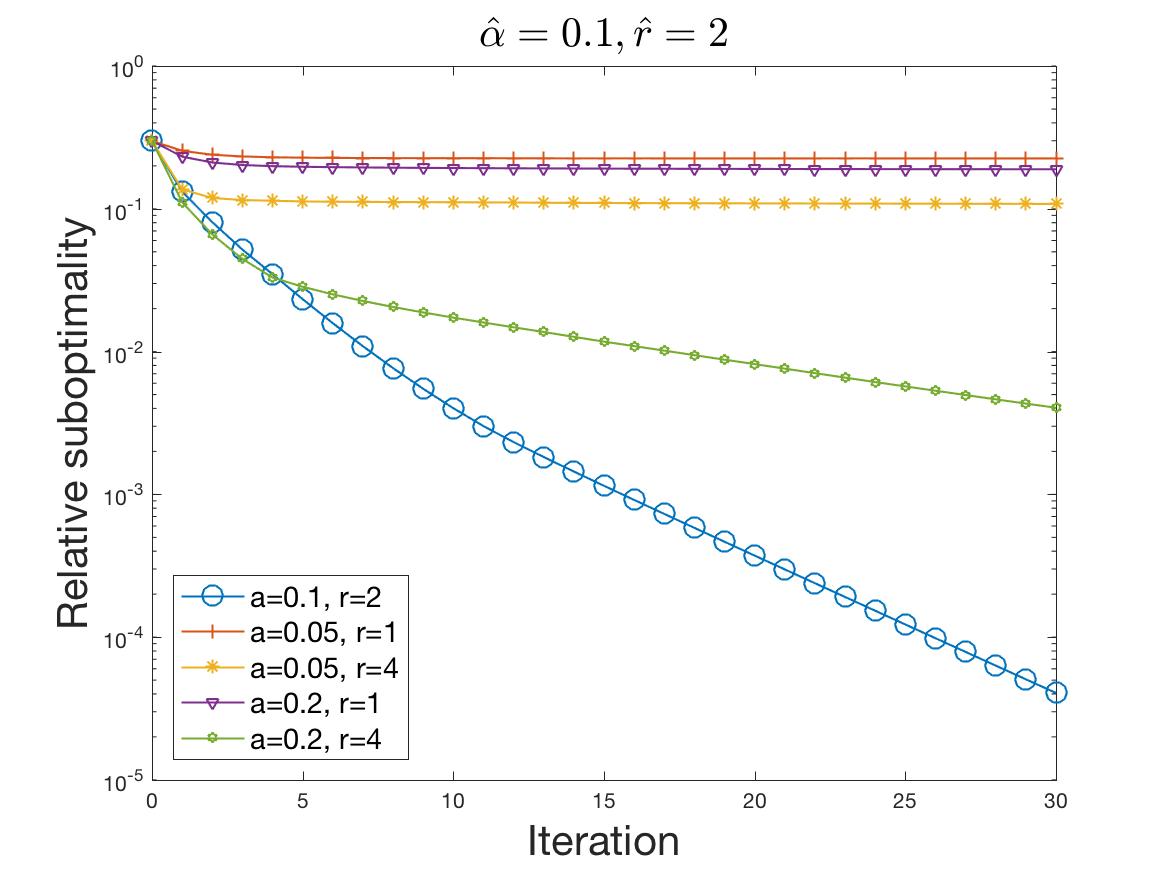}
\end{minipage}%
\begin{minipage}{0.25\textwidth}
  \centering
\includegraphics[width =  \textwidth ]{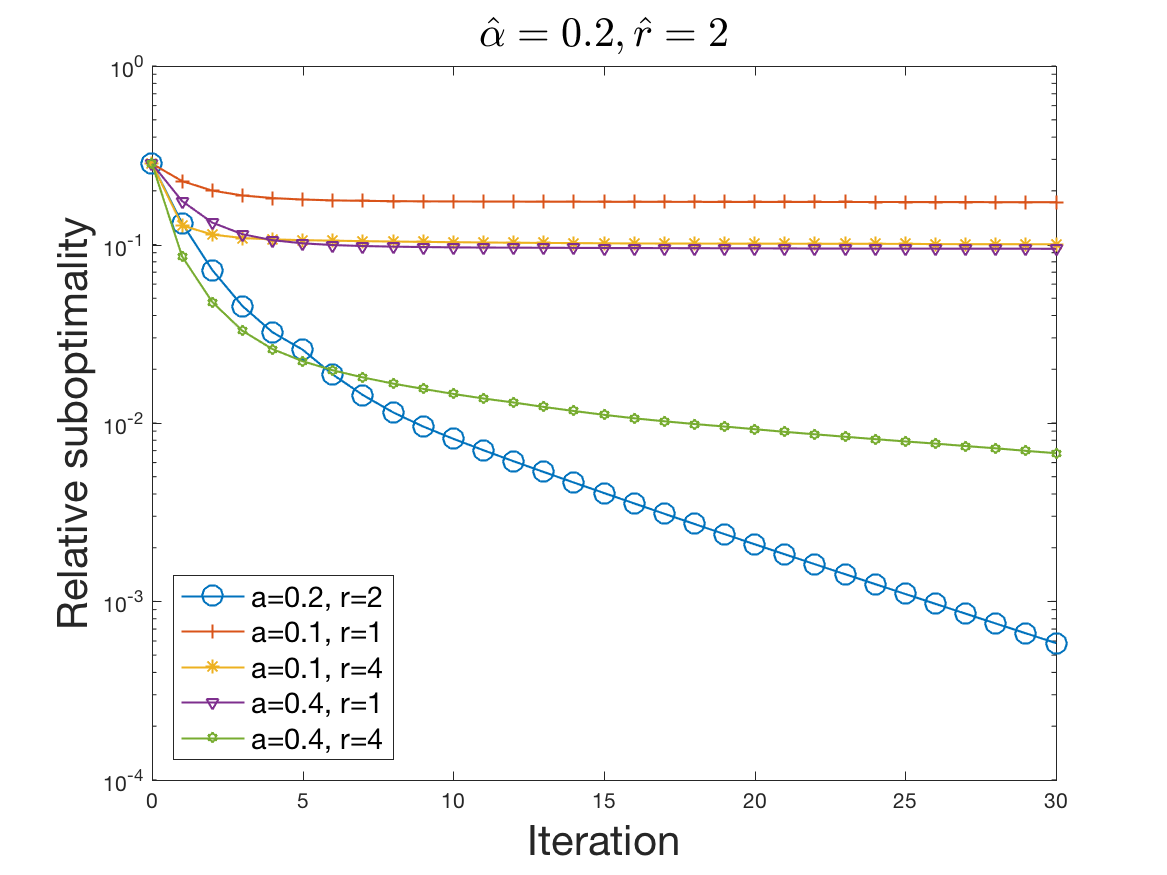}
\end{minipage}%
\begin{minipage}{0.25\textwidth}
  \centering
\includegraphics[width =  \textwidth ]{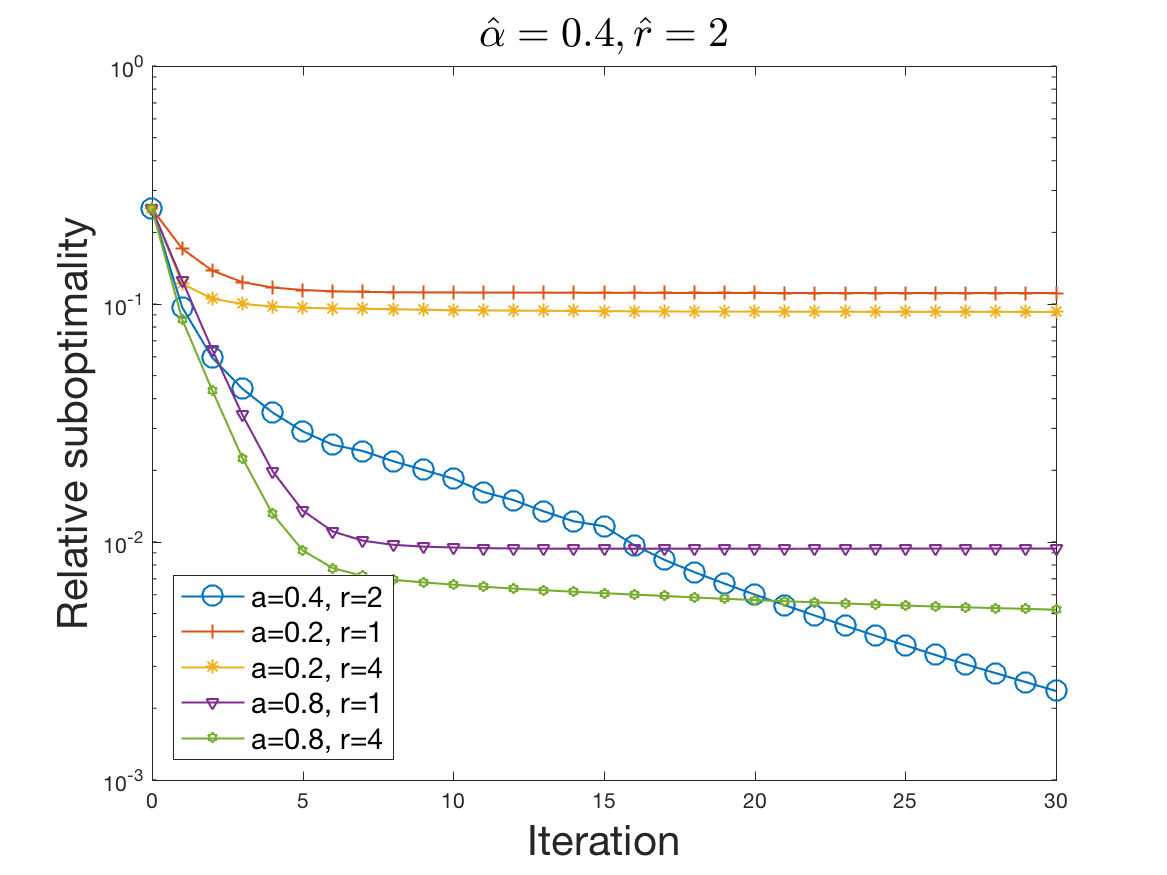}
\end{minipage}%
\\
\begin{minipage}{0.25\textwidth}
  \centering
\includegraphics[width =  \textwidth ]{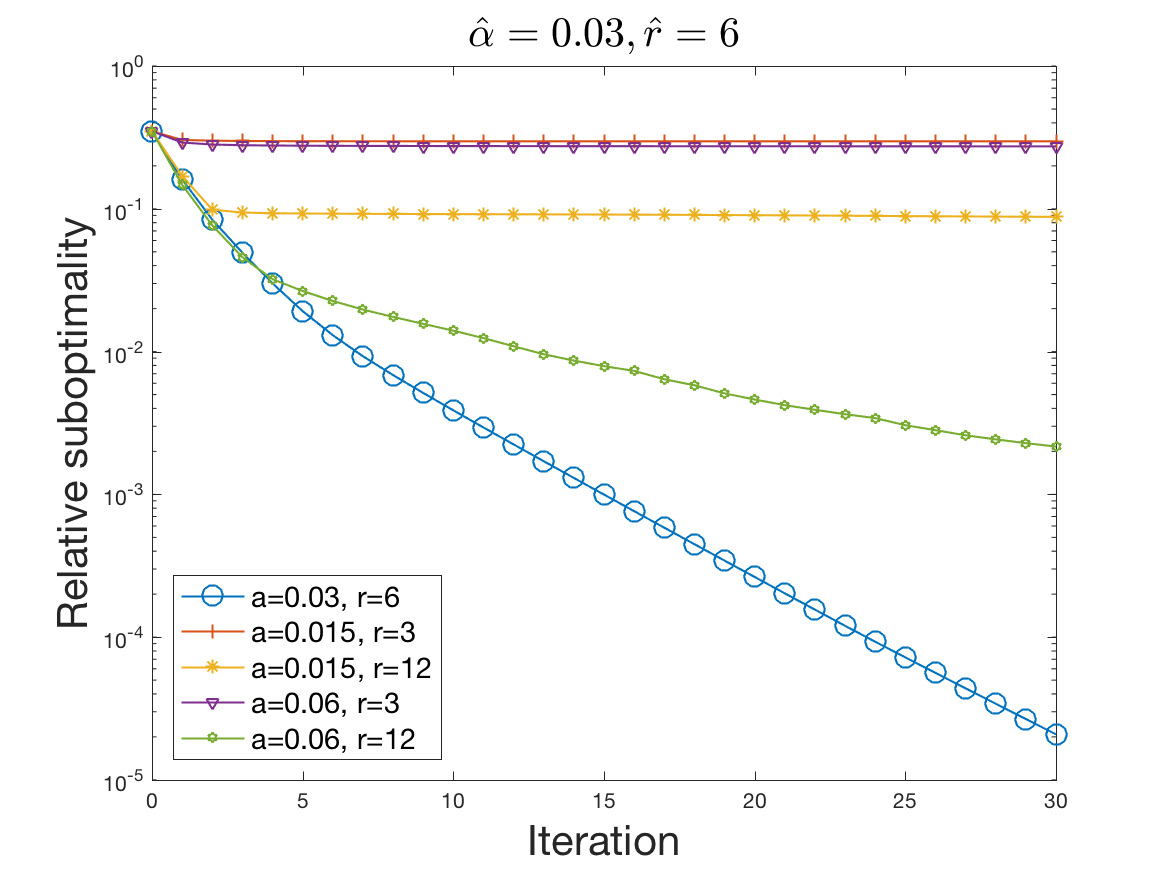}
\end{minipage}%
\begin{minipage}{0.25\textwidth}
  \centering
\includegraphics[width =  \textwidth ]{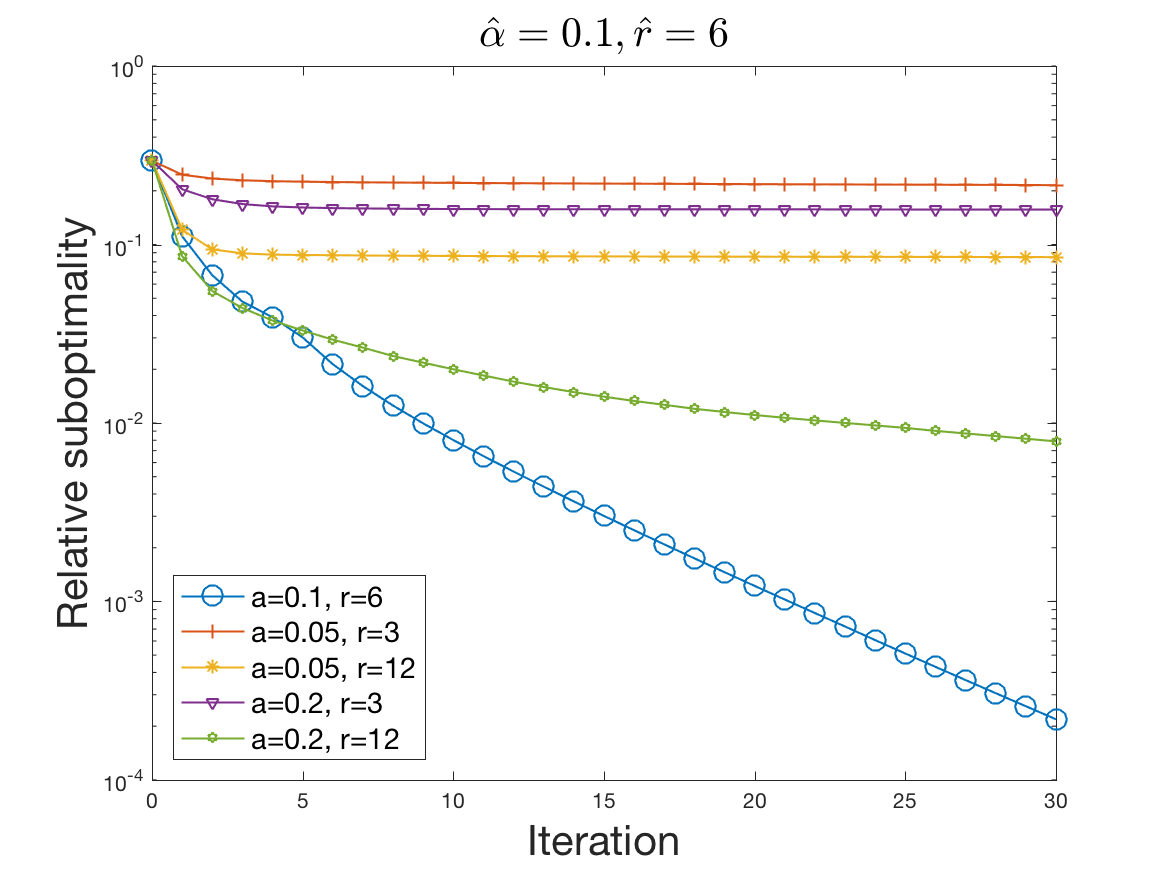}
\end{minipage}%
\begin{minipage}{0.25\textwidth}
  \centering
\includegraphics[width =  \textwidth ]{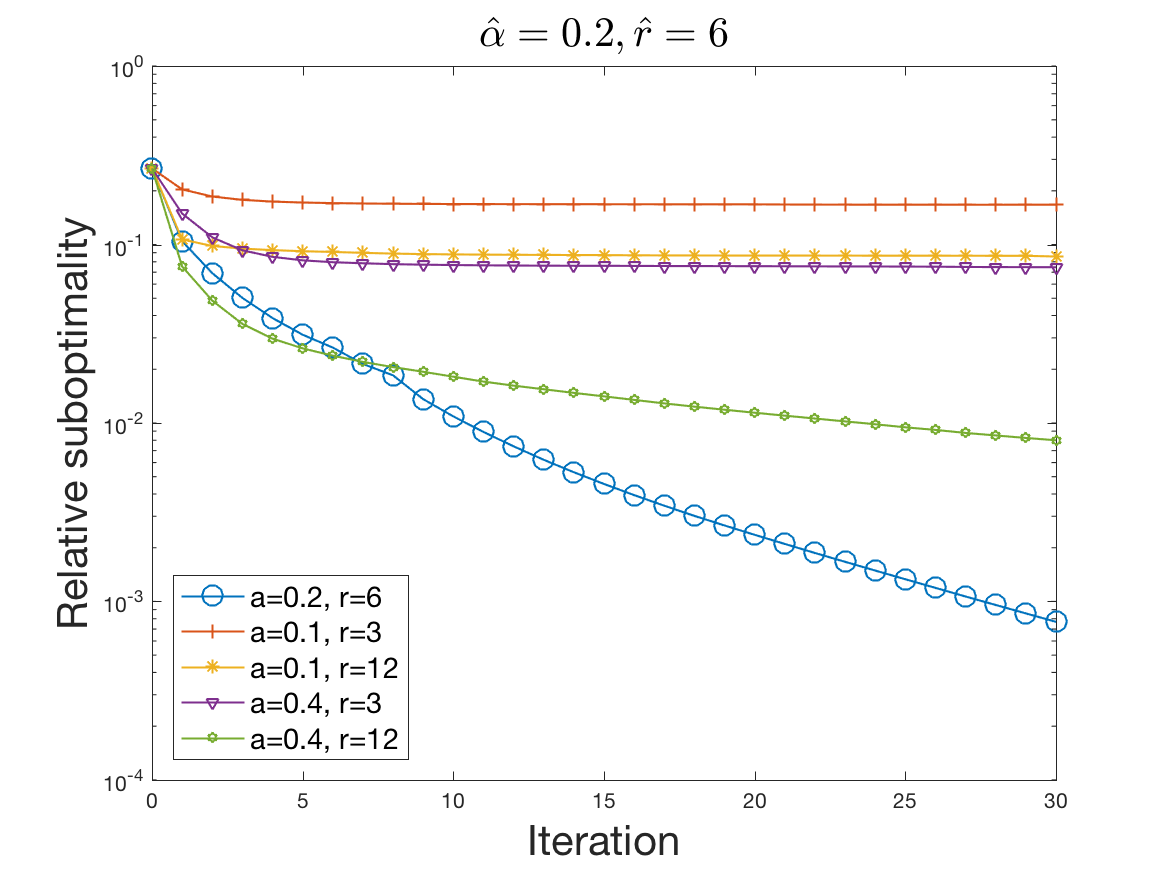}
\end{minipage}%
\begin{minipage}{0.25\textwidth}
  \centering
\includegraphics[width =  \textwidth ]{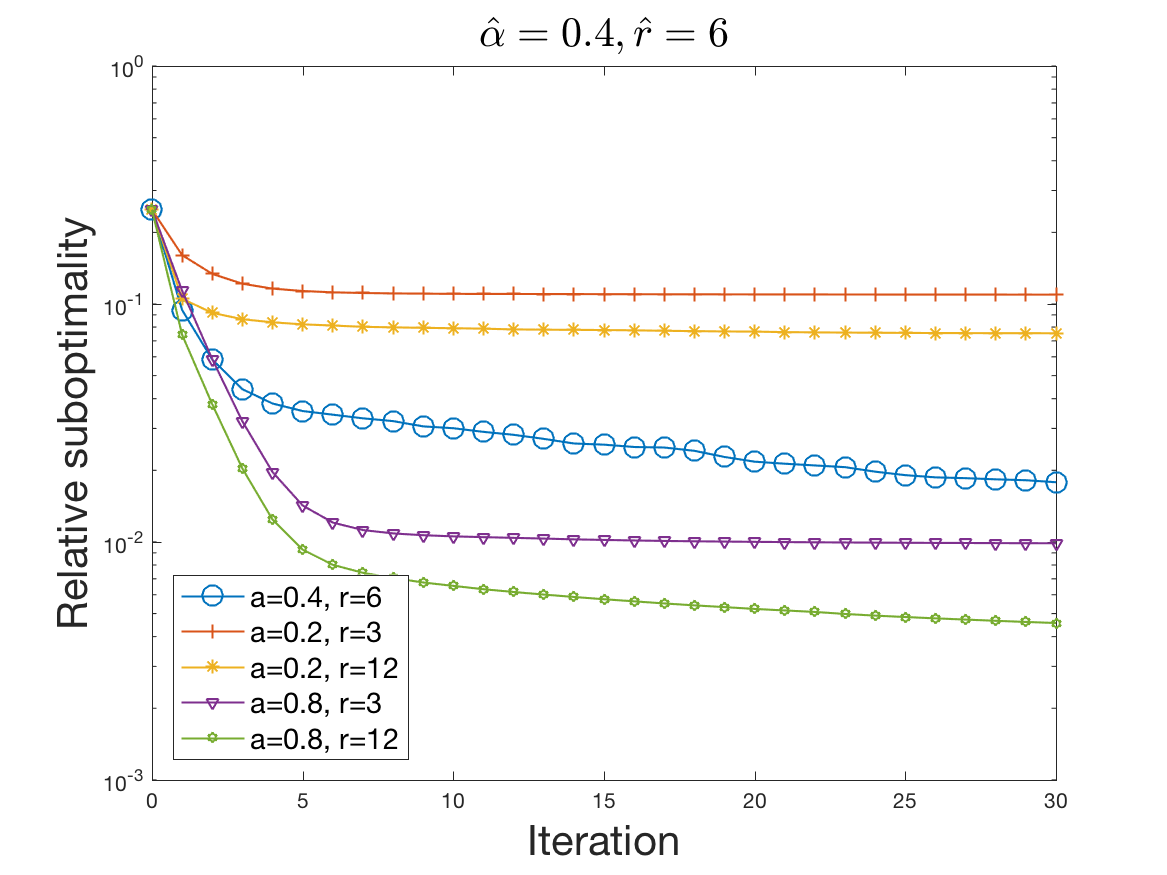}
\end{minipage}%
\\
\begin{minipage}{0.25\textwidth}
  \centering
\includegraphics[width =  \textwidth ]{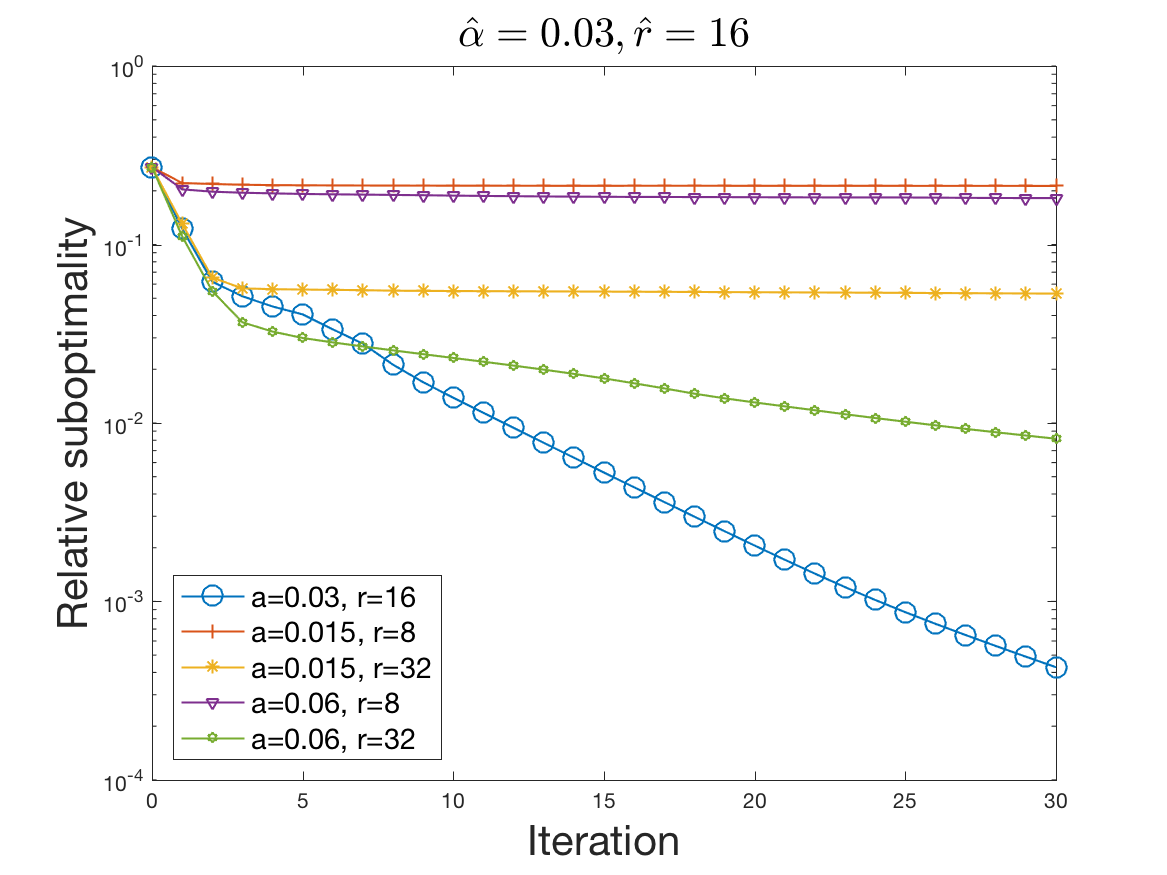}
\end{minipage}%
\begin{minipage}{0.25\textwidth}
  \centering
\includegraphics[width =  \textwidth ]{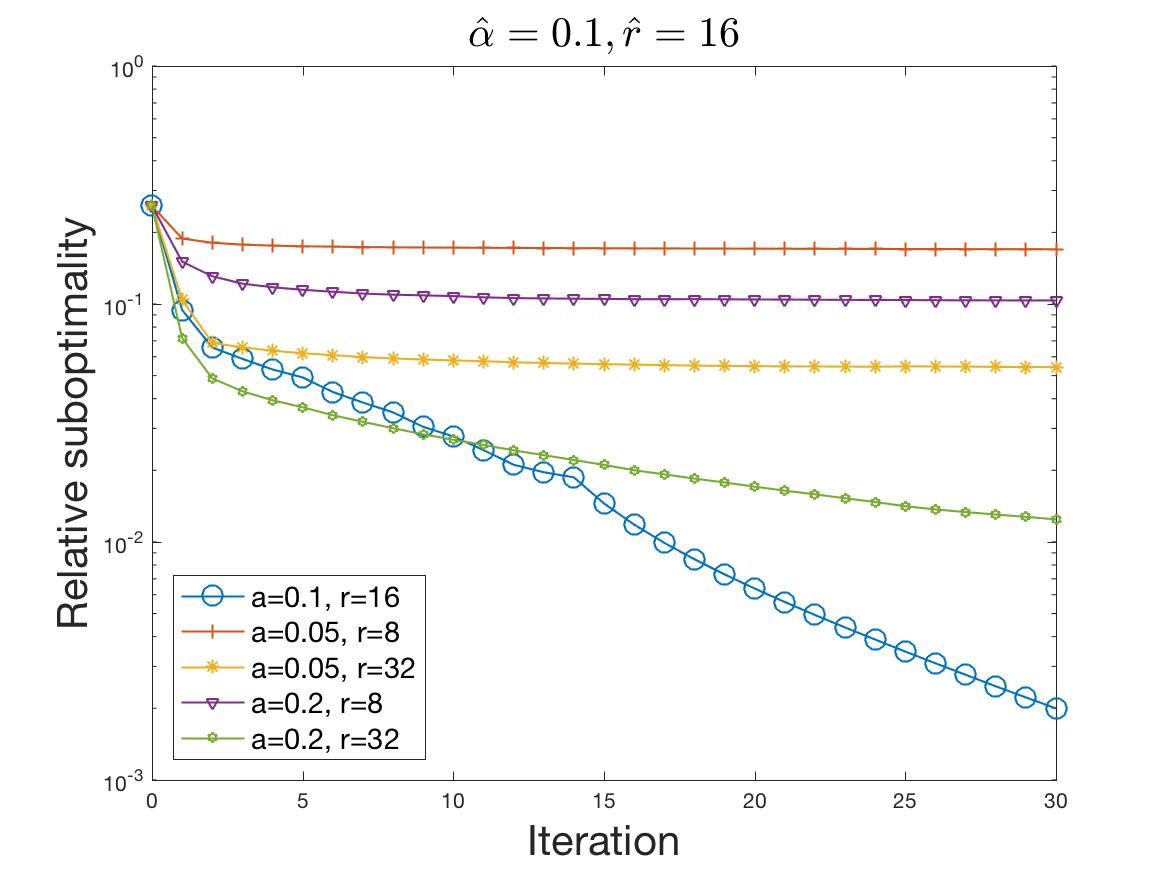}
\end{minipage}%
\begin{minipage}{0.25\textwidth}
  \centering
\includegraphics[width =  \textwidth ]{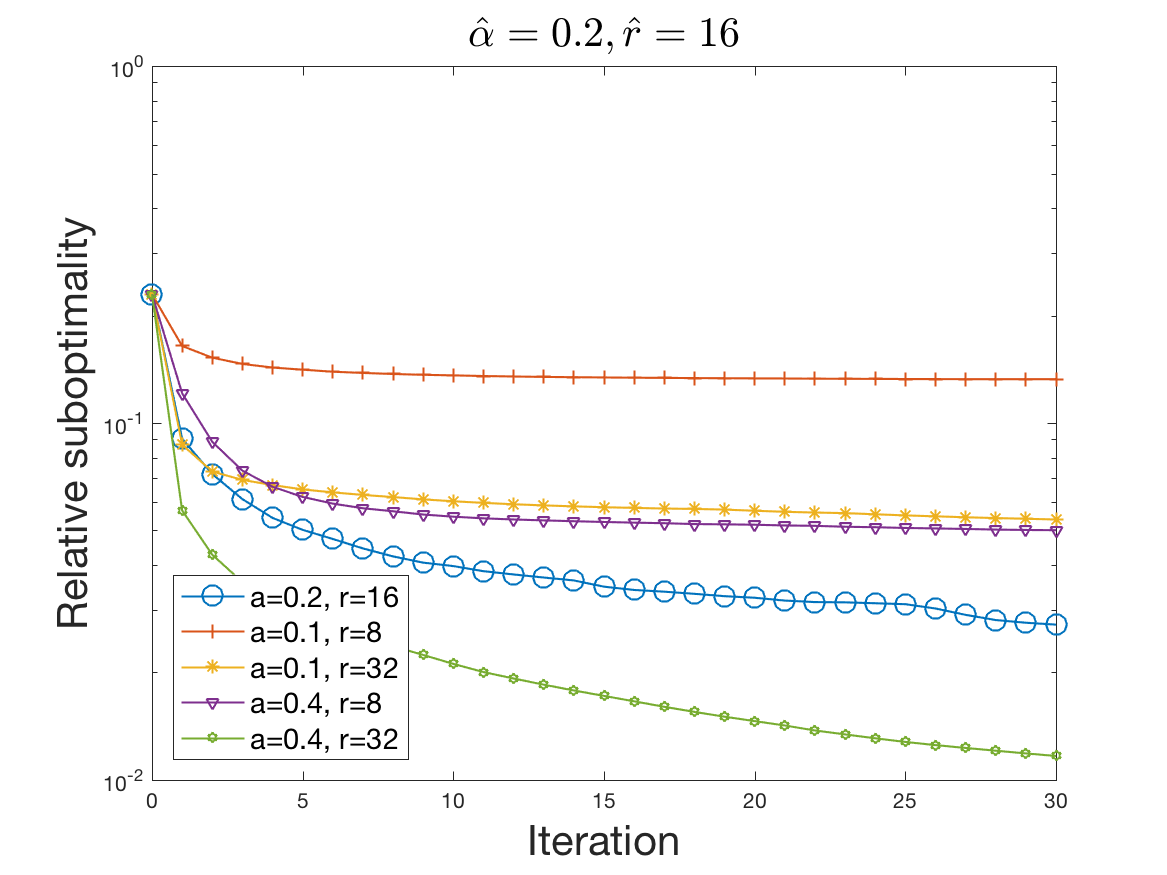}
\end{minipage}%
\begin{minipage}{0.25\textwidth}
  \centering
\includegraphics[width =  \textwidth ]{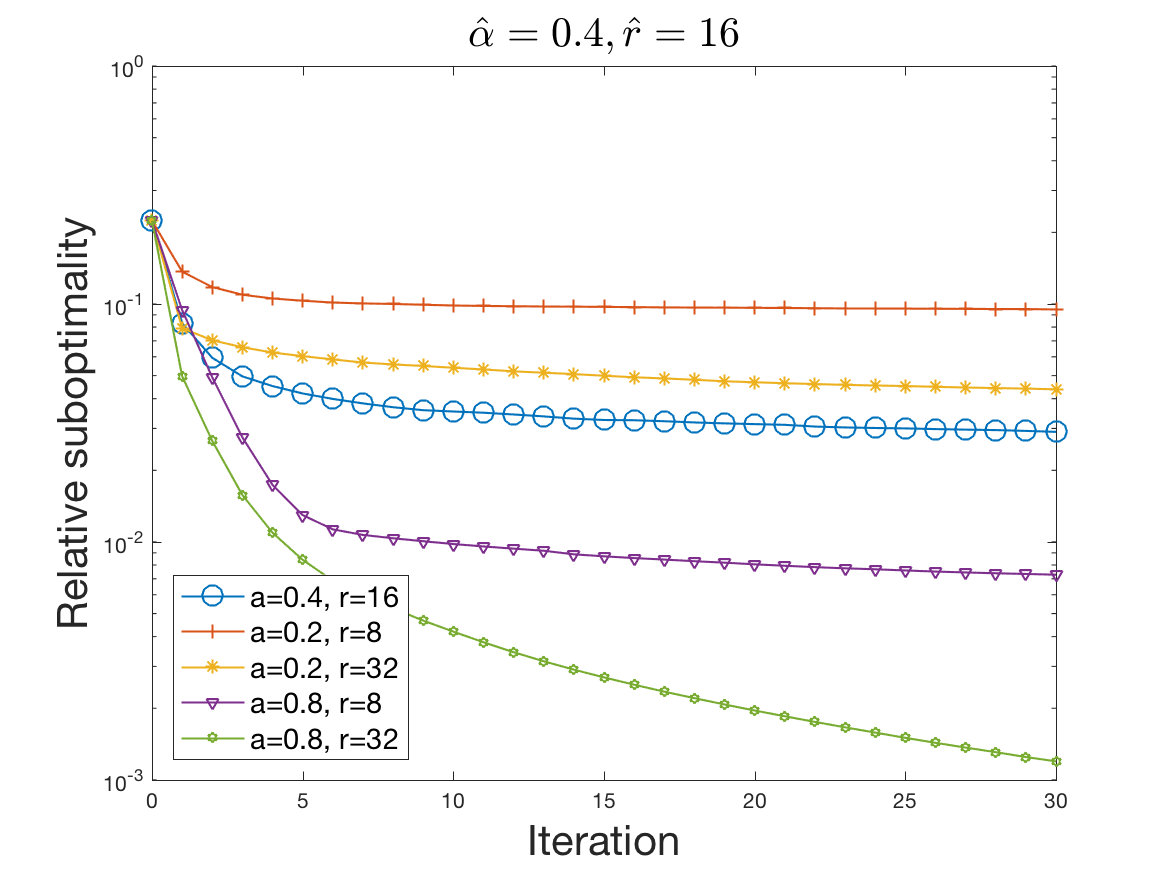}
\end{minipage}%
\\
\begin{minipage}{0.25\textwidth}
  \centering
\includegraphics[width =  \textwidth ]{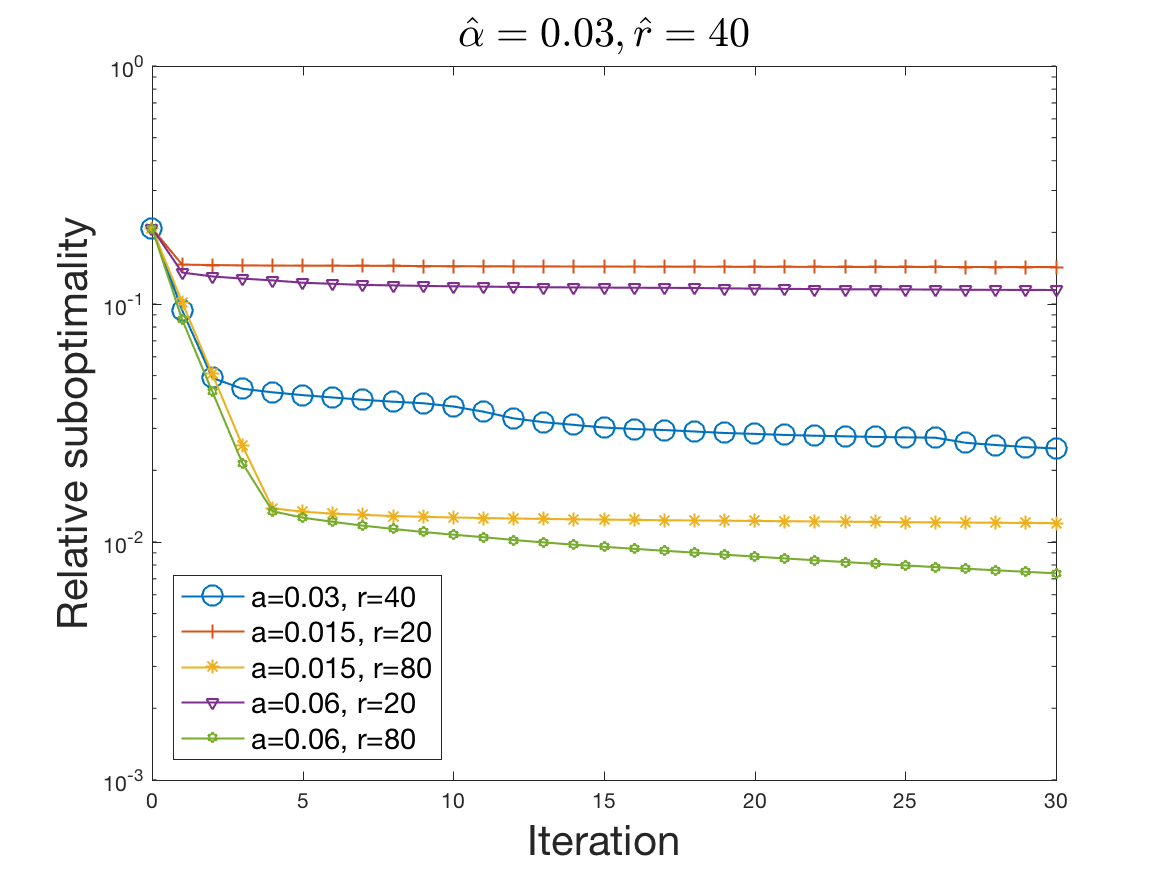}
\end{minipage}%
\begin{minipage}{0.25\textwidth}
  \centering
\includegraphics[width =  \textwidth ]{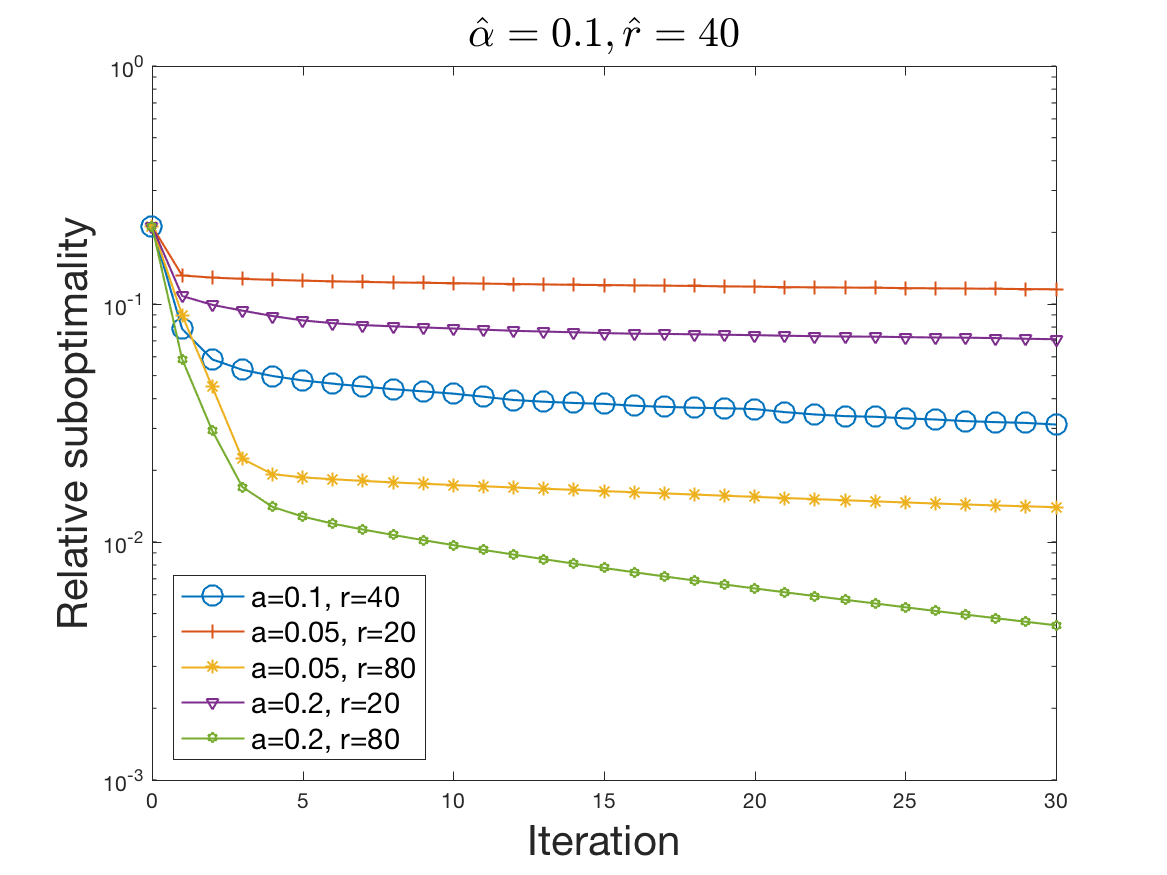}
\end{minipage}%
\begin{minipage}{0.25\textwidth}
  \centering
\includegraphics[width =  \textwidth ]{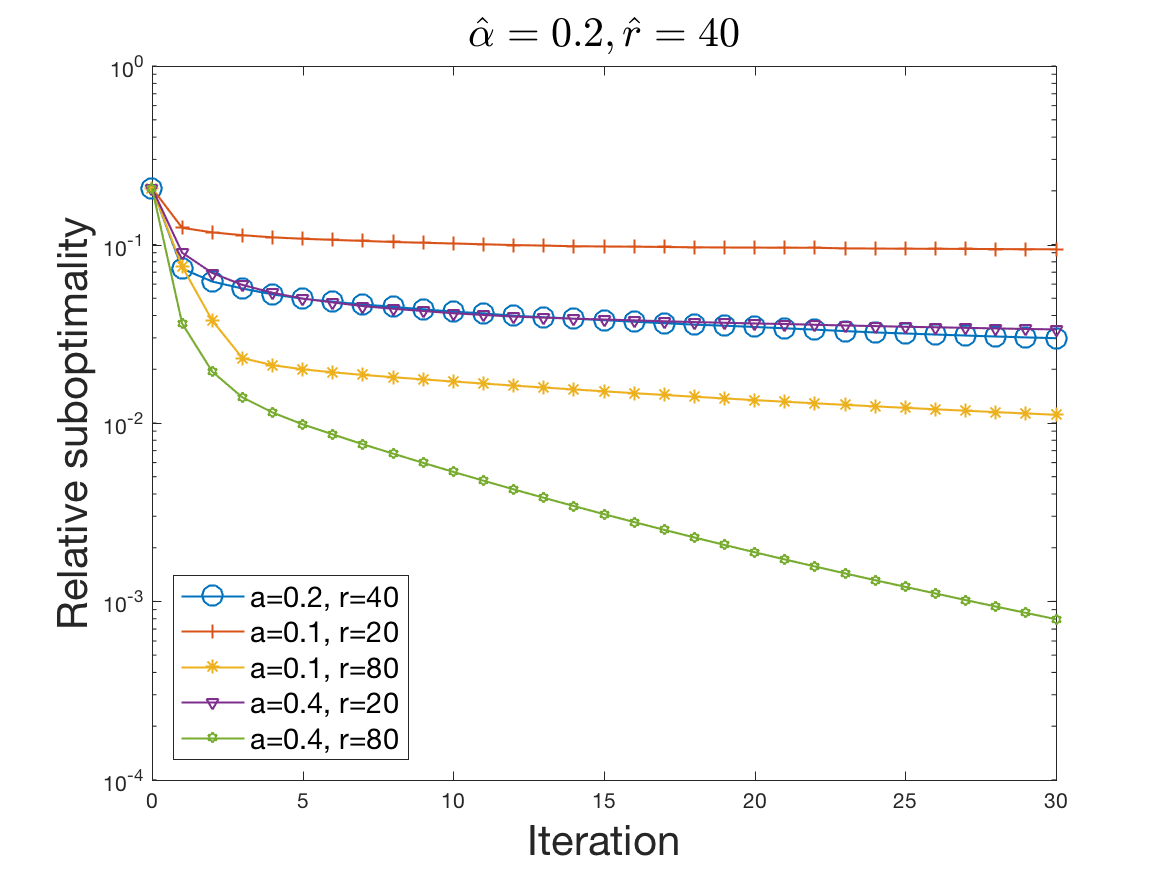}
\end{minipage}%
\begin{minipage}{0.25\textwidth}
  \centering
\includegraphics[width =  \textwidth ]{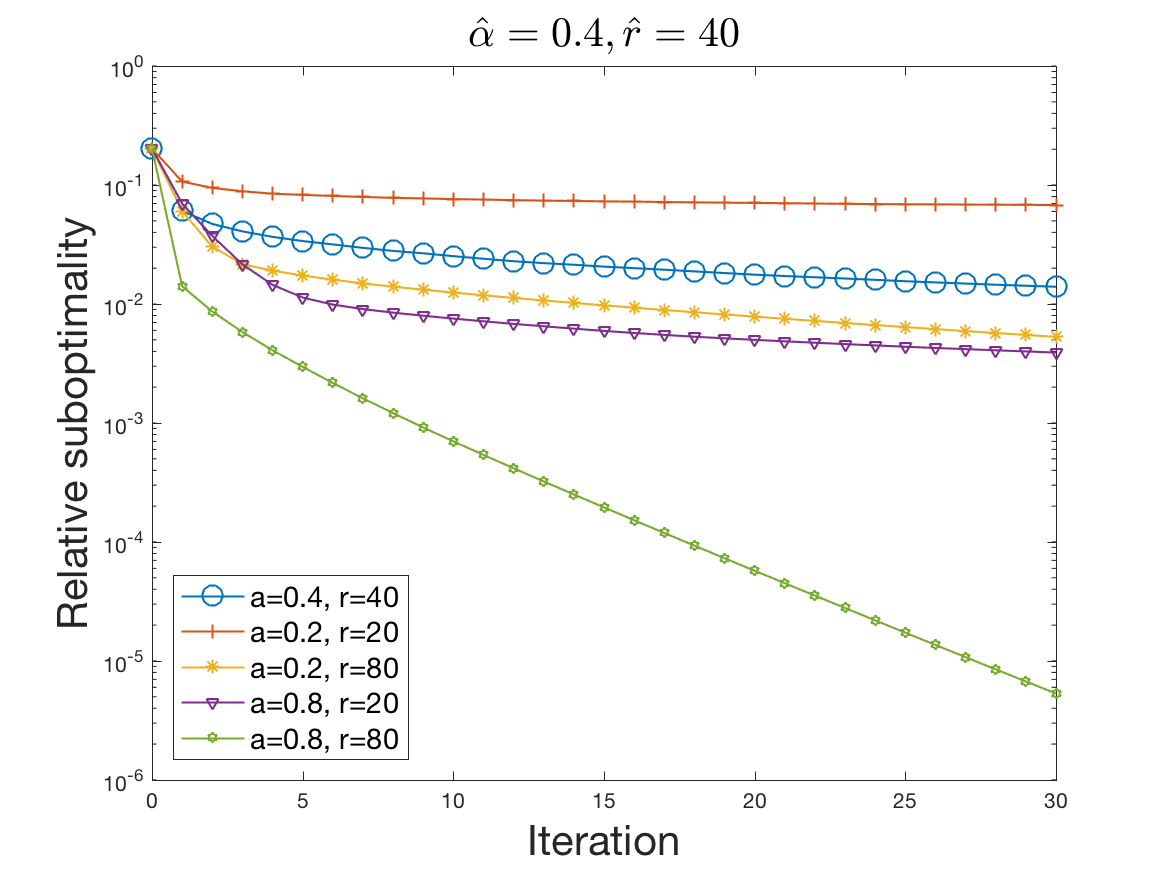}
\end{minipage}%
\caption{Convergence of Algorithm~\ref{rmc_algo} for the different choices of $\Omega$. Each line corresponds to a random percentage of observable entries and shows normalized $\ell_2$ norm of $(A-L_k-S_k)_{\Omega}$. }\label{fig:sens}
\end{figure}

\section{Two Examples of Convergence \label{sec:convergence_app}}

In this section we give two examples of a convex version of the alternating projection method on a problem with similar (block) structure as~\eqref{eq:feas_prob}. The first example shows that the convergence might be extremely fast and independent on $A$, and the second one demonstrates that the rate might not be linear even under convexity. 

\begin{example}\label{ex:linear}
Consider problem~\eqref{eq:feas_prob} with $\cX_1$ defined as~\eqref{eq:X1_defF} and both $\cX_2, \cX_3$ satisfy the same linear constraint.
\end{example}

\begin{lemma}\label{lem:linear}
Alternating projection algorithm applied on Example~\ref{ex:linear} converges in 1 iteration. 
\end{lemma}
\begin{proof}
For simplicity, let us vectorize $L,S$: $x\eqdef\text{vec}(L), y\eqdef\text{vec}(S)$ and denote $I$ to be $mn\times mn$ identity matrix.
Since the constraints are linear, the alternating projection algorithm applied on them converges as fast as alternating projection applied on any affine translation of them such that nonempty intersection property holds.  Let us therefore, without loss of generality consider the following linearly translated problem for some matrix $Q$:
\begin{eqnarray*}
\cX_1 &\eqdef& \left\{ \begin{pmatrix}
x \\y
\end{pmatrix} \, \Big| \,  \begin{pmatrix}
I &I
\end{pmatrix} \begin{pmatrix}
x \\y
\end{pmatrix} =0 \right \}
\\
\cX_2 &\eqdef& \left\{ \begin{pmatrix}
x \\0
\end{pmatrix} \, \Big| \,  Qx =0 \right \}
\\
\cX_3 &\eqdef& \left\{ \begin{pmatrix}
0 \\y
\end{pmatrix} \, \Big| \,  Qy =0 \right \}.
\end{eqnarray*}
Therefore we have for some projection matrix $P=P(Q)$
\begin{eqnarray*}
\pi_{\cX_2\cap \cX_3}\left(\pi_{\cX_1}  \begin{pmatrix}
x \\y
\end{pmatrix} \right)
&=&
 \underbrace{
\begin{pmatrix}
P & 0\\0 &P
\end{pmatrix}
\begin{pmatrix}
\frac12I & -\frac12I \\-\frac12I & \frac12I 
\end{pmatrix}
}_{R}
 \begin{pmatrix}
x \\y
\end{pmatrix}
\end{eqnarray*}
and the convergence of the algorithm is determined by the maximal eigenvalue of $R$ which is not $0$ or $1$ in the absolute value. Clearly, vectors of type
\[
 \begin{pmatrix}p_1 \\ p_1\end{pmatrix},  \begin{pmatrix}p_2 \\-p_2\end{pmatrix}, \begin{pmatrix}p_1' \\ p_1'\end{pmatrix},   \begin{pmatrix}p _2'\\-p_2'\end{pmatrix}, 
\]
might form an orthonormal basis of the space for $p_1,p_2\in \range{P} $ and $p_1',p_2'\perp \range{P} $. However, each of them is an eigenvector of $R$ with eigenvalue 0 or 1, which finishes the proof. 

\end{proof}

We will now present an example where linear convergene rate cannot be attained. 
\begin{example}\label{ex:balls}
Consider problem~\eqref{eq:feas_prob} for $A\in \R^2$ with $\cX_1$ defined as~\eqref{eq:X1_defF} and both $\cX_2, \cX_3$ are unit balls. 
\end{example}

The next lemma shows that there exist a problem of structure~\eqref{eq:feas_prob}, for which alternating projection algorithm does not attain a linear convergence rate. 
\begin{lemma}\label{lem:balls}
Suppose that  $\cX$ is nonempty. 
There exists a starting point such that for Example~\ref{ex:balls}, alternating projection algorithm does not converge linearly. 
\end{lemma}
\begin{proof}
Choose 
\[A=\begin{pmatrix}
2\\0
\end{pmatrix}, \qquad
L_0=
\begin{pmatrix}
\sqrt{2}\\\sqrt{2}
\end{pmatrix},\qquad \mbox{and} \qquad
S_0=
\begin{pmatrix}
\sqrt{2}\\-\sqrt{2}
\end{pmatrix}. 
\]
Clearly, in optimum we must have $L^*=e_1, S^*=e_1$. 
It is a simple exercise to notice that $L,S$ are projected each iteration onto line $x=1$ and then back to the unit circle. Therefore, alternating projection onto $\eqref{eq:feas_prob}$ converges as fast as alternating projection onto unit ball and its tangent line. However, it is easy to see that the latter algorithm does not enjoy a linear convergence. 
\end{proof}

\section{Block Krylov SVD~\cite{musco2015randomized} \label{sec:svd_alg}}
\begin{algorithm}[H]
\SetAlgoLined
	\SetKwInOut{Input}{Input}
	\SetKwInOut{Output}{Output}
	\SetKwInOut{Init}{Initialize}
\caption{Block Krylov SVD \cite{musco2015randomized} (BKSVD)}
\Input{$L \in \mathbb{R}^{m \times n}$, tolerance $\tilde{\epsilon} \in (0,1)$, rank $r \le m,n$} 

\nl \quad $q \eqdef \Theta(\frac{\log d}{\sqrt{\tilde{\epsilon}}})$, $\Pi \sim \mathcal{N}(0,1)^{n \times r}$

\nl \quad$K \eqdef \left [ L\Pi, (LL^\top  )L\Pi,..., (LL^\top  )^q L\Pi\right]$  \label{nl:K_form} 

\nl  \quad Orthonormalize the columns of $K$ to obtain $Q \in \R^{m\times qr}$ \label{nl:K_ort}

\nl \quad Compute $M \eqdef Q^\top LL^\top Q \in \mathbb{R}^{qr \times qr}$

\nl \quad Set ${\bar U}_r$ to the top $r$ singular vectors of $M$

\nl \Output{$Z = Q{\bar U}_r$}
\end{algorithm}

\end{document}